\documentclass[a4paper,10pt]{amsart}
\usepackage{amssymb,amsmath,amsfonts,amsthm}
\usepackage{mathrsfs}
\usepackage[left=1 in, right=1 in,top=1 in, bottom=1 in]{geometry}

\providecommand{\abs}[1]{\left\vert#1\right\vert}
\providecommand{\nm}[1]{\left\Vert#1\right\Vert}
\providecommand{\lnm}[2]{\left\Vert#1\right\Vert_{L^{#2}}}
\providecommand{\lnms}[3]{\left\Vert#1\right\Vert_{L^{#2}(#3)}}
\providecommand{\tnm}[2]{\left\Vert#1\right\Vert_{L^{2}H^{#2}}}
\providecommand{\tnms}[3]{\left\Vert#1\right\Vert_{L^{2}H^{#2}(#3)}}
\providecommand{\inm}[2]{\left\Vert#1\right\Vert_{L^{\infty}H^{#2}}}
\providecommand{\inms}[3]{\left\Vert#1\right\Vert_{L^{\infty}H^{#2}(#3)}}
\providecommand{\hm}[2]{\left\Vert#1\right\Vert_{H^{#2}}}
\providecommand{\hms}[3]{\left\Vert#1\right\Vert_{H^{#2}(#3)}}

\providecommand{\ts}[1]{\in L^2([0,T]; H^{#1})}
\providecommand{\ths}[1]{\in L^2([0,T]; H^{#1}(\Sigma))}
\providecommand{\is}[1]{\in L^{\infty}([0,T]; H^{#1})}
\providecommand{\ihs}[1]{\in L^{\infty}([0,T]; H^{#1}(\Sigma))}
\providecommand{\cs}[1]{C^0([0,T];H^{#1})}
\providecommand{\chs}[1]{C^0([0,T];H^{#1}(\Sigma))}
\providecommand{\s}[1]{\in H^{#1}} \providecommand{\hs}[1]{\in
H^{#1}(\Sigma)}

\providecommand{\br}[1]{\langle #1 \rangle}
\providecommand{\brs}[1]{\langle #1 \rangle_{\h^0(\Sigma)}}
\providecommand{\brh}[1]{\langle #1 \rangle_{\h^0}}
\providecommand{\brht}[1]{\langle #1 \rangle_{\l^2\h^0}}
\providecommand{\brhs}[1]{\langle #1 \rangle_{\l^2\h^0(\Sigma)}}
\providecommand{\ud}[1]{\mathrm{d{#1}}}

\def\dt{\partial_t}
\def\half{\frac{1}{2}}
\def\ls{\lesssim}
\def\gs{\gtrsim}
\def\p{\partial}
\def\dm{\mathbb{D}}
\def\la{\Delta_{\mathcal{A}}}
\def\na{\nabla_{\mathcal{A}}}
\def\da{\nabla_{\mathcal{A}}\cdot}
\def\sa{S_{\mathcal{A}}}
\def\ma{\mathbb{D}_{\mathcal{A}}}
\def\e{\eta}
\def\xx{\xi}
\def\ee{\eta_0}
\def\zz{\zeta}
\def\eb{\bar\eta}

\def\wst{W^{\ast}}
\def\wwst{\mathcal{W}^{\ast}}

\def\xxst{\mathcal{X}^{\ast}}

\def\a{\mathcal{A}}

\def\d{\mathcal{D}}
\def\f{\mathcal{F}}
\def\g{\mathcal{G}}
\def\h{\mathcal{H}}
\def\k{\mathcal{K}}
\def\l{\mathcal{L}}
\def\m{\mathcal{M}}
\def\n{\mathcal{N}}
\def\q{\mathcal{Q}}

\def\w{\mathcal{W}}
\def\x{\mathcal{X}}
\def\y{\mathcal{Y}}
\def\z{\mathcal{Z}}

\def\ce{\mathcal{E}}
\def\pp{\mathcal{P}}
\def\ss{\mathcal{S}}

\def\i{I_{\lambda}}

\def\mm{\mathfrak{M}}
\def\nn{\mathfrak{N}}

\newtheorem{theorem}{Theorem}[section]
\newtheorem{lemma}[theorem]{Lemma}
\newtheorem{proposition}[theorem]{Proposition}

\newtheorem{remark}[theorem]{Remark}
\newtheorem{definition}[theorem]{Definition}

\makeatletter
\renewcommand \theequation {%
\ifnum \c@section>\z@ \@arabic\c@section.%
\fi \ifnum\c@subsection>\z@\@arabic\c@subsection.%
\fi\@arabic\c@equation} \@addtoreset{equation}{section}
\@addtoreset{equation}{subsection} \makeatother

\begin{document}

\title{Wellposedness and Decaying Property of \\Viscous Surface Wave}
\author{Lei Wu}
\address{
Division of Applied Mathematics\\
Brown University \\
182 George St., Providence, RI 02912, USA } \email[L.
Wu]{zjkwulei1987@brown.edu} \subjclass[2000]{35Q30, 35R35, 76D03,
76E17}\maketitle

\begin{abstract}
In this paper, we consider an incompressible viscous flow without
surface tension in a finite-depth domain of three dimensions, with
free top boundary and fixed bottom boundary. This system is governed
by a Naiver-Stokes equation in above moving domain and a transport
equation for the top boundary. Traditionally, we consider this
problem in Lagrangian coordinates with perturbed linear form. In the
series of papers \cite{book1}, \cite{book9} and \cite{book10}, I.
Tice and Y. Guo introduced a new framework using geometric structure
in Eulerian coordinates to study both local and global wellposedness
of this system. Following this path, we extend their result in local
wellposedness from small data case to arbitrary data case. Also, we
give a simpler proof for global wellposedness in infinite domain.
Other than the geometric energy estimates, time-dependent Galerkin
method, and interpolation estimate with Riesz potential and minimum
count, which are introduced in these papers, we utilize three new
techniques: (1) using $\epsilon$-Poisson integral to construct a
diffeomorphism between fixed domain and moving domain; (2) using
bootstrapping argument to prove a comparison result for steady
Navier-Stokes equation for arbitrary data of free surface; (3)
redefining the energy and dissipation to replace the original
complicated bootstrapping argument to show interpolation estimate.\\
\textbf{Keywords:} free surface, geometric structure, interpolation
\end{abstract}

\maketitle

\section{Introduction}

\subsection{Problem Presentation}

We consider a viscous incompressible flow in the moving domain.
\begin{equation}
\Omega(t)=\{y\in\Sigma\times R\ | -b(y_1,y_2)<y_3<\eta(y_1,y_2,t)\}
\end{equation}
Here we can take either $\Sigma=R^2$ or
$\Sigma=(L_1\mathbb{T})\times(L_2\mathbb{T})$ for which $\mathbb{T}$
denotes the $1$-torus and $L^1,L^2>0$ the periodicity lengths. The
lower boundary $b$ is fixed and given satisfying
\begin{equation}
0<b(y_1,y_2)<\bar b=\rm{constant}\quad \rm{and}\quad b\in
C^{\infty}(\Sigma)
\end{equation}
When $\Sigma=R^2$, we further require that $b(y_1,y_2)\rightarrow
b_0=\rm{positive\ constant}$ as
$\abs{y_1}+\abs{y_2}\rightarrow\infty$ and $b-b_0\in H^s(\Sigma)$
for any $s\geq0$. When
$\Sigma=(L_1\mathbb{T})\times(L_2\mathbb{T})$, we just denote
$b_0=1/2(\max\{b(y_1,y_2)\}+\min\{b(y_1,y_2)\})$. It is easy to see
this also implies $b-b_0\in H^s(\Sigma)$ for any $s\geq0$ since
$\Sigma$ is bounded. We denote the initial domain
$\Omega(0)=\Omega_0$. For each $t$, the flow is described by
velocity and pressure $(u,p):\Omega(t)\mapsto R^3\times R$ which
satisfies the incompressible Navier-Stokes equation
\begin{equation}\label{origin}
\left\{
\begin{array}{ll}
\partial_tu+u\cdot\nabla u+\nabla p=\mu\Delta u &\quad \rm{in} \quad\Omega(t)\\
\nabla\cdot u=0 &\quad \rm{in} \quad \Omega(t)\\
(pI-\mu\mathbb{D}(u))\nu=g\eta\nu&\quad \rm{on} \quad \{y_3=\eta(y_1,y_2,t)\}\\
u=0&\quad \rm{on} \quad \{y_3=-b(y_1,y_2)\}\\
\partial_t\eta=u_3-u_1\partial_{y_1}\eta-u_2\partial_{y_2}\eta &\quad \rm{on} \quad \{y_3=\eta(y_1,y_2,t)\}\\
&\\
u(t=0)=u_0&\quad  \rm{in}\quad\Omega_0\\
\e(t=0)=\ee&\quad  \rm{on}\quad\Sigma
\end{array}
\right.
\end{equation}
for $\nu$ the outward-pointing unit normal vector on $\{y_3=\eta\}$,
$I$ the $3\times 3$ identity matrix,
$(\mathbb{D}u)_{ij}=\partial_iu_j+\partial_ju_i$ the symmetric
gradient of $u$, $g$ the gravitational constant and $\mu>0$ the
viscosity. As described in \cite{book1}, the fifth equation in
(\ref{origin}) implies that the free surface is convected with the
fluid. Note that in (\ref{origin}), we have make the shift of actual
pressure $\bar p$ by constant atmosphere pressure $p_{atm}$
according to $p=\bar p+gy_3-p_{atm}$.\\
We will always assume the natural condition that there exists a
positive number $\rho$ such that $\eta_0+b\geq \rho>0$ on $\Sigma$,
which means that the initial free surface is always strictly
separated from the bottom. Also without loss of generality, we may
assume that $\mu=g=1$, which in fact will not infect our proof. In
the following, we will use the term ``infinite case" when
$\Sigma=R^2$ and ``periodic case" when
$\Sigma=(L_1\mathbb{T})\times(L_2\mathbb{T})$.

\subsection{Previous Results}

Historically, this problem is studied in several different settings
according to whether we consider the viscosity and surface tension,
and the different choices of domains.

For the inviscid case, traditionally, we replace the no-slip
condition $u=0$ with no penetration condition $u\cdot n=0$ on
$\Sigma_b$. It is often assumed that the initial fluid is
irrotational and then this curl-free condition will be preserved in
the later time. This allows to reformulate the problem to one only
on the free surface. The local wellposedness in this framework was
proved by Wu \cite{book12,book13} and Lannes \cite{book14}. Local
wellposedness without irrotational assumption was proved by
Zhang-Zhang \cite{book15}, Christodoulou-Lindblad \cite{book16},
Lindblad \cite{book17}, Coutand-Shkoller \cite{book18} and
Shatah-Zheng \cite{book19}. In the viscous case, vorticity will be
naturally introduced at the free surface, so the surface formulation
will not work. Also, only in the irrotational case the global
wellposedness was shown as in Wu \cite{book22} and
Germain-Masmoudi-Shatah \cite{book23}.

In the viscous case without surface tension, the local wellposedness
of equation (\ref{origin}) was proved by Solonnikov and Beale.
Solonnikov \cite{book11} employed the framework of Holder spaces to
study the problem in a bounded domain, all of whose boundary is
free. Beale \cite{book2} utilized the framework of $L^2$ spaces to
study the domain like ours. Both of them took the equation as a
perturbation of the parabolic equation and make use of the
regularity results for linear equations. Abel \cite{book20} extended
the result to the framework of $L^p$ spaces. Also, Hataya
\cite{book24} proved the global wellposedness in periodic case.

Many authors have also considered the effect of surface tension,
which can help stabilize the problem and gain regularity. However,
most of these results are related to global wellposedness for small
data, e.g. in Bae \cite{book21}.

Almost all of above results are built in the Lagrangian coordinate
and ignore the natural energy structure of the problem. In
\cite{book1}, \cite{book9} and \cite{book10}, Y. Guo and I. Tice
introduced the geometric energy framework and prove both the local
and global wellposedness in small data. The main idea in their proof
is that instead of considering perturbation of the parabolic
equation, we restart to prove the regularity under time-dependent
basis. In our paper, we will follow this path and employ a similar
argument.

In \cite{book1}, it is always assumed that the initial data $u_0$
and $\ee$ is sufficiently small to guarantee that: (1) the transform
between fixed domain and moving domain is a diffeomorphism; (2) the
elliptic estimate in linear Navier-Stokes problem achieves the best
regularity; (3) in proving boundedness and contraction of iteration
sequence, the constant is small enough to get required estimate. In
Section 2 of our paper, we will drop this assumption and prove the
local wellposedness for arbitrary initial data. Hence, we have to
use different techniques to recover the above three results.

Also, in proving global wellposedness and decaying property in
\cite{book9}, the author introduced a quite complicated
interpolation argument to increase the decaying rate of energy with
minimum count. In Section 3 of our paper, through redefining the
energy and dissipation, we greatly simplify the proof in that: (1)
avoiding the complicated bootstrapping argument for interpolation
results; (2)avoiding the bootstrapping argument in comparison
theorem; (3) handling the energy and dissipation 1-minimum count and
2-minimum simultaneously instead of separately.

\subsection{Geometric Formulation}

In order to work in a fixed domain, we want to flatten the free
surface via a coordinate transformation. Beale introduced a flatten
transform in \cite{book2} and we will use a slightly modified
version as the geometric transform. We define a fixed domain
\begin{equation}
\Omega=\{x\in\Sigma\times R\ |-b_0<x_3<0\}
\end{equation}
for which we will write the coordinate $x\in\Omega$. In this slab,
we take $\Sigma:\{x_3=0\}$ as the upper boundary and $\Sigma_{b}:
\{x_3=-b_0\}$ as the lower boundary. Simply denote $x'=(x_1,x_2)$
and $x=(x_1,x_2,x_3)$. In order to introduce the mapping between
$\Omega$ and $\Omega(t)$, we need to utilize slightly different
techniques to construct the extension of free surface.

\subsubsection{Geometric Formulation for Local Wellposedness}

We will prove the local wellposedness for arbitrary initial data and
arbitrary smooth bottom, so the mapping should be related to the
data itself. Hence, we define the $\epsilon$-Poisson integral for
$\e$:
\begin{equation}
\bar\eta^{\epsilon}=\pp^{\epsilon}\e= \left\{
\begin{array}{ll}
\int_{R^2}\hat{\e}(\xi)e^{\epsilon\abs{\xi}x_3}e^{2\pi
ix'\cdot\xi}\ud{\xi}& \rm{for}\ \rm{infinite}\ \rm{case}\\
\\
\sum_{n\in(L_1^{-1}\mathbb{Z})\times(L_2^{-1}\mathbb{Z})}e^{2\pi
in\cdot x'}e^{\epsilon\abs{n}x_3}\hat{\e}(n)& \rm{for}\ \rm{periodic}\ \rm{case}\\
\end{array}
\right.
\end{equation}
where $\hat{\e}(\xi)$ and $\hat{\e}(n)$ denotes the Fourier
transform of $\e(x')$ in either continuous or discrete form and
$0<\epsilon<1$ is the parameter. The detailed
definitions are shown in (\ref{appendix poisson def 1}) and (\ref{appendix poisson def 2}).\\
Consider the geometric transform from $\Omega$ to $\Omega(t)$:
\begin{equation}\label{map}
\Phi^{\epsilon}:(x_1,x_2,x_3)\mapsto
(x_1,x_2,\frac{b}{b_0}x_3+\bar\eta^{\epsilon}(1+\frac{x_3}{b_0}))=(y_1,y_2,y_3)
\end{equation}
This transform maps $\Omega$ into $\Omega(t)$ and its Jacobian
matrix
\begin{equation}
\nabla\Phi^{\epsilon}= \left(
\begin{array}{ccc}
1&0&0\\
0&1&0\\
A^{\epsilon}&B^{\epsilon}&J^{\epsilon}
\end{array}
\right)
\end{equation}
and the transform matrix
\begin{equation}
\mathcal{A}^{\epsilon}= ((\nabla\Phi^{\epsilon})^{-1})^T=\left(
\begin{array}{ccc}
1&0&-A^{\epsilon}K^{\epsilon}\\
0&1&-B^{\epsilon}K^{\epsilon}\\
0&0&K^{\epsilon}
\end{array}
\right)
\end{equation}
where
\begin{eqnarray}
\tilde b=1+\frac{x_3}{b_0}\nonumber\\
A^{\epsilon}=\frac{\p_1b}{b_0}x_3+\p_1\eb^{\epsilon}\tilde b&&
B^{\epsilon}=\frac{\p_2b}{b_0}x_3+\p_2\eb^{\epsilon}\tilde b\\
J^{\epsilon}=\frac{b}{b_0}+\frac{\eb^{\epsilon}}{b_0}+\partial_3\eb^{\epsilon}\tilde{b}&&
K^{\epsilon}=\frac{1}{J^{\epsilon}}\nonumber
\end{eqnarray}
Note that this transform is not necessarily a homomorphism for
arbitrary $\epsilon$. By our assumption on initial data that
$\ee(x')+b_0(x')\geq\rho>0$ and theorem \ref{geometric transform},
we can always choose a sufficiently small $\epsilon$ depending on
$\hm{\eta_0}{5/2}$ such that there exists a $\delta>0$ satisfying
$J^{\epsilon}(0)>\delta>0$, then this nonzero Jacobi implies the
transform $\Phi^{\epsilon}$ is a homomorphism at $t=0$. If we
further assume $\ee\in H^{7/2}(\Sigma)$, we can conclude
$\Phi^{\epsilon}$ is a $C^1$ diffeomorphism from $\Omega$ to
$\Omega_0$. In the following, we will just write $\bar\eta$ instead
of $\bar\eta^{\epsilon}$ for simplicity, and the same fashion
applies to
$\a$, $\Phi$, $A$, $B$, $J$ and $K$.\\
Define some transformed operators as follows.
\begin{equation}\label{introduction temp 1}
\begin{array}{l}
(\nabla_{\mathcal{A}}f)_i=\mathcal{A}_{ij}\partial_jf\\
\nabla_{\mathcal{A}}\cdot\vec g=\mathcal{A}_{ij}\partial_jg_i\\
\Delta_{\mathcal{A}}f=\nabla_{\mathcal{A}}\cdot\nabla_{\mathcal{A}}f\\
\mathcal{N}=(-\partial_1\eta,-\partial_2\eta,1)\\
(\mathbb{D}_{\mathcal{A}}u)_{ij}=\mathcal{A}_{ik}\partial_ku_j+\mathcal{A}_{jk}\partial_ku_i\\
S_{\mathcal{A}}(p,u)=pI-\mathbb{D}_{\mathcal{A}}u
\end{array}
\end{equation}
where the summation should be understood in the Einstein convention.
If we extend the divergence $\da$ to act on symmetric tensor in the
natural way,
then a straightforward computation reveals that $\da\sa(p,u)=\na p-\la u$ for vector fields satisfying $\da u=0$.\\
In our new coordinate, the original equation system (\ref{origin})
becomes
\begin{equation}\label{transform}
\left\{
\begin{array}{ll}
\partial_tu-\partial_t\bar\eta\tilde{b}K\partial_3u
+u\cdot\nabla_{\mathcal {A}}u-\Delta_{\mathcal {A}}u+\nabla_{\mathcal {A}}p=0 &\rm{in}\quad \Omega\\
\nabla_{\mathcal {A}}\cdot u=0  &\rm{in} \quad\Omega \\
S_{\mathcal {A}}(p,u)\mathcal {N}=\eta\mathcal {N}& \rm{on} \quad\Sigma\\
u=0  &\rm{on}\quad\Sigma_{b}\\
u(x,0)=u_0(x)&\rm{in}\quad\Omega\\
&\\
\partial_t\eta+u_1\partial_1\eta+u_2\partial_2\eta=u_3 &\rm{on}\quad \Sigma\\
\eta(x',0)=\eta_0(x')&\rm{on}\quad\Sigma
\end{array}
\right.
\end{equation}
where we can split the system into a Naiver-Stokes equation and a
transport equation.\\
Since $\a$ is determined by $\e$ through the transform, so all the
quantities above is related to $\e$, i.e. the geometric structure of
the free surface. This is the central idea of our proof. It is
noticeable that in proving local wellposedness of above equation
system, we must verify $\Phi(t)$ is a $C^1$ diffeomorphism for any
$t\in[0,T]$, where the theorem holds.

\subsubsection{Geometric Formulation for Global Wellposedness}

In the global part, we need to consider the infinite case when
domain $\Omega(t)$ always possess a flat bottom, which means
$b=\rm{constant}$. Also our theorem only holds for sufficiently
small data, so we may directly define $\eb$ as the Poisson integral
of $\e$ which is shown in (\ref{appendix poisson integral}), i.e.
$\eb=\pp\e$. Then the map can be taken as
\begin{equation}\label{map_}
\Phi:(x_1,x_2,x_3)\mapsto
(x_1,x_2,x_3+\bar\eta(1+\frac{x_3}{b}))=(y_1,y_2,y_3)
\end{equation}
which also successfully transform $\Omega$ to $\Omega(t)$. It is
obvious that the Jacobi matrix and transform matrix are identical to
those in the local case except that we need to redefine the
quantities
\begin{eqnarray}
\tilde b=1+x_3/b\nonumber\\
A=\p_1\eb\tilde b&&
B=\p_2\eb\tilde b\\
J=1+\eb/b+\partial_3\eb\tilde{b}&& K=1/J\nonumber
\end{eqnarray}
Naturally, the definition of weighted operators in
(\ref{introduction temp 1}) and the transformed equation
(\ref{transform}) are the same as those in local part with replacing
the related quantities with above definition. We can easily see that
as long as $\hms{\ee}{5/2}{\Sigma}$ is sufficiently small, $\Phi$ is
a diffeomorphism. Hence, we do not need the discussion as local
part.

\subsection{Main Theorem}

In this paper, we will prove both the local wellposedness and global
wellposedness for higher order regularity.
\begin{theorem}
Let $N\geq3$ be an integer. Assume the initial data
$\ee+b\geq2b_0\delta>0$ for some $\delta>0$. Suppose that $u_0$ and
$\ee$ satisfy the estimate
$K_0=\hm{u_0}{2N}^2+\hms{\ee}{2N+1/2}{\Sigma}^2<\infty$ as well as
the $N^{th}$ compatible condition (\ref{compatible condition}). Then
there exists $0<T_0<1$ such that for $0<T<T_0$, there exists a
unique solution $(u,p,\e)$ to system (\ref{transform}) on the
interval $[0,T]$ that achieves the initial data. Furthermore, the
solution obeys the estimate
\begin{eqnarray}
\bigg(\sum_{j=0}^N\sup_{0\leq t\leq
T}\hm{\dt^ju}{2N-2j}^2+\sum_{j=0}^N\int_0^T\hm{\dt^ju}{2N-2j+1}^2+\nm{\dt^{N+1}u}_{\xxst}^2\bigg)\\
+\bigg(\sum_{j=0}^{N-1}\sup_{0\leq t\leq
T}\hm{\dt^jp}{2N-2j-1}^2+\sum_{j=0}^N\int_0^T\hm{\dt^jp}{2N-2j}^2\bigg)\nonumber\\
+\bigg(\sup_{0\leq t\leq
T}\hms{\e}{2N+1/2}{\Sigma}^2+\sum_{j=1}^N\sup_{0\leq t\leq
T}\hms{\dt^j\e}{2N-2j+3/2}{\Sigma}^2\nonumber\\
+\int_0^T\hms{\e}{2N+1/2}{\Sigma}^2+\int_0^T\hms{\dt\e}{2N-1/2}{\Sigma}^2+\sum_{j=2}^{N+1}\int_0^T\hms{\dt^j\e}{2N-2j+5/2}{\Sigma}^2\bigg)\nonumber\\
\leq C(\Omega_0,\delta)P(K_0)\nonumber
\end{eqnarray}
where $C(\Omega_0,\delta)$ is a positive constant depending on the
initial domain $\Omega_0$ and separation quantity $\delta$ and
$P(\cdot)$ is a single variable polynomial satisfying $P(0)=0$. The
solution is unique among functions that achieves the initial data
and for which the left hand side of the estimate is finite.
Moreover, $\e$ is such that the mapping $\Phi(t)$ defined by
(\ref{map}) is a $C^{2N-2}$ diffeomorphism for each $t\in[0,T]$.
\end{theorem}
\begin{remark}
In above theorem, $H^k$ denotes the usual Sobolev space in $\Omega$
and $H^k(\Sigma)$ denotes the usual Sobolev space on $\Sigma$, while
we will state the definition of $\nm{\cdot}_{\xxst}$ later.\\
Because the map $\Phi(t)$ is a $C^{2N-2}$ diffeomorphism, then we
may change variable to $y\in \Omega(t)$ to produce solution of
(\ref{origin}).
\end{remark}
\begin{theorem}
Let $N\geq4$ be an integer. Suppose the initial data $(u_0,\ee)$
satisfy the appropriate compatible conditions stated in local
wellposedness theorem. Then there exists a $\kappa>0$ such that if
\begin{eqnarray}
\ce_{2N}(0)+\f_{2N}(0)\leq\kappa
\end{eqnarray}
then there exists a unique solution $(u,p,\e)$ on the interval
$[0,\infty)$ that achieves the initial data. Also, the solution
obeys the estimate
\begin{eqnarray}
\g_{N}(\infty)\leq C(\ce_{2N}(0)+\f_{2N}(0))\leq C\kappa
\end{eqnarray}
where $C>0$ is a universal constant.
\end{theorem}
\begin{remark}
The detailed definition for $\ce_{2N}$, $\f_{2N}$ and $\g_{N}$ is
presented in (\ref{definition 1}) to (\ref{definition 2}). This
theorem easily implies the algebraic decaying of energy
$\ce_{N+2,2}$ and $\ce_{N+2,1}$.
\end{remark}

\subsection{Convention and Terminology}

We now mention some of the definitions, bits of notations and
conventions we will use throughout the paper.
\begin{enumerate}
\item
We will employ Einstein summation convention to sum up repeated
indices for vector and tensor operations.
\item
Throughout the paper $C>0$ will denote a constant only depend on the
parameter of the problem, $N$ and $\Omega$, but does not depend on
the data. They are referred as universal and can change from one
inequality to another one.  When we write $C(z)$, it means a certain
positive constant depending
on quantity $z$. \\
There are two exceptions to above rules in the local part. The first
one is that in the elliptic estimates and Korn's inequality, there
are constants depending on the initial domain $\Omega_0$. Although
this should be understood as depending on the initial free surface
$\ee$, since its dependent relation is given implicitly and cannot
be simplified further, we will also call them universal. The second
one is that we will also call the constant $C(\delta)$ universal,
though $\delta$ is determined by initial domain $\Omega_0$. To note
that apart from these, all the other constants related to initial
data $\Omega_0$, $u_0$ and $\ee$ should be specified in detail.
\item
We will employ notation $a\ls b$ to denote $a\leq Cb$, where $C$ is
a universal constant as defined above.
\item
We will write $P(\cdot)$ to denote the single variable polynomial.
This type of polynomial always satisfies $P(0)=1$ unless it is
specified as an exception. This notation is used to denote some very
complicated polynomial expressions, however whose details we do not
really care. This kind of polynomial may change from line to line.
\item
For convenience, we will typically write $H^0=L^2$, except for
notation $L^2([0,T];H^k)$. We write $H^k(\Omega)$ with $k\geq0$ and
$H^s(\Sigma)$ with $s\in R$ for standard Sobolev space. Same style
of notations also holds for $L^2([0,T];H^k)$ and
$L^{\infty}([0,T];H^k)$. When we write $\hm{\cdot}{k}$, this always
means the Sobolev norm in $\Omega$, otherwise, we will point out the
exact space it stands for, e.g. $\hms{\cdot}{k}{\Sigma}$. A similar
fashion is adopted for $\tnm{\cdot}{k}$ and $\inm{\cdot}{k}$.
\item
We write the multi-indices
$\mathbb{N}^{1+m}=\{\alpha=(\alpha_0,\alpha_1,\ldots,\alpha_m)\}$ to
emphasize that the $0$-index term is related to temporal
derivatives. For just spatial derivatives, we write $\mathbb{N}^m$
without the $0$-index. We define the parabolic counting of such
multi-indices as $\abs{\alpha}=2\alpha_0+\alpha_1+\ldots+\alpha_m$.
We always write $Df$ to denote the horizontal derivative of $f$ and
$\nabla f$ for the full derivative.\\
For a given norm $\nm{\cdot}$ and integers $k,m\geq0$, we introduce
the following notation for the sums of spatial derivatives.
\begin{eqnarray*}
\nm{D_m^kf}^2=\sum_{\alpha\in\mathbb{N}^2\ m\leq\abs{\alpha}\leq
k}\nm{\p^{\alpha}f}^2&&
\nm{\nabla_m^kf}^2=\sum_{\alpha\in\mathbb{N}^3\
m\leq\abs{\alpha}\leq k}\nm{\p^{\alpha}f}^2\\
\nm{\bar D_m^kf}^2=\sum_{\alpha\in\mathbb{N}^{1+2}\
m\leq\abs{\alpha}\leq k}\nm{\p^{\alpha}f}^2&&
\nm{\bar\nabla_m^kf}^2=\sum_{\alpha\in\mathbb{N}^{1+3}\
m\leq\abs{\alpha}\leq k}\nm{\p^{\alpha}f}^2
\end{eqnarray*}
where $D$ and $\nabla$ means horizontal derivatives and $\bar D$ and
$\bar\nabla$ means full derivatives.\\
Also we define
\begin{eqnarray*}
\nm{D^kf}^2=\nm{D_k^kf}^2\ \nm{\nabla^kf}^2=\nm{\nabla_k^kf}^2\\
\nm{\bar D^kf}^2=\nm{\bar D_k^kf}^2\ \nm{\bar \nabla^kf}^2=\nm{\bar
\nabla_k^kf}^2
\end{eqnarray*}
\end{enumerate}

\subsection{Structure of This Paper}

In Section 2, we will discuss the local behavior for arbitrary
initial data, which includes three main parts: estimates for linear
Navier-Stokes equation, estimates for transport equation and the
construction of iteration. Combining all these, the local
wellposedness easily follows.\\
In Section 3, we will prove the global wellposedness for small data
in horizontally infinite domain. Since many lemmas have been proved
in \cite{book9}, here we only present the main part of our
simplification in interpolation theorem and comparison theorem, with
merely stating the result for other parts.

\section{Local Wellposedness for Large Data}

\subsection{$\epsilon$-Poisson Integral}

In this section, we will define the $\epsilon$-Poisson integral,
which is utilized to construct geometric transform as (\ref{map}),
and discuss the basic properties of it.

\subsubsection{Poisson Integral in Infinite Case}

For a function $f$ defined on $\Sigma=R^2$, the $\epsilon$-Poisson
integral is defined by
\begin{eqnarray}\label{appendix poisson def 1}
\pp^{\epsilon}f(x',x_3)=\int_{R^2}\hat{f}(\xi)e^{\epsilon\abs{\xi}x_3}e^{2\pi
ix'\cdot\xi}\ud{\xi}
\end{eqnarray}
where $\hat{f}(\xi)$ denotes the Fourier transform of $f(x')$ in
$R^2$ and $0<\epsilon<1$ is the parameter. Although
$\pp^{\epsilon}f$ is defined on $R^2\times(-\infty,0)$, we will only
consider the part $R^2\times(-b_0,0)$ here.
\begin{lemma}\label{appendix poisson 1}
Let $\pp^{\epsilon}f$ be the $\epsilon$-Poisson integral of function
$f$ which is in homogeneous Sobolev space $\dot{H}^{q-1/2}(\Sigma)$
for $q\in\mathbb{N}$. Then we have
\begin{eqnarray}
\hm{\nabla^q\pp^{\epsilon}f}{0}^2\leq\frac{C}{\epsilon}\nm{f}^2_{\dot{H}^{q-1/2}}
\end{eqnarray}
where $C>0$ is a constant independent of $\epsilon$. In particular,
we have
\begin{eqnarray}\label{Appendix temp 1}
\hm{\pp^{\epsilon}f}{q}^2\leq
\frac{C}{\epsilon}\hms{f}{q-1/2}{\Sigma}^2
\end{eqnarray}
\end{lemma}
\begin{proof}
By definition of Fourier transform, Fubini theorem and Parseval
identity, we may bound
\begin{eqnarray*}
\hm{\nabla^q\pp^{\epsilon}f}{0}^2&\leq&
(2\pi)^{2q}\int_{R^2}\int_{-b_0}^0\abs{\xi}^{2q}\abs{\hat{f}(\xi)}^2e^{2\epsilon\abs{\xi}x_3}\ud{\xi}\ud{x_3}
\leq(2\pi)^{2q}\int_{R^2}\abs{\xi}^{2q}\abs{\hat{f}(\xi)}^2\ud{\xi}\int_{-b_0}^0e^{2\epsilon\abs{\xi}x_3}\ud{x_3}\\
&=&(2\pi)^{2q}\int_{R^2}\abs{\xi}^{2q}\abs{\hat{f}(\xi)}^2
\bigg(\frac{1-e^{-2\epsilon
b_0\abs{\xi}}}{2\epsilon\abs{\xi}}\bigg)\ud{\xi}
\leq\frac{(2\pi)^{2q}}{2\epsilon}\int_{R^2}\abs{\xi}^{2q-1}\abs{\hat{f}(\xi)}^2\ud{\xi}\\
&\leq&\frac{\pi}{\epsilon}\nm{f}^2_{\dot{H}^{q-1/2}(\Sigma)}\nonumber
\end{eqnarray*}
We can simply take $C=\pi$ to achieve the estimate. Furthermore,
considering
\begin{displaymath}
\frac{1-e^{-2\epsilon b_0\abs{\xi}}}{2\epsilon\abs{\xi}}\leq
\min\bigg\{\frac{1}{2\epsilon\abs{\xi}},b_0\bigg\}
\end{displaymath}
and $0<\epsilon<1$, we have
\begin{equation}
\hm{\pp^{\epsilon}f}{0}^2\leq \frac{\pi
b_0}{\epsilon}\hms{f}{0}{\Sigma}^2
\end{equation}
Hence, (\ref{Appendix temp 1}) easily follows.
\end{proof}
We also need the $L^{\infty}$ estimate for Poisson integral.
\begin{lemma}\label{appendix poisson 2}
Let $\pp^{\epsilon}f$ be the $\epsilon$-Poisson integral of function
$f$. Then for $q\in\mathbb{N}$ and $s>1$, we have
\begin{eqnarray}
\lnm{\nabla^q\pp^{\epsilon}f}{\infty}^2\leq C\hms{f}{q+s}{\Sigma}^2
\end{eqnarray}
\end{lemma}
\begin{proof}
A simple application of Sobolev embedding and lemma \ref{appendix
poisson 1} reveals
\begin{eqnarray*}
\lnm{\nabla^q\pp^{\epsilon}f}{\infty}^2\leq
C\hm{\nabla^q\pp^{\epsilon}f}{s+1/2}^2\leq C\hms{f}{q+s}{\Sigma}^2
\end{eqnarray*}
\end{proof}
The following lemma illustrate the specialty of derivative in
vertical direction.
\begin{lemma}\label{appendix poisson 3}
Let $\pp^{\epsilon}f$ be the $\epsilon$-Poisson integral of function
$f$. Then we have
\begin{eqnarray}
\lnm{\p_3\pp^{\epsilon}f}{\infty}^2\leq
C\epsilon\hms{f}{5/2}{\Sigma}^2
\end{eqnarray}
where $C>0$ is a constant independent of $\epsilon$.
\end{lemma}
\begin{proof}
We can simply bound as follows.
\begin{equation}
\lnm{\p_3\pp^{\epsilon}f}{\infty}^2\leq
C\hm{\p_3\pp^{\epsilon}f}{2}^2\leq C\epsilon^2
\hm{\pp^{\epsilon}f}{3}^2\leq C\epsilon\hms{f}{5/2}{\Sigma}^2
\end{equation}
Note that the first inequality is based on Sobolev embedding theorem
and the third one is a straightforward application of
lemma \ref{appendix poisson 1}, so we only need to verify the second inequality in detail.\\
By definition, it is easy to see
\begin{eqnarray}
\p_3\pp^{\epsilon}f=\epsilon\int_{R^2}\abs{\xi}\hat{f}(\xi)e^{\epsilon\abs{\xi}x_3}e^{2\pi
ix'\cdot\xi}\ud{\xi}
\end{eqnarray}
Hence, we have
\begin{eqnarray*}
\hm{\p_3\pp^{\epsilon}f}{0}^2&=&\epsilon^2\int_{R^2}\int_{-b_0}^0\abs{\xi}^2\abs{\hat{f}(\xi)}^2e^{2\epsilon\abs{\xi}x_3}\ud{\xi}\ud{x_3}
=
\frac{\epsilon^2}{(2\pi)^2}\bigg((2\pi)^2\int_{R^2}\int_{-b_0}^0\abs{\xi}^2\abs{\hat{f}(\xi)}^2e^{2\epsilon\abs{\xi}x_3}\ud{\xi}\ud{x_3}\bigg)\nonumber\\
&\leq&\frac{\epsilon^2}{(2\pi)^2}\hm{\p_1\pp^{\epsilon}f}{0}^2
\end{eqnarray*}
Similarly, we can show that for $i,j=1,2,3$
\begin{eqnarray*}
\hm{\p_{3i}\pp^{\epsilon}f}{0}^2\leq
\frac{\epsilon^2}{(2\pi)^4}\hm{\p_{11}\pp^{\epsilon}f}{0}^2
\end{eqnarray*}
\begin{eqnarray*}
\hm{\p_{3ij}\pp^{\epsilon}f}{0}^2\leq
\frac{\epsilon^2}{(2\pi)^6}\hm{\p_{111}\pp^{\epsilon}f}{0}^2
\end{eqnarray*}
So we have
\begin{equation}
\hm{\p_3\bar\eta^{\epsilon}}{2}^2\leq C\epsilon^2
\hm{\bar\eta^{\epsilon}}{3}^2
\end{equation}
\end{proof}

\subsubsection{Poisson Integral in Periodic Case}

Suppose $\Sigma=(L_1\mathbb{T})\times(L_2\mathbb{T})$, we define the
$\epsilon$-Poisson integral as follows.
\begin{eqnarray}\label{appendix poisson def 2}
\pp^{\epsilon}f(x',x_3)=\sum_{n\in(L_1^{-1}\mathbb{Z})\times(L_2^{-1}\mathbb{Z})}e^{2\pi
in\cdot x'}e^{\epsilon\abs{n}x_3}\hat{f}(n)
\end{eqnarray}
where
\begin{eqnarray}
\hat{f}(n)=\int_{\Sigma}f(x')\frac{e^{-2\pi in\cdot
x'}}{L_1L_2}\ud{x'}
\end{eqnarray}
\begin{lemma}\label{appendix poisson 4}
Let $\pp^{\epsilon}f$ be the $\epsilon$-Poisson integral of function
$f$ which is in homogeneous Sobolev space $\dot{H}^{q-1/2}(\Sigma)$
for $q\in\mathbb{N}$. Then we have
\begin{eqnarray}
\hm{\nabla^q\pp^{\epsilon}f}{0}^2\leq\frac{C}{\epsilon}\nm{f}^2_{\dot{H}^{q-1/2}}
\end{eqnarray}
where $C>0$ is a constant independent of $\epsilon$. In particular,
we have
\begin{eqnarray}\label{Appendix temp 2}
\hm{\pp^{\epsilon}f}{q}^2\leq
\frac{C}{\epsilon}\hms{f}{q-1/2}{\Sigma}^2
\end{eqnarray}
\end{lemma}
\begin{proof}
By Fubini theorem and Parseval identity, we may bound
\begin{eqnarray*}
\hm{\nabla^q\pp^{\epsilon}f}{0}^2&\leq&
(2\pi)^{2q}\sum_{n\in(L_1^{-1}\mathbb{Z})\times(L_2^{-1}\mathbb{Z})}
\int_{-b_0}^0\abs{n}^{2q}e^{2\epsilon\abs{n}x_3}\abs{\hat{f}(n)}^2\ud{x_3}\nonumber\\
&=&(2\pi)^{2q}\sum_{n\in(L_1^{-1}\mathbb{Z})\times(L_2^{-1}\mathbb{Z})}
\abs{n}^{2q}\abs{\hat{f}(n)}^2\bigg(\frac{1-e^{-2\epsilon b_0\abs{n}}}{2\epsilon\abs{n}}\bigg)\nonumber\\
&\leq&\frac{(2\pi)^{2q}}{2\epsilon}\sum_{n\in(L_1^{-1}\mathbb{Z})\times(L_2^{-1}\mathbb{Z})}
\abs{n}^{2q-1}\abs{\hat{f}(n)}^2\leq\frac{\pi}{\epsilon}\nm{f}^2_{\dot{H}^{q-1/2}}
\end{eqnarray*}
We can simply take $C=\pi$ to achieve the estimate. Furthermore,
considering
\begin{displaymath}
\frac{1-e^{-2\epsilon b_0\abs{n}}}{2\epsilon\abs{n}}\leq
\min\bigg\{\frac{1}{2\epsilon\abs{n}},b_0\bigg\}
\end{displaymath}
and $0<\epsilon<1$, we have
\begin{equation}
\hm{\pp^{\epsilon}f}{0}^2\leq \frac{\pi
b_0}{\epsilon}\hms{f}{0}{\Sigma}^2
\end{equation}
Hence, (\ref{Appendix temp 2}) easily follows.
\end{proof}
Similar to infinite case, we give an $L^{\infty}$ estimate for
Poisson integral.
\begin{lemma}\label{appendix poisson 5}
Let $\pp^{\epsilon}f$ be the $\epsilon$-Poisson integral of function
$f$. Then for $q\in\mathbb{N}$ and $s>1$, we have
\begin{eqnarray*}
\lnm{\nabla^q\pp^{\epsilon}f}{\infty}^2\leq C\hms{f}{q+s}{\Sigma}^2
\end{eqnarray*}
\end{lemma}
\begin{proof}
A simple application of Sobolev embedding and lemma (\ref{appendix
poisson 4}) reveals
\begin{eqnarray}
\lnm{\nabla^q\pp^{\epsilon}f}{\infty}^2\leq
C\hm{\nabla^q\pp^{\epsilon}f}{s+1/2}^2\leq C\hms{f}{q+s}{\Sigma}^2
\end{eqnarray}
\end{proof}
We still need the following lemma to illustrate the specialty of
derivative in vertical direction.
\begin{lemma}\label{appendix poisson 6}
Let $\pp^{\epsilon}f$ be the $\epsilon$-Poisson integral of function
$f$. Then we have
\begin{eqnarray}
\lnm{\p_3\pp^{\epsilon}f}{\infty}^2\leq
C\epsilon\hms{f}{5/2}{\Sigma}^2
\end{eqnarray}
where $C>0$ is a constant independent of $\epsilon$.
\end{lemma}
\begin{proof}
We can simply bound as follows.
\begin{equation}
\lnm{\p_3\pp^{\epsilon}f}{\infty}^2\leq
C\hm{\p_3\pp^{\epsilon}f}{2}^2\leq C\epsilon^2
\hm{\pp^{\epsilon}f}{3}^2\leq C\epsilon\hms{f}{5/2}{\Sigma}^2
\end{equation}
Note that the first inequality is based on Sobolev embedding theorem
and the third one is a straightforward application of
lemma (\ref{appendix poisson 1}), so we only need to verify the second inequality in details.\\
By definition, it is easy to see
\begin{eqnarray}
\p_3\pp^{\epsilon}f=\epsilon\sum_{n\in(L_1^{-1}\mathbb{Z})\times(L_2^{-1}\mathbb{Z})}\abs{n}\hat{f}(\xi)e^{\epsilon\abs{n}x_3}e^{2\pi
ix'\cdot n}
\end{eqnarray}
Hence, we have
\begin{eqnarray*}
\hm{\p_3\pp^{\epsilon}f}{0}^2&=&\epsilon^2\sum_{n\in(L_1^{-1}\mathbb{Z})\times(L_2^{-1}\mathbb{Z})}
\int_{-b_0}^0\abs{n}^2\abs{\hat{f}(n)}^2e^{2\epsilon\abs{n}x_3}\ud{x_3}\nonumber\\
&=&
\frac{\epsilon^2}{(2\pi)^2}\bigg((2\pi)^2\sum_{n\in(L_1^{-1}\mathbb{Z})\times(L_2^{-1}\mathbb{Z})}
\int_{-b_0}^0\abs{n}^2\abs{\hat{f}(n)}^2e^{2\epsilon\abs{n}x_3}\ud{x_3}\bigg)\nonumber\\
&\leq&\frac{\epsilon^2}{(2\pi)^2}\hm{\p_1\pp^{\epsilon}f}{0}^2
\end{eqnarray*}
Similarly, we can show that for $i,j=1,2,3$
\begin{eqnarray*}
\hm{\p_{3i}\pp^{\epsilon}f}{0}^2\leq
\frac{\epsilon^2}{(2\pi)^4}\hm{\p_{11}\pp^{\epsilon}f}{0}^2
\end{eqnarray*}
\begin{eqnarray*}
\hm{\p_{3ij}\pp^{\epsilon}f}{0}^2\leq
\frac{\epsilon^2}{(2\pi)^6}\hm{\p_{111}\pp^{\epsilon}f}{0}^2
\end{eqnarray*}
So we have
\begin{equation}
\hm{\p_3\bar\eta^{\epsilon}}{2}^2\leq C\epsilon^2
\hm{\bar\eta^{\epsilon}}{3}^2
\end{equation}
\end{proof}

\subsubsection{Homomorphism of the Geometric Transform}

\begin{theorem}\label{geometric transform}
If the initial data satisfies $\ee+b_0>\rho>0$ for all
$x'\in\Sigma$, then we can choose a $\epsilon>0$ such that there
exists a $\delta>0$ satisfying $J^{\epsilon}(0)\geq\delta>0$.
\end{theorem}
\begin{proof}
Naturally there always exists $\delta>0$, such that
$\eta_0(x')+b_0(x')\geq2b_0\delta>0$ for arbitrary $x'\in\Sigma$.
Based on lemma (\ref{appendix poisson 3}) and lemma (\ref{appendix
poisson 6}), if we take $\epsilon\leq
\delta^2/(4C^2\hm{\eta_0}{5/2}^2)$, we have
\begin{equation}
\lnm{\p_3\bar\eta_0^{\epsilon}}{\infty}\leq\delta/2
\end{equation}
where $\bar\eta_0^{\epsilon}$ denotes $\eb^{\epsilon}$ at $t=0$.
Furthermore, by fundamental theorem of calculus, we can estimate for
any $x_3\in[-b_0,0]$ and $x'\in\Sigma$
\begin{eqnarray}
\abs{\eb_0^{\epsilon}(x',x_3)-\eb_0^{\epsilon}(x',0)}\leq
\int_{-b_0}^0\abs{\p_3\eb_0^{\epsilon}(x', z)}\ud{z} \leq
b_0\lnm{\p_3\bar\eta_0^{\epsilon}}{\infty} \leq b_0\delta/2
\end{eqnarray}
Since $\eb_0^{\epsilon}(x',0)=\ee(x')$, we have
\begin{eqnarray}
J^{\epsilon}(0)&=&\frac{b}{b_0}+\frac{\eb^{\epsilon}_0}{b_0}+\partial_3\eb^{\epsilon}_0\tilde{b}
=\frac{b+\ee}{b_0}+\frac{\eb^{\epsilon}_0-\ee}{b_0}+\partial_3\eb^{\epsilon}_0\tilde{b}\nonumber\\
&\geq&2\delta-\delta/2-\delta/2=\delta>0
\end{eqnarray}
This means for given initial data $\ee$, if we take
$\epsilon=\delta^2/(4C^2\hm{\eta_0}{5/2}^2)$, then
$J^{\epsilon}\geq\delta>0$.
\end{proof}

\subsection{Preliminaries}

\subsubsection{Transform Estimates}

In order to study the linear problem in the slab domain, we will
employ the idea that transforming the variable-coefficient problem
into a constant-coefficient problem through diffeomorphism. So
before estimating, we need to confirm the mapping $\Phi$ is an
isomorphism from $\Omega$ to $\Omega(t)$ and determine the relation
of corresponding norms between these two spaces.
\begin{lemma}\label{transform estimate}
Let $\Psi: \Omega\rightarrow\Omega'$ be a $C^1$ diffeomorphism
satisfying $\Psi \s{k+1}_{loc}$, the Jacobi
$J=\det(\nabla\Psi)>\delta>0$ a.e. in $\Omega$ and $\nabla\Psi-I
\s{k}(\Omega)$ for an integer $k\geq3$. If $v \s{m}(\Omega')$, then
$v\circ\Psi\s{m}(\Omega)$ for $m=0,1,\ldots,k+1$, and
\begin{equation}\label{transform estimate 1}
\|v\circ\Psi\|_{H^m(\Omega)}\lesssim
P(\hms{\nabla\Psi-I}{k}{\Omega})\|v\|_{H^m(\Omega')}
\end{equation}
for a polynomial $P(\cdot)$. Similarly, for $u\in H^m(\Omega)$,
$u\circ\Psi^{-1}\in H^m(\Omega')$ for $m=0,1,\ldots,k+1$ and
\begin{equation}\label{transform estimate 2}
\|u\circ\Psi^{-1}\|_{H^m(\Omega')}\lesssim
P(\hms{\nabla\Psi-I}{k}{\Omega})\|u\|_{H^m(\Omega)}
\end{equation}
Let $\Sigma'=\Psi(\Sigma)$ denote the upper boundary of $\Omega'$.
If $v\in H^{m-1/2}(\Sigma')$ for $m=1,\ldots,k-1$, then
$v\circ\Psi\in H^{m-1/2}(\Sigma)$, and
\begin{equation}\label{transform estimate 3}
\|v\circ\Psi\|_{H^{m-1/2}(\Sigma)}\lesssim
C(\hms{\nabla\Psi-I}{k}{\Omega})\|v\|_{H^{m-1/2}(\Sigma')}
\end{equation}
If $u\in H^{m-1/2}(\Sigma)$ for $m=1,\ldots,k-1$, then
$u\circ\Psi^{-1}\in H^{m-1/2}(\Sigma')$ and
\begin{equation}\label{transform estimate 4}
\|u\circ\Psi^{-1}\|_{H^{m-1/2}(\Sigma')}\lesssim
C(\hms{\nabla\Psi-I}{k}{\Omega})\|u\|_{H^{m-1/2}(\Sigma)}
\end{equation}
where $C(\hms{\nabla\Psi-I}{k}{\Omega})$ denote a constant depending
on $\hms{\nabla\Psi-I}{k}{\Omega}$.
\end{lemma}
\begin{proof}
The proof is almost identical to that of lemma 2.3.1 in
\cite{book1}. It is easy to see from the original proof, in
(\ref{transform estimate 1}) and (\ref{transform estimate 2}), the
constant $C$ is actually a polynomial of
$\hms{\nabla\Psi-I}{k}{\Omega}$. All the other part of the proof is
exactly the same.
\end{proof}
\begin{remark}\label{transform remark}
Based on our assumptions on $\bar\eta$ and $b$, it is easy to see
$\Phi$ defined in $\ref{map}$ is a $C^1$ diffeomorphism satisfying
the hypothesis in lemma $\ref{transform estimate}$. Moreover, the
universal constant $C$ now can be taken in a succinct form
\begin{equation}
C=C(\hms{\e}{k+1/2}{\Sigma})
\end{equation}
If we only consider the estimate in the domain $\Omega$ and need to
specify the constant, we can take the explicit form
\begin{equation}
C=P(\hms{\e}{k+1/2}{\Sigma})
\end{equation}
\end{remark}

\subsubsection{Functional Spaces}

Now we introduce several functional spaces. We write $H^{k}(\Omega)$
and $H^{k}(\Sigma)$ for usual Sobolev space of either scalar or
vector functions. Define
\begin{equation}
\begin{array}{l}
W(t)=\{u(t)\in H^{1}(\Omega):\ u(t)|_{\Sigma_{b}}=0\}\\
X(t)=\{u(t)\in H^{1}(\Omega):\
u(t)|_{\Sigma_{b}}=0,\nabla_{\mathcal{A}(t)}\cdot u(t)=0\}
\end{array}
\end{equation}
Moreover, let $L^{2}([0,T];H^{k}(\Omega))$ and
$L^{2}([0,T];H^{k}(\Sigma))$ denotes the usual time-involved Sobolev
space. Define
\begin{equation}
\begin{array}{l}
\mathcal{W}=\{u\in L^{2}([0,T];H^{1}(\Omega)):\ u|_{\Sigma_{b}}=0\}\\
\mathcal{X}=\{u\in L^{2}([0,T];H^{1}(\Omega)):\
u|_{\Sigma_{b}}=0,\nabla_{\mathcal{A}}\cdot u=0\}
\end{array}
\end{equation}
It is easy to see $W$, $X$, $\mathcal{W}$ and $\mathcal{X}$ are all
Hilbert spaces with the inner product defined exactly the same as in
$H^{1}(\Omega)$ or $L^{2}([0,T];H^{1}(\Omega))$. Hence, in the
subsequence, we will take the norms of these spaces the same as
$\|\cdot\|_{H^{1}}$ or $\|\cdot\|_{L^{2}H^{1}}$.\\
 Furthermore, we define the functional space for zero upper boundary.
\begin{equation}
V(t)=\{u(t)\in H^{1}(\Omega):\ u(t)|_{\Sigma}=0\}
\end{equation}
Also we will need an orthogonal decomposition
$H^{0}(\Omega)=Y(t)\oplus Y^{\bot}(t)$, where
\begin{equation}
Y^{\bot}(t)=\{\nabla_{\mathcal{A}(t)}\varphi(t):\ \varphi(t)\in
V(t)\}
\end{equation}
In our use of these spaces, we will often drop the $(t)$ when there is no potential of confusion. \\
Finally we will define the regular divergence-free space in
$H^1(\Omega)$.
\begin{equation}
\begin{array}{l}
X_0=\{u\in H^{1}(\Omega):\
u|_{\Sigma_{b}}=0,\nabla\cdot u=0\}\\
\mathcal{X}_0=\{u\in L^{2}([0,T];H^{1}(\Omega)):\
u|_{\Sigma_{b}}=0,\nabla\cdot u=0\}
\end{array}
\end{equation}
Sometimes, we will consider the functional space with the similar
properties as defined above in transformed space $\Omega'$. If that
is the case, we will employ the notation $W(\Omega')$ or
$X(\Omega')$ to specify the space they live in. The same fashion can
be applied to other functional spaces. \\
We will use $\brh{\cdot,\cdot}$ to denote the normalized inner
product in $H^0$, i.e.
\begin{equation}
\brh{u,v}=\int_{\Omega}J u\cdot v
\end{equation}
and $\brht{\cdot,\cdot}$ to denote the normalized inner product in
$L^2([0,T];H^0)$, i.e.
\begin{equation}
\brht{u,v}=\int_0^t\int_{\Omega}J u\cdot v
\end{equation}
Also we will use $\br{\cdot,\cdot}_{W}$ and $\br{\cdot,\cdot}_{\w}$
to denote the regular inner product in $H^1$ if both arguments are
in $W$ or $\w$. Moreover, we define the regular inner product on
$\Sigma$ as follows.
\begin{equation}
\brs{u,v}=\int_{\Sigma}u\cdot v
\end{equation}
\begin{equation}
\brhs{u,v}=\int_0^t\int_{\Sigma}u\cdot v
\end{equation}
\ \\
The following result gives a version of Korn's type inequality for
initial domain $\Omega_0$. Here we employ Beale's idea in
\cite{book2}.
\begin{lemma}
Suppose $\Omega_0$ is the initial domain and $\e_0\in
H^{5/2}(\Sigma)$. Then
\begin{eqnarray}
\hms{v}{1}{\Omega_0}^2\ls\hms{\dm v}{0}{\Omega_0}^2
\end{eqnarray}
for any $v\in H^1(\Omega_0)$ and $v=0$ on $\{y_3=-b\}$.
\end{lemma}
\begin{proof}
In the periodic case, $\Sigma$ is bounded and $v|_{\Sigma_b}=0$, so
naturally Korn's inequality is valid (see proof of lemma 2.7 in
\cite{book2}). Then we consider the infinite case. First we prove
the decaying property of initial surface $\ee$, i.e. for
$\abs{\alpha}\leq1$, $\abs{\p^{\alpha}\eb_0}\rightarrow0$ as
$\abs{x}\rightarrow\infty$ in the slab. For horizontal derivative
$\p^{\alpha}$,
\begin{eqnarray*}
\p^{\alpha}\eb_0&=&\int_{R^2}(2\pi i\xi)^{\alpha}e^{\epsilon
x_3\abs{\xi}}e^{2\pi ix'\cdot\xi}\hat{\e_0}(\xi)\ud{\xi}
\end{eqnarray*}
where $\hat{\e_0}$ is the Fourier transform of $\ee$ in $R^2$. Then
\begin{eqnarray*}
\int_{R^2}\abs{(2\pi i\xi)^{\alpha}e^{\epsilon
x_3\abs{\xi}}\hat{\e_0}(\xi)}\ud{\xi}
&\ls&\int_{R^2}\abs{\xi}^{\alpha}e^{\epsilon
x_3\abs{\xi}}\abs{\hat{\e_0}(\xi)}(1+\abs{\xi}^{5/4})\bigg(\frac{1}{1+\abs{\xi}^{5/4}}\bigg)\ud{\xi}\nonumber\\
&\leq&\bigg(\int_{R^2}\abs{\xi}^{2\alpha}\abs{\hat{\e_0}(\xi)}^2(1+\abs{\xi}^{5/4})^2\ud{\xi}\bigg)^{1/2}
\bigg(\int_{R^2}\frac{e^{2\epsilon
x_3\abs{\xi}}\ud{\xi}}{(1+\abs{\xi}^{5/4})^2}\bigg)^{1/2}\nonumber\\
&\ls&\hm{\e_0}{\alpha+5/4}\bigg(\int_{R^2}\frac{e^{2\epsilon
x_3\abs{\xi}}\ud{\xi}}{(1+\abs{\xi}^{5/4})^2}\bigg)^{1/2}<\infty
\end{eqnarray*}
The last inequality is valid since $x_3<0$. Hence, $(2\pi
i\xi)^{\alpha}e^{\epsilon x_3\abs{\xi}}\hat{\e_0}(\xi)\in L^1(R^2)$.
A similar proof can justify the result if $\p^{\alpha}=\p_3$. Since
$\abs{x}\rightarrow\infty$ naturally implies
$\abs{x'}\rightarrow\infty$ in the slab, the Riemann-Lebesgue lemma
implies $\abs{\p^{\alpha}\eb_0}\rightarrow0$ as $\abs{x}\rightarrow\infty$.\\
We construct a mapping $\bar\sigma=\Phi(0)$ as defined in
(\ref{map}), which maps $\Omega=\{R^3:-b_0<x_3<0\}$ to $\Omega_0$.
Denote $\bar\sigma(x)=x+\sigma(x)$. The above decaying property and
our assumption on $b$ leads to
$\p^{\alpha}\sigma(x)\rightarrow0$ as $\abs{x}\rightarrow\infty$ for $\abs{\alpha}\leq1$.\\
We partition the slab $\Omega$ into cubes
\begin{eqnarray*}
Q_{j,k}=\{x: j<x_1<j+1, k<x_2<k+1, -b_0<x_3<0\}
\end{eqnarray*}
In each cube, we have Korn's inequality
\begin{eqnarray*}
\hms{v}{1}{Q}^2\leq C(\hms{\dm v}{0}{Q}^2+\hms{v}{0}{Q}^2)
\end{eqnarray*}
Employing the compactness argument as Beale did, under the condition
$v=0$ on $\{x_3=-b_0\}$, we can strengthen this result to
\begin{eqnarray*}
\hms{v}{1}{Q}^2\leq C\hms{\dm v}{0}{Q}^2
\end{eqnarray*}
This argument relies on the fact that for such $v$
\begin{eqnarray*}
\hms{v}{0}{Q}^2\leq C\hms{\dm v}{0}{Q}^2
\end{eqnarray*}
Now suppose $D\subseteq\Omega_0$ is the image of union of such cubes
contained in $\{\abs{x}\geq R\}$ for some $R$. Applying last
estimate to $v=u\circ\bar\sigma$ and then transform to $\Omega_0$,
we have
\begin{eqnarray*}
\hms{u}{1}{D}\leq C\hms{\dm u}{0}{D}+\epsilon \hms{u}{1}{D}
\end{eqnarray*}
where $\epsilon<1$ when $R$ is large enough based on the decaying
property. Then we use Korn's inequality in the bounded domain
$\Omega_0-D$ and combine with above estimate to obtain the estimate
required.
\end{proof}
Next, we will show the equivalence of certain quantities. Although
the result is quite obvious for small initial data (see lemma 2.1 of
\cite{book1}), now we need to utilize above Korn's inequality to
show this for arbitrary data.
\begin{lemma}\label{norm equivalence}
There exists a $0<\epsilon_0<1$, such that if
$\hm{\e-\e_0}{5/2}<\epsilon_0$, then the following relation
\begin{eqnarray}
&&\hm{u}{0}^2\ls\int_{\Omega}J\abs{u}^2\ls(1+\hms{\ee}{5/2}{\Sigma})\hm{u}{0}^2\label{norm equivalence 1}\\
&&\frac{1}{(1+\hms{\ee}{5/2}{\Sigma})^3}\hm{u}{1}^2\ls\int_{\Omega}J\abs{\ma
u}^2\ls(1+\hms{\ee}{5/2}{\Sigma})^3\hm{u}{1}^2\label{norm
equivalence 2}
\end{eqnarray}
holds for all $u\in W$.
\end{lemma}
\begin{proof}
(\ref{norm equivalence 1}) is just a natural corollary of relation
$\delta\ls\lnm{J}{\infty}\ls(1+\lnm{\nabla\eb}{\infty})\ls(1+\hm{\e}{5/2})\ls(1+\hm{\ee}{5/2})$
based on our assumption and Sobolev embedding.\\
To derive (\ref{norm equivalence 2}), notice that
\begin{eqnarray}
\int_{\Omega}J\abs{\ma u}^2=\int_{\Omega}J\abs{\dm_{\a_0}
u}^2+\int_{\Omega}J(\ma u+\dm_{\a_0}u):(\ma u-\dm_{\a_0}u)
\end{eqnarray}
For the first term on the right-handed side, based on the integral
substitution and Korn's inequality we proved above, we have
\begin{eqnarray*}
\int_{\Omega}J\abs{\dm_{\a_0} u}^2\ud{x}&\gs&
\frac{1}{1+\hms{\ee}{5/2}{\Sigma}}\int_{\Omega}J_0\abs{\dm_{\a_0}
u}^2\ud{x}
=\frac{1}{1+\hms{\e_0}{5/2}{\Sigma}}\int_{\Omega_0}\abs{\dm v}^2\ud{y}\nonumber\\
&\gs&\frac{1}{1+\hms{\e_0}{5/2}{\Sigma}}\hms{v}{1}{\Omega_0}^2\gs\frac{1}{(1+\hms{\ee}{5/2}{\Sigma})^3}\hm{u}{1}^2
\end{eqnarray*}
For the second term on the right-handed side, using Korn's
inequality in slab $\Omega$, we can naturally estimate
\begin{eqnarray*}
\abs{\int_{\Omega}J(\ma u+\dm_{\a_0}u):(\ma u-\dm_{\a_0}u)}&\ls&
\lnm{J}{\infty}\lnm{\a+\a_0}{\infty}\lnm{\a-\a_0}{\infty}\int_{\Omega}\abs{\nabla u}^2\nonumber\\
&\ls&\epsilon_0(1+\hms{\ee}{5/2}{\Sigma})^3\int_{\Omega}\abs{\dm
u}^2 \ls \epsilon_0(1+\hms{\ee}{5/2}{\Sigma})^3\hms{u}{1}{\Omega}^2
\end{eqnarray*}
For given initial data, we can always take $\epsilon_0$ sufficiently
small to absorb the second term into the first one. Then we have
\begin{eqnarray*}
\int_{\Omega}J\abs{\ma
u}^2\gs\frac{1}{(1+\hms{\ee}{5/2}{\Sigma})^3}\hm{u}{1}^2
\end{eqnarray*}
The first part of (\ref{norm equivalence 2}) is verified. For the
second part, notice that
\begin{eqnarray*}
\int_{\Omega}J\abs{\ma u}^2\ls(1+\hm{\e}{5/2})\int_{\Omega}\abs{\ma
u}^2\ls(1+\hm{\ee}{5/2})\max\{1,\lnm{AK}{\infty}^2,
\lnm{BK}{\infty}^2, \lnm{K}{\infty}^2\}\hm{u}{1}^2\quad\quad
\end{eqnarray*}
Since it is easy to see
\begin{eqnarray*}
\max\{1,\lnm{AK}{\infty}^2, \lnm{BK}{\infty}^2,
\lnm{K}{\infty}^2\}&\ls&1+(1+\lnm{\nabla\eb}{\infty}^2)\lnm{K}{\infty}^2
\ls(1+\hms{\ee}{5/2}{\Sigma})^2
\end{eqnarray*}
then the second part of (\ref{norm equivalence 2}) follows.\\
\end{proof}
\begin{remark}
Throughout this section, we can always assume the restriction of
$\e$ depending on $\epsilon_0$ is justified, and finally we will
verify this condition for $t\in[0,T]$ in the nonlinear part.
\end{remark}
Now we present a lemma about the differentiability of norms in
time-dependent space.
\begin{lemma}
Suppose that $u\in\w$ and $\dt u\in\wwst$. Then the mapping
$t\rightarrow\nm{u(t)}_{H^0}^2$ is absolute continuous, and
\begin{eqnarray}
\frac{\rm{d}}{\rm{dt}}\hm{u(t)}{0}^2=2\br{\dt
u(t),u(t)}+\int_{\Omega}\abs{u(t)}^2\dt J(t)
\end{eqnarray}
Moreover, $u\in C^0([0,T];H^0(\Omega))$. If $v\in\w$ and $\dt
v\in\wwst$ as well, we have
\begin{eqnarray}
\frac{\rm{d}}{\rm{dt}}\int_{\Omega}u(t)\cdot v(t)=\br{\dt
u(t),v(t)}+\br{\dt v(t),u(t)}+\int_{\Omega}u(t)\cdot v(t)\dt J(t)
\end{eqnarray}
\end{lemma}
\begin{proof}
This is exactly the same result as lemma 2.4 in \cite{book1}, so we
omit the proof here.
\end{proof}
Next we show the estimate for $H^{-1/2}$ boundary functions.
\begin{lemma}\label{boundary negative estimate}
If $v\s{0}$ and $\da v\s{0}$, then $v\cdot\n\in H^{-1/2}(\Sigma)$,
$v\cdot\nu\in H^{-1/2}(\Sigma_b)$ ($\nu$ is the outer normal vector
on $\Sigma_b$) and satisfies the estimate
\begin{equation}
\hms{v\cdot\n}{-1/2}{\Sigma}+\hms{v\cdot\nu}{-1/2}{\Sigma_b}\ls(1+\hm{\e}{5/2})^2\bigg(\hm{v}{0}+\hm{\da
v}{0}\bigg)
\end{equation}
\end{lemma}
\begin{proof}
We only need to prove the result on $\Sigma$; the result on
$\Sigma_b$ can be derived in s similar fashion considering
$J\a e_3=\nu$ on $\Sigma_b$.\\
Let $\phi\hs{1/2}$ be a scalar function and let $\tilde{\phi}\in V$
be a bounded extension. We have
\begin{eqnarray*}
&&\int_{\Sigma}\phi v\cdot\n=\int_{\Sigma}J\a_{ij}v_i\phi(e_j\cdot
e_3)=\int_{\Omega}\da(v\tilde\phi)J=\int_{\Omega}\tilde\phi\da v J+
v\cdot\na\tilde\phi J\nonumber\\
&&\ls(1+\hm{\e(t)}{5/2})\bigg(\hm{\tilde\phi}{0}\hm{\da
v}{0}+\hm{v}{0}\hm{\na\tilde\phi}{0}\bigg)\\
&&\ls(1+\hm{\e(t)}{5/2})^2(\hm{v}{0}+\hm{\da
v}{0})\hms{\phi}{1/2}{\Sigma}
\end{eqnarray*}
Since $\phi$ is arbitrary, our result naturally follows.
\end{proof}
\begin{remark}
Recall the space $\y(t)\subset H^0$. It can be shown that if
$v\in\y(t)$, then $\da v=0$ in the weak sense, such that lemma
\ref{boundary negative estimate} implies that $v\cdot\n\hs{-1/2}$
and $v\cdot\nu\in H^{-1/2}(\Sigma_b)$. Moreover, since the elements
of $\y(t)$ are orthogonal to each $\na\phi$ for $\phi\in V(t)$, we
find that $v\cdot\nu=0$ on $\Sigma_b$.
\end{remark}
We want to connect the divergence-free space and
divergence-$\a$-free space, so we define
\begin{equation}\label{divergence preserving matrix}
M=M(t)=K\nabla\Phi= \left(
\begin{array}{lll}
K&0&0\\
0&K&0\\
AK&BK&1
\end{array}
\right)
\end{equation}
Note that $M$ is invertible and $M^{-1}=J\a^T$. Since
$\p_j(J\a_{ij})=0$ for $i=1,2,3$, we have the following relation
\begin{eqnarray*}
p=\da v\Leftrightarrow Jp=J\da
v=J\a_{ij}\p_jv_i=\p_j(J\a_{ij}v_i)=\p_j(J\a^Tv)_j=\p_j(M^{-1}v)_j=\nabla\cdot(M^{-1}v)\qquad
\end{eqnarray*}
Then if $\da v=0$, we have $\nabla\cdot(M^{-1}v)=0$. $M$ induces a
linear operator $\m_t: u\rightarrow\m_t(u)=M(t)u$. It has the
following property.
\begin{lemma}\label{divergence preserving }
For each $t\in[0,T]$, $k=0,1,2$, $\m_t$ is a bounded linear
isomorphism from $H^k(\Omega)$ to $H^k(\Omega)$ and from $X_0$ to
$X$. The bounding constant is given by
$(1+\hms{\e(t)}{9/2}{\Sigma})^2$. If we further define $\m$ defined
by $\m u(t)=\m_tu(t)$, then it is a bounded linear isomorphism from
$L^2([0,T];H^k(\Omega))$ to $L^2([0,T];H^k(\Omega))$ and from $\x_0$
to $\x$. The bounding constant is given by $(1+\sup_{0\leq t\leq
T}\hms{\e(t)}{9/2}{\Sigma})^2$.
\end{lemma}
\begin{proof}
It is easy to see for each $t\in[0,T]$,
\begin{eqnarray*}
\hm{\m_tu}{k}\ls\nm{\m_t}_{C^2}\hm{u}{k}\ls(1+\hms{\e(t)}{9/2}{\Sigma})^2\hm{u}{k}
\end{eqnarray*}
for $k=0,1,2$, which implies $\m_t$ is a bounded operator from $H^k$
to $H^k$. Since $\m_t$ is invertible, we can estimate
$\hm{\m^{-1}v}{k}\ls(1+\hms{\e(t)}{9/2}{\Sigma})^2\hm{v}{k}$. Hence
$\m_t$ is an isomorphism between $H^k$. Also above analysis implies
$\m_t$ maps
divergence-free function to divergence-$\a$-free function. Then it is also an isomorphism from $X_0$ to $X(t)$.\\
A similar argument can justify the case for $L^2([0,T];H^k)$ and
$\x$.
\end{proof}

\subsubsection{Pressure as a Lagrangian Multiplier}

It is well-known that in usual Navier-Stokes equation, pressure can
be taken as a Lagrangian multiplier. In our new settings now,
proposition 2.2.9 in \cite{book1} gives construction of pressure
from transformed equation \ref{transform}, which is valid for small
free surface.
So our result in arbitrary initial data is not completely ready, which we will present here.\\
For $p\in H^0$, we define the functional $S_t\in W^{\ast}$ by
$S_t(v)=\brh{p,\nabla_{\mathcal{A}}\cdot v}$. By the Riesz
representation theorem, there exists a unique $Q_t(p)\in W$ such
that $S_t(v)=\br{Q_t(p),v}_{W}$ for all $v\in W$. This defines a
linear operator $Q_t: H^0\rightarrow W$, which is bounded since we
may take $v=Q_t(p)$ to see
\begin{eqnarray*}
\nm{Q_t(p)}_{W}^2&=&S_t(v)=\brh{p,\nabla_{\mathcal{A}}\cdot
v}\ls(1+\hms{\e(t)}{5/2}{\Sigma})
\hm{p}{0}\hm{\nabla_{\mathcal{A}}\cdot v}{0}\nonumber\\
&\lesssim& (1+\hms{\e(t)}{5/2}{\Sigma})^3\hm{p}{0}\nm{v}_{W} =
(1+\hms{\e(t)}{5/2}{\Sigma})^3\hm{p}{0}\nm{Q_t(p)}_{W}
\end{eqnarray*}
so we have
$\nm{Q_t(p)}_{W}\lesssim(1+\hms{\e(t)}{5/2}{\Sigma})^3\hm{p}{0}$.
Similarly, for $\ss\in\wwst$, we may also define a bounded linear
operator $\q:L^2([0,T];H^0)\rightarrow \mathcal{W}$ via the relation
$\brht{p,\nabla_{\mathcal{A}}\cdot v}=\br{\q(p),v}_{\w}=\ss(v)$ for
all $v\in \mathcal{W}$. Similar argument shows
$\nm{\q(p)}_{\w}\lesssim(1+\sup_{0\leq t\leq
T}\hms{\e(t)}{5/2}{\Sigma})^3\tnm{p}{0}$.
\begin{lemma}\label{lagrange 1}
Let $p\in H^0$, then there exists a $v\in W$ such that
$\nabla_{\mathcal{A}}\cdot v=p$ and
$\nm{v}_{W}\lesssim(1+\hms{\e(t)}{5/2}{\Sigma})^6\hm{p}{0}$. If
instead $p\ts{0}$, then there exists a $v\in \mathcal{W}$ such that
$\nabla_{\mathcal{A}}\cdot v=p$ for a.e. $t$ and
$\nm{v}_{\w}\lesssim(1+\sup_{0\leq t\leq
T}\hms{\e(t)}{5/2}{\Sigma})^6\tnm{p}{0}$.
\end{lemma}
\begin{proof}
In the proof of lemma 3.3 of \cite{book2}, it is established that
for any $q\in L^2(\Omega)$, the problem $\nabla\cdot u=q$ admits a
solution $u\in W$ such that $\hm{u}{1}\ls\hm{q}{0}$. A simple
modification of this proof in infinite case can be applied to
periodic case. Let define $q=Jp$, then
\begin{eqnarray*}
\hm{q}{0}^2=\int_{\Omega}\abs{q}^2=\int_{\Omega}\abs{p}^2J^2\leq
\lnm{J}{\infty}^2\hm{p}{0}^2\ls
(1+\hms{\e(t)}{5/2}{\Sigma})^2\hm{p}{0}^2
\end{eqnarray*}
Hence, we know $v=M(t)u\in W$ satisfies $\na v=p$ and
\begin{eqnarray*}
\nm{v}_{W}^2&\ls&(1+\hms{\e(t)}{9/2}{\Sigma})^4\nm{u}_{W}^2\ls
(1+\hms{\e(t)}{9/2}{\Sigma})^4\hm{q}{0}^2\nonumber\\
&\ls&(1+\hms{\e(t)}{9/2}{\Sigma})^6\hm{p}{0}^2
\end{eqnarray*}
A similar argument may justify the case for $\w$.
\end{proof}
With this lemma in hand, we can show the range of operator $Q_t$ and
$\q$ is closed in $W$ and $\mathcal{W}$.
\begin{lemma}\label{lagrange 2}
$R(Q_t)$ is closed in $W$, and $R(\q)$ is closed in $\mathcal{W}$.
\end{lemma}
\begin{proof}
For $p\in H^0$, let $v\in W$ be the solution of
$\nabla_{\mathcal{A}_0}\cdot v=p$ provided by lemma \ref{lagrange
1}. Then
\begin{displaymath}
\hm{p}{0}^2\ls\brh{p,\nabla_{\mathcal{A}}\cdot
v}=\br{Q_t(p),v}_{W}\ls\nm{Q_t(p)}_{W}\nm{v}_{W}
\lesssim\nm{Q_t(p)}_{W}(1+\hms{\e(t)}{9/2}{\Sigma})^6\hm{p}{0}
\end{displaymath}
such that
\begin{displaymath}
\frac{1}{(1+\hms{\e(t)}{5/2}{\Sigma})^3}\nm{Q_t(p)}_{W}\lesssim\hm{p}{0}\lesssim(1+\hms{\e(t)}{9/2}{\Sigma})^6\nm{Q_t(p)}_{W}
\end{displaymath}
Hence $R(Q_t)$ is closed in $W$. A similar analysis shows that
$R(\q)$ is closed in $\mathcal{W}$.
\end{proof}
Now we can perform a decomposition of $W$ and $\mathcal{W}$.
\begin{lemma}\label{lagrange 3}
We have that $W=X\oplus R(Q_t)$, i.e. $X^{\bot}=R(Q_t)$. Also,
$\mathcal{W}=\mathcal{X}\oplus R(\q)$, i.e.
$\mathcal{X}^{\bot}=R(\q)$.
\end{lemma}
\begin{proof}
By lemma $\ref{lagrange 2}$, $R(Q_t)$ is closed subspace of $W$, so
it suffices to show $R(Q_t)^{\bot}=X$. Let $v\in R(Q_t)^{\bot}$,
then for all $p\in H^0$, we have
\begin{displaymath}
\brh{p,\nabla_{\mathcal{A}}\cdot v}=\br{Q_t(p),v}_{W}=0
\end{displaymath}
and hence $\nabla_{\mathcal{A}}\cdot v=0$, which implies
$R(Q_t)^{\bot}\subseteq X$. On the other hand, suppose $v\in X$,
then $\nabla_{\mathcal{A}}\cdot v=0$ implies
\begin{displaymath}
\br{Q_t(p),v}_{W}=\brh{p,\nabla_{\mathcal{A}}\cdot v}=0
\end{displaymath}
for all $p\in H^0$. Hence $V\in R(Q_t)^{\bot}$ and we see
$X\subseteq R(Q_t)^{\bot}$. So we finish the proof. A similar
argument can show $\mathcal{W}=\mathcal{X}\oplus R(\q)$.
\end{proof}
\begin{proposition}\label{lagrange estimate}
If $\lambda\in W^{\ast}$ such that $\lambda(v)=0$ for all $v\in X$,
then there exists a unique $p(t)\in H^0$ such that
\begin{equation}
\brh{p(t), \nabla_{\mathcal{A}}\cdot v}=\lambda(v)\ for \ all\ v\in
W
\end{equation}
and $\hm{p(t)}{0}\lesssim(1+\hms{\e(t)}{9/2}{\Sigma})^6\|\lambda\|_{W^{\ast}}$.\\
If $\Lambda\in \mathcal{W}^{\ast}$ such that $\Lambda(v)=0$ for all
$v\in \mathcal{X}$, then there exists a unique $p\ts{0}$ such that
\begin{equation}
\brht{p, \nabla_{\mathcal{A}}\cdot v}=\Lambda(v)\ for \ all\ v\in
\mathcal{W}
\end{equation}
and $\tnm{p}{0}\lesssim(1+\sup_{0\leq t\leq
T}\hms{\e(t)}{9/2}{\Sigma})^6\|\Lambda\|_{\mathcal{W}^{\ast}}$
\end{proposition}
\begin{proof}
If $\lambda(v)=0$ for all $v\in X$, then the Riesz representation
theorem yields the existence of a unique $u\in X^{\bot}$ such that
$\lambda(v)=\br{u,v}_{W}$ for all $v\in W$. By lemma $\ref{lagrange
3}$, $u=Q(p)$ for some $p\in H^0$.
Then $\lambda(v)=\br{Q_t(p),v}_{W}=\brh{p(t),\nabla_{\mathcal{A}}\cdot v }$ for all $v\in W$.\\
As for the estimate, by lemma $\ref{lagrange 1}$, we may find $v\in
W$ such that $\nabla_{\mathcal{A}}\cdot v=p$ and
$\nm{v}_{W}\lesssim(1+\hms{\e(t)}{9/2}{\Sigma})^6\hm{p}{0}$. Hence,
\begin{displaymath}
\hm{p}{0}\ls\brh{p, \nabla_{\mathcal{A}}\cdot
v}=\lambda(v)\lesssim\|\lambda\|_{W^{\ast}}(1+\hms{\e(t)}{9/2}{\Sigma})^6\hm{p}{0}
\end{displaymath}
So the desired estimate holds and a similar argument can show the
result for $\Lambda$.
\end{proof}

\subsection{Elliptic Estimates}

In this section, we will study two types of elliptic problems, which
will be employed in deriving the linear estimates. For both
equations, we will present wellposedness theorems to the higher
regularity.

\subsubsection{$\mathcal{A}$-Stokes Equation}

Let us consider the stationary Navier-Stokes problem.
\begin{equation}\label{elliptic equation}
\left\{
\begin{array}{ll}
\da\sa(p,u)=F&in\quad\Omega\\
\da u=G&in\quad\Omega\\
\sa(p,u)\n=H&on\quad\Sigma\\
u=0&on\quad\Sigma_b
\end{array}
\right.
\end{equation}
Since this problem is stationary, we will temporarily ignore the
time dependence of $\e$,$\a$, etc.\\
Before discussing the higher regularity result for this equation, we
should first define the weak formulation. Our method is quite
standard; we will introduce $p$ after first solving a pressureless
problem. Suppose $F\in W^{\ast}$, $G\s{0}$ and $H\s{-1/2}(\Sigma)$.
We say $(u,p)\in W\times H^0$ is a weak solution to $\ref{elliptic
equation}$ if $\da u=G$ a.e. in $\Omega$, and
\begin{equation}\label{elliptic weak solution}
\half\brh{\ma u,\ma v}-\brh{p,\da v}=\brh{F,v}-\brs{H,v}
\end{equation}
for all $v\in W$.
\begin{lemma}\label{elliptic weak statement}
Suppose $F\in W^{\ast}$, $G\s{0}$ and $H\s{-1/2}(\Sigma)$, then
there exists a unique weak solution $(u,p)\in W\times H^0$ to
$\ref{elliptic equation}$.
\end{lemma}
\begin{proof}
By lemma \ref{lagrange 1}, there exists a $\bar u\in W$ such that
$\da\bar u=F^2$. Naturally we can switch the unknowns to $w=u-\bar
u$ such that in the weak formulation $w$ is such that $\da w=0$ and
satisfies
\begin{equation}\label{elliptic pressure weak solution}
\half\brh{\ma w, \ma v}-\brh{p, \da v}=-\half\brh{\ma\bar u, \ma
v}+\brh{F,v}-\brs{H,v}
\end{equation}
for all $v\in W$.\\
To solve this problem, we may restrict our test function to $v\in X$
such that the pressure term vanishes. A direct application of Riesz
representation theorem to the Hilbert space whose inner product is
defined as $\br{u,v}=\brh{\ma u,\ma v}$ provides a unique $w\in X$
such that
\begin{equation}\label{elliptic pressureless weak solution}
\half\brh{\ma w, \ma v}=-\half\brh{\ma\bar u, \ma
v}+\brh{F,v}-\brs{H,v}
\end{equation}
for all $v\in X$.\\
In order to introduce the pressure, we can define $\lambda\in\wst$
as the difference of the left and right hand sides in (\ref{elliptic
pressureless weak solution}). So $\lambda(v)=0$ for all $v\in X$.
Then by proposition \ref{lagrange estimate}, there exists a unique
$p\s{0}$ satisfying $\brh{p, \da v}=\lambda(v)$ for all $v\in W$,
which is equivalent to (\ref{elliptic pressure weak solution}).
\end{proof}
The regularity gain available for solution to (\ref{elliptic
equation}) is limited by the regularity of the coefficients of the
operator $\la$, $\na$ and $\da$, and hence by the regularity of
$\e$. In the next lemma, we will present some preliminary elliptic
estimates.
\begin{lemma}\label{elliptic prelim}
Suppose that $\e\s{k+1/2}(\Sigma)$ for $k\geq3$ such that the map
$\Phi$ defined in (\ref{map}) is a $C^1$ diffeomorphism of $\Omega$
to $\Omega'=\Phi(\Omega)$. If $F\in H^0(\Omega)$, $G\in H^1(\Omega)$
and $H\s{1/2}(\Sigma)$, then the equation (\ref{elliptic equation})
admits a unique strong solution $(u,p)\in H^2\times H^1$, i.e.
$(u,p)$ satisfies (\ref{elliptic equation}) in strong sense.
Moreover, for $r=2,\ldots,k-1$ we have the estimate
\begin{equation}\label{elliptic prelim estimate}
\hm{u}{r}+\hm{p}{r-1}\ls
C(\e)\bigg(\hm{F}{r-2}+\hm{G}{r-1}+\hms{H}{r-3/2}{\Sigma}\bigg)
\end{equation}
whenever the right hand side is finite, where $C(\e)$ is a constant
depending on $\hms{\e}{k+1/2}{\Sigma}$.
\end{lemma}
\begin{proof}
This lemma is exactly the same as lemma 3.6 in \cite{book1}, so we
omit the proof here.
\end{proof}
Notice that the estimate (\ref{elliptic prelim estimate}) can only
go up to $k-1$ order, which does not fully satisfy our requirement.
Hence, in the following we will employ approximating argument to
improve this estimate. For clarity, we divide it into two steps. In
the next lemma, we first prove that the constant can actually only
depend on the initial free surface.
\begin{lemma}\label{elliptic improved}
Let $k\geq6$ be an integer and suppose that $\e\hs{k+1/2}$ and
$\ee\hs{k+1/2}$. Then there exists $\epsilon_0>0$ such that if
$\hms{\e-\ee}{k-3/2}{\Sigma}\leq\epsilon_0$, solution to
\ref{elliptic equation} satisfies
\begin{equation}\label{elliptic improved estimate}
\hm{u}{r}+\hm{p}{r-1}\ls
C(\ee)\bigg(\hm{F}{r-2}+\hm{G}{r-1}+\hms{H}{r-3/2}{\Sigma}\bigg)
\end{equation}
for $r=2,\ldots,k-1$, whenever the right hand side is finite, where
$C(\ee)$ is a constant depending on $\hms{\ee}{k+1/2}{\Sigma}$.
\end{lemma}
\begin{proof}
Based on lemma \ref{elliptic prelim}, we have the estimate
\begin{equation}\label{elliptic temp 1}
\hm{u}{r}+\hm{p}{r-1}\ls
C(\e)\bigg(\hm{F}{r-2}+\hm{G}{r-1}+\hms{H}{r-3/2}{\Sigma}\bigg)
\end{equation}
for $r=2,\ldots,k-1$, whenever the right hand side is finite,  where
$C(\e)$ is a constant depending on $\hms{\e}{k+1/2}{\Sigma}$. Define $\xi=\e-\ee$, then $\xi\hs{k+1/2}$.\\
Let us denote $\a_0$ and $\n_0$ of quantities in terms of $\ee$. We
rewrite the equation \ref{elliptic equation} as a perturbation of
initial status
\begin{equation}
\left\{
\begin{array}{ll}
\nabla_{\a_0}\cdot\ss_{\a_0}(p,u)=F+F^{0}&in\quad\Omega\\
\nabla_{\mathcal{A}_0}\cdot u=G+G^{0}&in\quad\Omega\\
\ss_{\a_0}(p,u)\mathcal{N}_0=H+H^{0}&on\quad\Sigma\\
u=0&on\quad\Sigma_b
\end{array}
\right.
\end{equation}
where
\begin{equation}
\begin{array}{l}
F^{0}=\nabla_{\a_0-\a}\cdot\sa(p,u)+\nabla_{\a_0}\cdot\ss_{\a_0-\a}(p,u)\\
G^{0}=\nabla_{\a_0-\a}\cdot u\\
H^{0}=\ss_{\a_0}(p,u)(\n_0-\n)+\ss_{\a_0-\a}(p,u)
\end{array}
\end{equation}
Suppose that $\hms{\xx}{k-3/2}{\Sigma}\leq1$, which implies
$\hms{\xx}{k-3/2}{\Sigma}^l\leq\hms{\xx}{k-3/2}{\Sigma}<1$ for any
$l>1$. A straightforward calculation reveals that
\begin{equation}
\begin{array}{l}
\hm{F^{0}}{r-2}\leq C(1+\hms{\ee}{k+1/2}{\Sigma})^4\hms{\xx}{k-3/2}{\Sigma}(\hm{u}{r}+\hm{p}{r-1})\\
\\
\hm{G^{0}}{r-1}\leq C(1+\hms{\ee}{k+1/2}{\Sigma})^2\hms{\xx}{k-3/2}{\Sigma}(\hm{u}{r})\\
\\
\hms{H}{r-3/2}{\Sigma}\leq
C(1+\hms{\ee}{k+1/2}{\Sigma})^2\hms{\xx}{k-3/2}{\Sigma}(\hm{u}{r}+\hm{p}{r-1})
\end{array}
\end{equation}
for $r=2,\ldots,k-1$. \\
Since the initial surface function $\ee$ satisfies all the
requirement of lemma \ref{elliptic prelim}, we arrive at the
estimate that for $r=2,\ldots,k-1$
\begin{equation}
\qquad\qquad\hm{u}{r}+\hm{p}{r-1}\ls
C(\ee)\bigg(\hm{F+F^{0}}{r-2}+\hm{G+G^{0}}{r-1}+\hms{H+H^{0}}{r-3/2}{\Sigma}\bigg)
\end{equation}
where $C(\ee)$ is a constant depending on
$\hms{\ee}{k+1/2}{\Sigma}$. Combining all above, we will have
\begin{equation}
\begin{array}{l}
\\
\hm{u}{r}+\hm{p}{r-1}\ls\\
C(\ee)\bigg(\hm{F}{r-2}+\hm{G}{r-1}+\hms{H}{r-3/2}{\Sigma}\bigg)+C(\ee)(1+\hms{\ee}{k+1/2}{\Sigma})^4\hms{\xx}{k-3/2}{\Sigma}(\hm{u}{r}+\hm{p}{r-1})
\end{array}
\end{equation}
So if
\begin{equation}
\hms{\xx}{k-3/2}{\Sigma}\leq\min\bigg\{\half,\frac{1}{4CC(\ee)(1+\hms{\ee}{k+1/2}{\Sigma})^4}\bigg\}
\end{equation}
we can absorb the extra term in right hand side into left hand side
and get a succinct form
\begin{equation}\label{elliptic temp 2}
\hm{u}{r}+\hm{p}{r-1}\ls
C(\ee)\bigg(\hm{F}{r-2}+\hm{G}{r-1}+\hms{H}{r-3/2}{\Sigma}\bigg)
\end{equation}
for $r=2,\ldots,k-1$.
\end{proof}
Note that the above lemma only concerns about regularity up to $k-1$
and we actually need two more order. Then the next result allows us
to achieve this with a bootstrapping argument.
\begin{proposition}\label{elliptic estimate}
Let $k\geq 6$ be an integer, and suppose that $\e\hs{k+1/2}$ as well
as $\ee\hs{k+1/2}$ satisfying
$\hms{\e-\ee}{k+1/2}{\Sigma}\leq\epsilon_0$. Then solution to
\ref{elliptic equation} satisfies
\begin{equation}\label{elliptic final estimate}
\hm{u}{r}+\hm{p}{r-1}\ls
C(\ee)\bigg(\hm{F}{r-2}+\hm{G}{r-1}+\hms{H}{r-3/2}{\Sigma}\bigg)
\end{equation}
for $r=2,\ldots,k+1$, whenever the right hand side is finite, where
$C(\ee)$ is a constant depending on $\hms{\ee}{k+1/2}{\Sigma}$.
\end{proposition}
\begin{proof}
If $r\leq k-1$, then this is just the conclusion of lemma
\ref{elliptic improved}, so our main aim is to gain two more
regularity here for $r=k$ and $r=k+1$. In the following, we first
define an approximate sequence for $\e$. In the case that
$\Sigma=R^2$, we let $\rho\in C_0^{\infty}(R^2)$ be such that
$supp(\rho)\subset B(0,2)$ and $\rho=1$ for $B(0,1)$. For
$m\in\mathbb{N}$, define $\e^m$ by
$\hat\e^m(\xi)=\rho(\xi/m)\hat\e(\xi)$ where $\hat{}$ denotes the
Fourier transform. For each $m$, $\e^m\in H^j(\Sigma)$ for arbitrary
$j\geq0$ and also $\e^m\rightarrow\e$ in $H^{k+1/2}(\Sigma)$ as
$m\rightarrow\infty$. In the periodic case, we define $\hat\e^m$ by
throwing away the higher frequencies: $\hat\e^m(n)=0$ for
$\abs{n}\geq m$. Then $\e^m$ has the same convergence property as
above. Let $\a^m$ and $\n^m$ be defined in terms of $\e^m$.\\
Consider the problem (\ref{elliptic equation}) with $\a$ and $\n$
replaced by $\a^m$ and $\n^m$. Since $\e^m\in H^{k+5/2}$, we can
apply lemma (\ref{elliptic prelim}) to deduce the existence of
$(u^m,p^m)$ that solves
\begin{equation}\label{elliptic temp 3}
\left\{
\begin{array}{ll}
\nabla_{\a^m}\cdot\ss_{\a^m}(p^m,u^m)=F&in\quad\Omega\\
\nabla_{\mathcal{A}}\cdot u^m=G&in\quad\Omega\\
\ss_{\a^m}(p^m,u^m)\mathcal{N}_0=H&on\quad\Sigma\\
u^m=0&on\quad\Sigma_b
\end{array}
\right.
\end{equation}
and such that
\begin{equation}
\begin{array}{l}
\hm{u^m}{r}+\hm{p^m}{r-1}\ls
C(\hms{\e^m}{k+5/2}{\Sigma})\bigg(\hm{F}{r-2}+\hm{G}{r-1}+\hms{H}{r-3/2}{\Sigma}\bigg)
\end{array}
\end{equation}
for $r=2,\ldots,k+1$. We can rewrite above equation in the following
shape, as long as we split the $\dm_{\a^m} u^m$ term.
\begin{equation}
\left\{
\begin{array}{ll}
-\Delta_{\a^m} u^m+\nabla_{\a^m} p^m=F+\nabla_{\a^m} G&in\quad\Omega\\
\nabla_{\a^m}\cdot u^m=G&in\quad\Omega\\
(p^mI-\dm_{\a^m} u^m)\n^m=H&on\quad\Sigma\\
u^m=0&on\quad\Sigma_b
\end{array}
\right.
\end{equation}
In the following, we will prove an improved estimate for $(u^m,p^m)$
in terms of $\hms{\e^m}{k+1/2}{\Sigma}$. We divide the proof into
several steps.
\ \\
\item
Step 1: Preliminaries\\
To abuse the notation, within this bootstrapping procedure, we
always use $(u,p,\e)$ instead of $(u^m, p^m, \e^m)$ to make the
expression succinct, but in fact they should be understood as the
approximate sequence. Also it is easy to see the term
$\nabla_{\a^m}G$ will not affect the shape of estimate because
$\hm{\nabla_{\a^m}G}{k-1}\ls \hms{\e^m}{k+1/2}{\Sigma}\hm{G}{k}$.
Hence, we still write $F$ here to indicate the forcing term in first
equation.
\ \\
We write explicitly each terms in above equations, which will be
hired in the following.
\begin{equation}
\begin{array}{l}
\p_{11}u_1+\p_{22}u_1+(1+A^2+B^2)K^2\p_{33}u_1-2AK\p_{13}u_1-2BK\p_{23}u_1\\\label{com1}
\quad\quad+(AK\p_3(AK)+BK\p_3(BK)-\p_1(AK)-\p_2(BK)+K\p_3K)\p_3u_1+\p_1p-AK\p_3p=F_1
\end{array}
\end{equation}
\begin{equation}
\begin{array}{l}
\p_{11}u_2+\p_{22}u_2+(1+A^2+B^2)K^2\p_{33}u_2-2AK\p_{13}u_2-2BK\p_{23}u_2\\\label{com2}
\quad\quad+(AK\p_3(AK)+BK\p_3(BK)-\p_1(AK)-\p_2(BK)+K\p_3K)\p_3u_2+\p_2p-BK\p_3p=F_2
\end{array}
\end{equation}
\begin{equation}
\begin{array}{l}
\p_{11}u_3+\p_{22}u_3+(1+A^2+B^2)K^2\p_{33}u_3-2AK\p_{13}u_3-2BK\p_{23}u_3\\\label{com3}
\quad\quad+(AK\p_3(AK)+BK\p_3(BK)-\p_1(AK)-\p_2(BK)+K\p_3K)\p_3u_3+K\p_3p=F_3
\end{array}
\end{equation}
\begin{equation}
\begin{array}{l}
\p_{1}u_1-AK\p_{3}u_1+\p_2u_2-BK\p_3u_2+K\p_3u_3=G\label{com4}
\end{array}
\end{equation}
where for all above, $A$, $B$ and $K$ should be understood in terms
of $\e^m$. For convenience, we define
\begin{eqnarray}
\z=C(\ee)P(\e)\bigg(\hm{F}{k-1}^2+\hm{G}{k}^2+\hms{H}{k-1/2}{\Sigma}^2\bigg)
\end{eqnarray}
where $C(\ee)$ is a constant depending on $\hms{\ee}{k+1/2}{\Sigma}$
and $P(\e)$ is the polynomial of $\hms{\e}{k+1/2}{\Sigma}$.
\ \\
\item
Step 2: $r=k$ case\\
For $k-1$ order elliptic estimate, we have
\begin{equation}
\hm{u}{k-1}^2+\hm{p}{k-2}^2\leq
C(\ee)\bigg(\hm{F}{k-3}^2+\hm{G}{k-2}^2+\hms{H}{k-5/2}{\Sigma}^2\bigg)\ls\z
\end{equation}
By lemma \ref{elliptic improved}, the bounding constant $C(\ee)$
only depend on $\hm{\ee}{k+1/2}$. \\
For $i=1,2$, since $(\p_iu, \p_ip)$ satisfies the equation
\begin{equation}
\left\{
\begin{array}{ll}
-\la(\p_iu)+\na (\p_ip)=\bar F&in\quad\Omega\\
\da (\p_iu)=\bar G&in\quad\Omega\\
((\p_ip)I-\ma (\p_iu))\n=\bar H&on\quad\Sigma\\
\p_iu=0&on\quad\Sigma_b
\end{array}
\right.
\end{equation}
where
\begin{displaymath}
\begin{array}{l}
\bar F=\p_iF+\nabla_{\p_i\a}\cdot\nabla_{\a}u+\nabla_{\a}\cdot\nabla_{\p_i\a}u-\nabla_{\p_i\a}p\\
\bar G=\p_iG-\nabla_{\p_i\a}\cdot u\\
\bar H=\p_iH-(p I-\ma u)\p_i\n+\dm_{\p_i\a}u\n
\end{array}
\end{displaymath}
Employing $k-1$ order elliptic estimate, we have
\begin{equation}
\hm{\p_iu}{k-1}^2+\hm{\p_ip}{k-2}^2\ls C(\ee)\bigg(\hm{\bar
F}{k-3}^2+\hm{\bar G}{k-2}^2+\hms{\bar H}{k-5/2}{\Sigma}^2\bigg)
\end{equation}
Since except for the derivatives of $F$, $G$ and $H$, all the other
terms on the right hand side has the form $\hm{A\cdot B}{r}$, in
which $A=\p^{\alpha}\eb$ and $B=\p^{\beta}u$ or $\p^{\beta}p$. These
kinds of estimates can be achieved by lemma \ref{Appendix product}.
Because this lemma will be repeated used in the following estimates,
we will not mention it every time and all of the following estimate
can be derived in the same fashion. Hence, we have the forcing
estimate
\begin{equation}
\begin{array}{ll}
&\hm{\bar F}{k-3}^2+\hm{\bar G}{k-2}^2+\hm{\bar H}{k-5/2}^2\\
&\ls \hm{
F}{k-2}^2+\hm{G}{k-1}^2+\hms{H}{k-3/2}{\Sigma}^2+C(\ee)P(\e)\bigg(\hm{u}{k-1}^2+\hm{p}{k-2}^2\bigg)\\
&\ls\z
\end{array}
\end{equation}
In detail, this means
\begin{equation}
\begin{array}{l}
\hm{\p_1u}{k-1}^2+\hm{\p_1p}{k-2}^2+\hm{\p_2u}{k-1}^2+\hm{\p_2p}{k-2}^2\ls\z
\end{array}
\end{equation}
Hence, most parts in $\hm{u}{k}+\hm{p}{k-1}$ has been covered by
this estimate, except those with highest order derivative of
$\p_3$.\\
Multiplying $A$ to (\ref{com3}) and adding it to (\ref{com1}) will
eliminate the $\p_3p$ term and get
\begin{equation}\label{come4}
\begin{array}{l}
(\p_{11}u_1+A\p_{11}u_3)+(\p_{22}u_1+A\p_{22}u_3)+(1+A^2+B^2)K^2(\p_{33}u_1+A\p_{33}u_3)
-2AK(\p_{13}u_1+A\p_{13}u_3)\nonumber\\
-2BK(\p_{23}u_1+A\p_{23}u_3)
+(AK\p_3(AK)+BK\p_3(BK)-\p_1(AK)-\p_2(BK)+K\p_3K)(\p_{1}u_1+A\p_{1}u_3)\nonumber\\
+\p_1p=F_1+AF_3
\end{array}
\end{equation}
Then taking derivative $\p_3^{k-2}$ on both sides and focus in the
term $(1+A^2+B^2)K^2(\p_3^ku_1+A\p_3^ku_3)$, the estimate of all the
other terms in $H^0$ norm implies that
\begin{equation}
\hm{\p_3^{k}u_1+A\p_3^ku_3}{0}^2\ls\z\label{come1}
\end{equation}
Similarly, we have
\begin{equation}
\hm{\p_3^{k}u_2+B\p_3^ku_3}{0}^2\ls\z\label{come2}
\end{equation}
Rearrange the terms in (\ref{com4}), we get
\begin{equation}
K(1+A^2+B^2)\p_3u_3=G-\p_1u_1-\p_2u_2+AK(\p_3u_1+A\p_3u_3)+BK(\p_3u_2+B\p_3u_3)
\end{equation}
Taking derivative $\p_3^{k-1}$ on both sides, focusing in the term
$K(1+A^2+B^2)\p_3^ku_3$ employing all of the estimate we have known,
we can show
\begin{equation}
\hm{\p_3^ku_3}{0}^2\ls\z\label{come3}
\end{equation}
Combining (\ref{come1}), (\ref{come2}) and (\ref{come3}), it is easy
to see we can get
\begin{equation}
\hm{\p_3^ku_1}{0}^2+\hm{\p_3^ku_2}{0}^2\ls\z
\end{equation}
Plugging this to (\ref{com3}) and taking derivative $\p_3^{k-2}$ on
both sides, we get
\begin{equation}
\hm{\p_3^{k-1}p}{0}^2\ls\z
\end{equation}
Combining this with all above estimate, we have proved
\begin{equation}
\hm{u}{k}^2+\hm{p}{k-1}^2\ls\z
\end{equation}
Therefore, we have proved the case $r=k$.
\ \\
\item
Step 3: $r=k+1$ case: first loop\\
For $i,j=1,2$, since $(\p_{ij}u, \p_{ij}p)$ satisfies the equation
\begin{equation}
\left\{
\begin{array}{ll}
-\la(\p_{ij}u)+\na (\p_{ij}p)=\tilde F&in\quad\Omega\\
\da (\p_{ij}u)=\tilde G&in\quad\Omega\\
((\p_{ij}p)I-\ma (\p_{ij}u))\n=\tilde H&on\quad\Sigma\\
(\p_{ij}u)=0&on\quad\Sigma_b
\end{array}
\right.
\end{equation}
where
\begin{eqnarray*}
\tilde F&=&\p_{ij}F+\nabla_{\p_{ij}\a}\cdot\na
u+\da\nabla_{\p_{ij}\a}u+\nabla_{\p_i\a}\cdot\nabla_{\p_j\a}u+\nabla_{\p_j\a}\cdot\nabla_{\p_i\a}u
+\nabla_{\p_i\a}\cdot\na(\p_ju)\\
&&+\nabla_{\p_j\a}\cdot\na(\p_iu)+\da\nabla_{\p_i\a}(\p_ju)+\da\nabla_{\p_j\a}(\p_iu)-\nabla_{\p_{ij}\a}p-\nabla_{\p_i\a}(\p_jp)-\nabla_{\p_j\a}(\p_ip)\\
\tilde G&=&\p_{ij}G-\nabla_{\p_{ij}\a}\cdot u-\nabla_{\p_i\a}\cdot(\p_ju)-\nabla_{\p_j\a}\cdot(\p_iu)\\
\tilde H&=&\p_{ij}H-(p I-\ma u)\p_{ij}\n-((\p_jp)I-\ma(\p_ju)-\dm_{\p_j\a}u)\p_i\n-((\p_ip)I-\ma(\p_iu)-\dm_{\p_i\a}u)\p_j\n\\
&&+(\dm_{\p_{ij}\a}u+\dm_{\p_i\a}(\p_ju)+\dm_{\p_j\a}(\p_iu))\n
\end{eqnarray*}
Employing $k-1$ order elliptic estimate, we have
\begin{equation}
\begin{array}{ll}
\hm{\p_{ij}u}{k-1}^2+\hm{\p_{ij}p}{k-2}^2&\ls\hm{\tilde
F}{k-3}^2+\hm{\tilde G}{k-2}^2+\hms{\tilde H}{k-5/2}{\Sigma}^2\ls\z
\end{array}
\end{equation}
where the forcing estimate can be taken in a similar argument as in $r=k$ case.\\
In detail, this is actually
\begin{equation}\label{com5}
\begin{array}{l}
\hm{\p_{11}u}{k-1}^2+\hm{\p_{12}u}{k-1}^2+\hm{\p_{22}u}{k-1}^2+\hm{\p_{11}p}{k-2}^2+\hm{\p_{12}p}{k-2}^2+\hm{\p_{22}p}{k-2}^2\ls\z
\end{array}
\end{equation}
\ \\
\item
Step 4: $r=k+1$ case: second loop\\
Similar to $r=k$ case argument, for $i=1,2$, taking derivative
$\p_3^{k-2}\p_i$ on both sides of (\ref{come4}) and focus in the
term $\p_3^k\p_iu_1+A\p_3^k\p_iu_3$, we get
\begin{equation}
\hm{\p_3^{k}\p_iu_1+A\p_3^k\p_iu_3}{0}^2\ls\z
\end{equation}
Similarly, we have
\begin{equation}
\hm{\p_3^{k}\p_iu_2+B\p_3^k\p_iu_3}{0}^2\ls\z
\end{equation}
Plugging in this result to (\ref{com4}) and taking derivative
$\p_3^{k-2}\p_i$ on both sides, we will get
\begin{equation}
\hm{\p_3^k\p_iu_3}{0}^2\ls\z
\end{equation}
An easy estimate for these three terms implies
\begin{equation}
\hm{\p_3^k\p_iu}{0}^2\ls\z
\end{equation}
Plugging this to (\ref{com3}) and taking derivative $\p_3^{k-2}\p_i$
on both sides, we get
\begin{equation}
\hm{\p_3^{k-1}\p_ip}{0}^2\ls\z
\end{equation}
Combining all above estimate with estimate of previous steps, we
have shown
\begin{equation}
\hm{\p_{3i}u}{k-1}^2+\hm{\p_{3i}p}{k-2}^2\ls\z
\end{equation}
In detail, this is actually
\begin{equation}\label{com6}
\begin{array}{l}
\hm{\p_{13}u}{k-1}^2+\hm{\p_{23}u}{k-1}^2+\hm{\p_{13}p}{k-2}^2+\hm{\p_{23}p}{k-2}^2
\ls \z
\end{array}
\end{equation}
\ \\
\item
Step 5: $r=k+1$ case: third loop\\
Again we use the same trick as above. Taking derivative $\p_3^{k-1}$
on both sides of (\ref{come4}) and focusing in the term
$\p_3^{k+1}u_1+A\p_3^{k+1}u_3$, we can bound
$\p_3^{k+1}u_1-A\p_3^{k+1}u_3$. Similarly, we control
$\p_3^{k+1}u_2-B\p_3^{k+1}u_3$. Plugging in this result to
(\ref{com4}) and taking derivative $\p_3^{k-1}$ on both sides, we
can estimate $\p_3^{k+1}u_3$. Then we have
\begin{equation}
\hm{\p_3^{k+1}u}{0}^2\ls\z
\end{equation}
Plugging this to (\ref{com3}) and taking derivative $\p_3^{k-1}$ on
both sides, we get
\begin{equation}
\hm{\p_3^kp}{0}^2\ls\z
\end{equation}
Combining this will all above estimate, we get
\begin{equation}\label{com7}
\hm{\p_{33}u}{k-1}^2+\hm{\p_{33}p}{k-2}^2\ls\z
\end{equation}
\ \\
\item
Step 6: $r=k+1$ case: conclusion\\
To synthesize, (\ref{com5}), (\ref{com6}) and (\ref{com7}) imply
that all the second order derivative of $u$ in $k-1$ norm and $p$ in
$k-2$ norm is controlled, so we naturally have the estimate
\begin{equation}
\hm{u}{k+1}^2+\hm{p}{k}^2\ls C(\ee)P(\e)\bigg(\hm{
F}{k-1}^2+\hm{G}{k}^2+\hm{H}{k-1/2}^2\bigg)
\end{equation}
This is just what we desired. \ \\
Now let us go back to original notation, since $P(\cdot)$ is a fixed
polynomial, this gives an estimate
\begin{equation}\label{elliptic temp 4}
\begin{array}{ll}
\hm{u^m}{k+1}^2+\hm{p^m}{k}^2&\ls C(\ee)P(\e^m)\bigg(\hm{
F}{k-1}^2+\hm{G}{k}^2+\hms{H}{k-1/2}{\Sigma}^2\bigg)\\
&\ls C(\ee)P(\e)\bigg(\hm{
F}{k-1}^2+\hm{G}{k}^2+\hms{H}{k-1/2}{\Sigma}^2\bigg)\\
&\ls C(\ee)P(\ee)\bigg(\hm{
F}{k-1}^2+\hm{G}{k}^2+\hms{H}{k-1/2}{\Sigma}^2\bigg)\\
&\ls
C(\ee)\bigg(\hm{F}{k-1}^2+\hm{G}{k}^2+\hms{H}{k-1/2}{\Sigma}^2\bigg)
\end{array}
\end{equation}
where $C(\ee)$ depends on $\hms{\ee}{k+1/2}{\Sigma}$.\\
This bound implies that the sequence $(u^m,p^m)$ is uniformly
bounded in $H^{k+1}\times H^{k}$, so we can extract weakly
convergent
subsequence $u^m\rightharpoonup u^0$ and $p^m\rightharpoonup p^0$.\\
In the second equation of (\ref{elliptic temp 3}), we multiplied
both sides by $J^mw$ for $w\in C_0^{\infty}$ to see that
\begin{eqnarray}
&&\int_{\Omega}GwJ^m=\int_{\Omega}(\nabla_{\mathcal{A}^m}\cdot
u^m)wJ^m=-\int_{\Omega}u^m\cdot(\nabla_{\mathcal{A}^m}w)J^m\rightarrow
-\int_{\Omega}u^0\cdot(\na w)J=\int_{\Omega}(\da u^0)wJ\nonumber\\
&&
\end{eqnarray}
which implies $\da u^0=G$. Then we multiply the first equation by
$wJ^m$ for $w\in W$ and integrate by parts to see that
\begin{eqnarray}
\half\int_{\Omega}\dm_{\a^m}u^m:\dm_{\a^m}wJ^m-\int_{\Omega}p^m\nabla_{\a^m}\cdot
wJ^m=\int_{\Omega}F\cdot wJ^m-\int_{\Sigma}H\cdot w
\end{eqnarray}
Passing to the limit $m\rightarrow\infty$, we deduce that
\begin{eqnarray}
\half\int_{\Omega}\dm_{\a}u^0:\dm_{\a}w
J-\int_{\Omega}p^0\nabla_{\a}\cdot wJ=\int_{\Omega}F\cdot w
J-\int_{\Sigma}H\cdot w
\end{eqnarray}
which reveals, upon integrating by parts again, $(u^0,p^0)$
satisfies (\ref{elliptic equation}). Since $(u,p)$ is the unique
strong solution to (\ref{elliptic equation}), we have $u^0=u$ and
$p^0=p$. This, weakly lower semi-continuity and estimate
(\ref{elliptic temp 4}) imply (\ref{elliptic final estimate}).
\end{proof}
\begin{remark}
The key part of this proposition is that as long as we can deduce
$\ee\hs{k+1/2}$ and $\e\hs{k+1/2}$, we can achieve $u\s{k+1}$ and
$p\s{k}$, which is the highest possible regularity we can expect.
\end{remark}
\begin{remark}
By our notation rule, since $C(\ee)$ depends on $\Omega_0$ and is
given implicitly, we can take it as a universal constant, so the
estimate may be rewritten as follows.
\begin{equation}\label{elliptic final estimate}
\hm{u}{r}+\hm{p}{r-1}\ls
\hm{F}{r-2}+\hm{G}{r-1}+\hms{H}{r-3/2}{\Sigma}
\end{equation}
for $r=2,\ldots,k+1$
\end{remark}

\subsubsection{$\mathcal{A}$-Poisson Equation}

We consider the elliptic problem
\begin{equation}\label{poisson equation}
\left\{
\begin{array}{ll}
\la p=f&in\quad\Omega\\
p=g&on\quad\Sigma\\
\na p\cdot\nu=h&on\quad\Sigma_b
\end{array}
\right.
\end{equation}
For the weak formulation, we suppose $f\in V^{\ast}$, $g\hs{1/2}$
and $h\in H^{-1/2}(\Sigma_b)$. Let $\bar p\s{1}$ be an extension of
$g$ such that $supp(\bar p)\subset\{-\inf(b)/2<x_3<0\}$. Then we can
switch the unknowns to $q=p-\bar p$. Hence, the weak formulation of
(\ref{poisson equation}) is
\begin{eqnarray}\label{poisson weak solution}
\brh{\na q, \na\phi}=-\brh{\na\bar p,
\na\phi}-\br{f,\phi}_{V^{\ast}}+\br{h,\phi}_{-1/2}
\end{eqnarray}
for all $\phi\in V$, where $\br{\cdot,\cdot}_{V^{\ast}}$ denotes the
dual pairing with $V$ and $\br{\cdot,\cdot}_{-1/2}$ denotes the dual
pairing with $H^{-1/2}(\Sigma_b)$. The existence and uniqueness of a
solution to (\ref{poisson weak solution})
is given by standard argument for elliptic equation.\\
If $f$ has a more specific fashion, we can rewrite the equation to
accommodate the structure of $f$. Suppose the action of $f$ on an
element $\phi\in V$ is given by
\begin{eqnarray}
\br{f,\phi}_{V^{\ast}}=\brh{f_0,\phi}+\brh{F_0,\na\phi}
\end{eqnarray}
where $(f_0,F_0)\in H^0\times H^0$ with
$\hm{f_0}{0}^2+\hm{F_0}{0}^2=\nm{f}_{V^{\ast}}^2$. Then we rewrite
(\ref{poisson weak solution}) into
\begin{eqnarray}
\brh{\na p+F_0, \na\phi}=-\brh{f_0, \phi}+\br{h,\phi}_{-1/2}
\end{eqnarray}
Hence, it is possible to say $p$ is a weak solution to equation
\begin{equation}\label{poisson improved weak solution}
\left\{
\begin{array}{l}
\da (\na p+F_0)= f_0\\
p=g\\
(\na p+F_0)\cdot\nu=h
\end{array}
\right.
\end{equation}
This formulation will be used to construct the higher order initial
conditions in later sections.
\begin{lemma}\label{poisson prelim}
Suppose that $\e\s{k+1/2}(\Sigma)$ for $k\geq3$ such that the map
$\Phi$ is a $C^1$ diffeomorphism of $\Omega$ to
$\Omega'=\Phi(\Omega)$. If $f\s{0}$,$g\s{3/2}$ and $h\s{1/2}$, then
the equation \ref{poisson equation} admits a unique strong solution
$p\s{2}$. Moreover, for $r=2,\ldots,k-1$, we have the estimate
\begin{equation}\label{poisson prelim estimate}
\hm{p}{r}\ls
C(\e)\bigg(\hm{f}{r-2}+\hms{g}{r-1/2}{\Sigma}+\hms{h}{r-3/2}{\Sigma_b}\bigg)
\end{equation}
whenever the right hand side is finite, where $C(\e)$ is a constant
depending on $\hms{\e}{k+1/2}{\Sigma}$.
\end{lemma}
\begin{proof}
This is exactly the same as lemma 3.8 in \cite{book1}, so we omit
the proof here.
\end{proof}
Next, we will prove the bounding constant for the estimate can
actually only depend on the initial surface. Since we do not need
optimal regularity result for this equation, we do not need the
bootstrapping argument now.
\begin{proposition}\label{poisson improved}
Let $k\geq6$ be an integer suppose that $\e\hs{k+1/2}$ and
$\ee\hs{k+1/2}$. There exists $\epsilon_0>0$ such that if
$\hm{\e-\ee}{k-3/2}\leq\epsilon_0$, then solution to \ref{poisson
equation} satisfies
\begin{equation}\label{poisson improved estimate}
\hm{p}{r}\ls
C(\ee)\bigg(\hm{f}{r-2}+\hms{g}{r-1/2}{\Sigma}+\hms{h}{r-3/2}{\Sigma_b}\bigg)
\end{equation}
for $r=2,\ldots,k-1$, whenever the right hand side is finite, where
$C(\ee)$ is a constant depending on $\hms{\ee}{k+1/2}{\Sigma}$.
\end{proposition}
\begin{proof}
The proof is similar to lemma \ref{elliptic improved}. We rewrite
the problem as a perturbation of the initial status.
\begin{equation}
\left\{
\begin{array}{ll}
\Delta_{\a_0} p^m=f+f^m&in\quad\Omega\\
p^m=g+g^m&on\quad\Sigma\\
\nabla_{\a_0} p^m\cdot\nu=h+h^m&on\quad\Sigma_b
\end{array}
\right.
\end{equation}
The constant in this elliptic estimate only depends on the initial
free surface. We may estimate $f^m$, $g^m$ and $h^m$ in terms of
$p^m$, $\ee$ and $\xx^m=\e^m-\ee$. The smallness of $\xx^m$ implies
that we can absorb these terms into left hand side, which is
actually (\ref{poisson improved estimate}).
\end{proof}
\begin{remark}
Similarly, we can also take this constant $C(\ee)$ as a universal
constant and rewrite the estimate as follows.
\begin{equation}\label{poisson improved estimate}
\hm{p}{r}\ls
\hm{f}{r-2}+\hms{g}{r-1/2}{\Sigma}+\hms{h}{r-3/2}{\Sigma_b}
\end{equation}
for $r=2,\ldots,k-1$.
\end{remark}

\subsection{Linear Navier-Stokes Estimates}

\begin{equation}\label{linear}
\left\{
\begin{array}{ll}
\dt u-\la u+\na u=F& in\quad\Omega\\
\da u=0& in\quad\Omega\\
(pI-\ma u)\n=H&on\quad\Sigma\\
u=0&on\quad\Sigma_b\\
\end{array}
\right.
\end{equation}
In this section, we will study the linear Navier-Stokes problem
(\ref{linear}). Following the path of I. Tice and Y. Guo in
\cite{book1}. We will employ two notions of solution: weak and
strong. Then we prove the wellposedness theorem from lower
regularity to higher regularity. To note that, most parts of the
results and proofs here are identical to section 4 of \cite{book1},
so we won's give all the details. The only difference is the
constants related to free surface since our surface has less
restrictions.

\subsubsection{Weak Solution and Strong Solution}

The weak formulation of linear problem (\ref{linear}) is as follows.
\begin{equation}\label{geometric structure}
\brht{\dt u, v}+\half\brht{\ma u, \ma v}-\brht{p, \da v}=\brht{F, v}
-\brhs{H,v}
\end{equation}
\begin{definition}\label{linear weak solution}
Suppose that $u_0\s{0}$, $F\in\wwst$ and $H\ths{-1/2}$. If there
exists a pair $(u,p)$ achieving the initial data $u_0$ and satisfies
$u\in \w$, $p\ts{0}$ and $\dt u\in\wwst$, such that the
(\ref{geometric structure}) holds for any $v\in\w$, we call it a
weak solution.
\end{definition}
If further restricting the test function $v\in\x$, we have a
pressureless weak formulation.
\begin{eqnarray}\label{geometric structure pressureless}
\brht{\dt u, v}+\half\brht{\ma u, \ma v}=\brht{F, v} -\brhs{H,v}
\end{eqnarray}
\begin{definition}\label{linear pressureless weak solution}
Suppose that $u_0\in\y(0)$, $F\in\xxst$ and $H\ths{-1/2}$. If there
exists a function $u$ achieving the initial data $u_0$ and satisfies
$u\in \x$ and $\dt u\in\xxst$, such that the (\ref{geometric
structure pressureless}) holds for any $v\in\x$, we call it a
pressureless weak solution.
\end{definition}
\begin{remark}
It is noticeable that in $\Omega$, $\wwst\subset\xxst$ and
$\nm{u}_{\xxst}\leq\nm{u}_{\wwst}$; on $\Sigma$, for $u\in\x$, we
have $u|_{\Sigma}\ths{1/2}$.
\end{remark}
Since our main aim is to prove higher regularity of the linear
problem, so the weak solution is a natural byproduct of the proof
for strong solution in next subsection. Hence, we will not give the
existence proof here and only present the uniqueness.
\begin{lemma}
Suppose $u$ is a pressureless weak solution of (\ref{linear}). Then
for a.e. $t\in[0,T]$,
\begin{equation}\label{linear temp 1}
\half\int_{\Omega}J\abs{u(t)}^2+\half\int_0^t\int_{\Omega}J\abs{\ma
u}^2=\half\int_{\Omega}J(0)\abs{u_0}^2+\half\int_0^t\int_{\Omega}\abs{u}^2\dt
J+ \int_0^t\int_{\Omega}J(F\cdot u)-\int_0^t\int_{\Sigma}H\cdot u
\end{equation}
Also
\begin{eqnarray}\label{linear temp 2}
\inm{u}{0}^2+\tnm{u}{1}^2\ls
C_0(\e)\bigg(\hm{u_0}{0}^2+\nm{F}_{\xxst}^2+\tnms{H}{-1/2}{\Sigma}^2\bigg)
\end{eqnarray}
where
\begin{eqnarray}
C_0(\e)=P(\hms{\e}{5/2}{\Sigma}+\hms{\dt\e}{5/2}{\Sigma})\exp\bigg(T
P(\hms{\dt\e}{5/2}{\Sigma})\bigg)
\end{eqnarray}
Hence, if a pressureless weak solution exists, then it is unique.
\end{lemma}
\begin{proof}
The proof is almost the same as that of lemma 4.1 and proposition
4.2 in \cite{book1}. The only difference is that we should replace
the lemma 2.1 in \cite{book1} used in that proof by our lemma
\ref{norm equivalence} now.
\end{proof}
Next, we define the strong solution.
\begin{definition}\label{linear strong solution}
Suppose that
\begin{equation}\label{linear strong solution condition}
\begin{array}{ll}
u_0\s{2}\cap \x(0)&\\
F\ts{1}\cap\cs{0}&\dt F\in \xxst\\
H\ths{3/2}\cap \chs{1/2}& \dt H \ths{-1/2}
\end{array}
\end{equation}
If there exists a pair $(u,p)$ achieving the initial data $u_0$ and
satisfies
\begin{equation}
\begin{array}{ll}
u\ts{3}\cap \cs{2}\cap\x&\dt u\ts{1}\cap\cs{0}\\
\dt^2u\in\xxst&p\ts{2}\cap\cs{1}
\end{array}
\end{equation}
such that they satisfies (\ref{linear}) in the strong sense, we call
it a strong solution.
\end{definition}

\subsubsection{Lower Regularity Theorem}

\begin{theorem}\label{linear lower regularity}
Assume the initial data and forcing terms satisfies the condition
(\ref{linear strong solution condition}). Define the divergence-$\a$
preserving operator $D_tu$ by
\begin{equation}\label{linear definition R}
D_tu=\dt u-Ru\quad for\quad R=\dt MM^{-1}
\end{equation}
where $M$ is defined in (\ref{divergence preserving matrix}).
Furthermore, define the orthogonal projection onto the tangent space
of the surface $\{x_3=\ee\}$ according to
\begin{equation}
\Pi_0(v)=v-(v\cdot\n_0)\n_0\abs{\n_0}^{-2}\quad for\quad
\n_0=(-\p_1\eta_0,-\p_2\eta_0,1)
\end{equation}
Suppose the initial data satisfy the compatible condition
\begin{equation}\label{linear lower regularity compatible condition}
\Pi_0(H(0)+\dm_{\a_0}u_0\n_0)=0
\end{equation}
Then the equation (\ref{linear}) admits a unique strong solution
$(u,p)$ which satisfies the estimate
\begin{equation}\label{linear lower regularity estimate}
\begin{array}{l}
\tnm{u}{3}^2+\tnm{\dt
u}{1}^2+\Vert\dt^2u\Vert^2_{\wst}+\inm{u}{2}^2+\inm{\dt
u}{0}^2+\tnm{p}{2}+\inm{p}{1}^2\\
\ls
L(\e)\bigg(\hm{u_0}{2}^2+\hm{F(0)}{0}^2+\hm{H(0)}{1/2}^2+\tnm{F}{1}^2+\Vert\dt
F\Vert^2_{\xxst}+\tnm{H}{3/2}^2+\tnm{\dt H}{-1/2}^2\bigg)
\end{array}
\end{equation}
where
\begin{equation}
L(\e)=P(K(\e))\exp\bigg(T P(K(\e))\bigg)
\end{equation}
for
\begin{eqnarray}
K(\e)=\sup_{0\leq t\leq
T}\bigg(\hms{\e(t)}{9/2}{\Sigma}+\hms{\dt\e(t)}{7/2}{\Sigma}+\hms{\dt^2\e(t)}{5/2}{\Sigma}\bigg)
\end{eqnarray}
The initial pressure $p(0)\s{1}$, is determined in terms of $u_0$,
$\ee$, $F(0)$ and $H(0)$ as the weak solution to
\begin{equation}\label{linear lower regularity initial pressure}
\left\{
\begin{array}{ll}
\nabla_{\a_0}\cdot(\nabla_{\a_0}p(0)-F(0))=-\nabla_{\a_0}(R(0)u_0)&in\quad\Omega\\
p(0)=(H(0)+\dm_{\a_0}u_0\n_0)\cdot\n_0\abs{\n_0}^{-2}&on\quad\Sigma\\
(\nabla_{\a_0}p(0)-F(0))\cdot\nu=\Delta_{\a_0}u_0\cdot\nu&on\quad\Sigma_b
\end{array}
\right.
\end{equation}
in the sense of (\ref{poisson improved weak solution}).\\
Also, $D_tu(0)=\dt u(0)-R(0)u_0$ satisfies
\begin{equation}
D_tu(0)=\Delta_{\a_0}u_0-\nabla_{\a_0}p(0)+F(0)-R(0)u_0\in\y(0).
\end{equation}
Moreover, $(D_tu,\dt p)$ satisfies
\begin{equation}\label{linear lower DT equation}
\left\{
\begin{array}{ll}
\dt(D_tu)-\la(D_tu)+\na(\dt p)=D_tF+G^F&in\quad\Omega\\
\da(D_tu)=0&in\quad\Omega\\
\sa(\dt p, D_tu)\n=\dt H+G^H&on\quad\Sigma\\
D_tu=0&on\quad\Sigma_b
\end{array}
\right.
\end{equation}
in the weak sense of (\ref{geometric structure pressureless}), where
\begin{eqnarray*}
G^F&=&-(R+\dt J K)\la u-\dt R u+(\dt J K+R+R^T)\na p+\da(\ma(R u)-R\ma u+\dm_{\dt\a}u)\nonumber\\
G^H&=&\ma(R u)\n-(p I-\ma u)\dt\n+\dm_{\dt\a}u\n
\end{eqnarray*}
\end{theorem}
\begin{proof}
Since this theorem is almost identical to theorem 4.3 in
\cite{book1}, we will only point out the differences here without
giving details. In order for contrast, we use the same notation and
statement as \cite{book1} here.
\begin{enumerate}
\item
The only difference in the assumption part is that in theorem 4.3 of
\cite{book1}, $K(\e)$ is sufficiently small, but our $K(\e)$ can be
arbitrary. Hence, in \cite{book1}, we can always bound $P(K(\e))\ls
1+K(\e)$, but now we have to keep this term as a polynomial (this is
the only difference of $L(\e)$ between these two theorems).
\item
It is noticeable that each lemma used in proving theorem 4.3 of
\cite{book1} has been recovered or slightly modified in our
preliminary section. To note that these kind of modification merely
involves the bounding constant due to the arbitrary $K(\e)$, but
does not change the conclusion and the shape of estimate at all.
Therefore, all these modification will only contribute to $P(K(\e))$
term.
\item
In elliptic estimate, we now use proposition \ref{elliptic estimate}
to replace proposition 3.7 in \cite{book1}, which gives exactly the
same estimates.
\end{enumerate}
Based on all above, we can easily repeat the proving process in
\cite{book1} with no more work, so our result naturally follows.
\end{proof}

\subsubsection{Initial Data and Compatible Condition}

We first define some quantities related to the initial data:
\begin{eqnarray}
\h_0(u,p)&=&\sum_{j=0}^{N}\hm{\dt^ju(0)}{2N-2j}^2+\sum_{j=0}^{N-1}\hm{\dt^jp(0)}{2N-2j-1}^2\\
\k_0(u_0)&=&\hm{u_0}{2N}^2\\
&&\nonumber\\
\h_0(\e)&=&\hms{\e(0)}{2N+1/2}{\Sigma}^2+\sum_{j=1}^{N}\hms{\dt^j\e(0)}{2N-2j+3/2}{\Sigma}^2\\
\k_0(\ee)&=&\hms{\ee}{2N+1/2}{\Sigma}^2\\
&&\nonumber\\
\h_0(F,H)&=&\sum_{j=0}^{N-1}\hm{\dt^jF(0)}{2N-2j-2}^2+\sum_{j=0}^{N-1}\hms{\dt^jH(0)}{2N-2j-3/2}{\Sigma}^2
\end{eqnarray}
Furthermore, we need to define mappings $\g_1$ on $\Omega$ and
$\g_2$ on $\Sigma$.
\begin{equation}
\begin{array}{l}
\\
\g^1(v,q)=-(R+\dt J K)\la v-\dt Rv+(\dt J K+R+R^T)\na q+\da(\ma(Rv)-R\ma v+\dm_{\dt\a}v)\\
\\
\g^2(v,q)=\ma(Rv)\n-(q I-\ma v)\dt\n+\dm_{\dt\a}v\n
\end{array}
\end{equation}
The mappings above allow us to define the arbitrary order forcing
terms. For $j=1,\ldots,N$, we have
\begin{equation}
\begin{array}{l}
F^0=F\\
H^0=H\\
F^j=D_tF^{j-1}+\g^1(D_t^{j-1}u,\dt^{j-1}p)=D_t^jF+\sum_{l=0}^{j-1}D_t^l\g_1(D_t^{j-l-1}u,\dt^{j-l-1}p)\\
H^j=D_tH^{j-1}+\g^2(D_t^{j-1}u,\dt^{j-1}p)=D_t^jH+\sum_{l=0}^{j-1}D_t^l\g_2(D_t^{j-l-1}u,\dt^{j-l-1}p)
\end{array}
\end{equation}
Finally, we define the general quantities we need to estimate as
follows.
\begin{eqnarray}
\ce(u,p)&=&\sum_{j=0}^{N}\inm{\dt^ju}{2N-2j}^2+\sum_{j=0}^{N-1}\inm{\dt^jp}{2N-2j-1}^2\\
\d(u,p)&=&\sum_{j=0}^{N}\tnm{\dt^ju}{2N-2j+1}^2+\nm{\dt^{N+1}u}_{\xxst}^2+\sum_{j=0}^{N-1}\tnm{\dt^jp}{2N-2j}^2\\
\k(u,p)&=&\ce(u,p)+\d(u,p)\\
\nonumber\\
\ce(\e)&=&\inms{\e}{2N+1/2}{\Sigma}^2+\sum_{j=1}^{N}\inms{\dt^j\e}{2N-2j+3/2}{\Sigma}^2\\
\d(\e)&=&\tnms{\e}{2N+1/2}{\Sigma}^2+\tnms{\dt\e}{2N-1/2}{\Sigma}^2+\sum_{j=2}^{N+1}\tnms{\dt^j\e}{2N-2j+5/2}{\Sigma}^2\\
\k(\e)&=&\ce(\e)+\d(\e)\\
\nonumber\\
\k(F,H)&=&\sum_{j=0}^{N-1}\tnm{\dt^jF}{2N-2j-1}^2+\nm{\dt^NF}_{\xxst}^2+\sum_{j=0}^{N}\tnms{\dt^jH}{2N-2j-1/2}{\Sigma}^2\\
&&+\sum_{j=0}^{N-1}\inm{\dt^jF}{2N-2j-2}^2+\sum_{j=0}^{N-1}\inms{\dt^jH}{2N-2j-3/2}{\Sigma}^2\nonumber
\end{eqnarray}
Before discussing the higher regularity, we first define the
compatible condition of initial data. For $j=0,\ldots,N-1$, we say
that the $j^{th}$ compatible condition is satisfied if
\begin{equation}\label{linear higher regularity compatible condition}
\left\{
\begin{array}{l}
D_t^ju(0)\in H^2\cap X(0)\\
\Pi_0(H^j(0)+\dm_{\a_0}D_t^ju(0)\n_0)=0
\end{array}
\right.
\end{equation}
Then we may construct the initial condition of higher order
satisfying the compatible condition. Since this section is almost
identical to section 4.3 in Y. Guo and I. Tice's paper (see pp 34-35
of \cite{book1}), we omit the details here. The core idea is to
construct the initial data iteratively from lower to higher order.
Then we have the estimate.
\begin{eqnarray}
\h_0(u,p)\ls P(\h_0(\e)) \bigg(\hm{u_0}{2N}^2+\h_0(F,H)\bigg)
\end{eqnarray}

\subsubsection{Higher Regularity Theorem}

\begin{theorem}\label{linear higher regularity}
Suppose the initial data $u_0\s{2N}$, $\ee\hs{2N+1/2}$ and the
initial condition is constructed as above to satisfy the $j^{th}$
compatible condition for $j=0,1,\ldots,N-1$, with the forcing
function $\k(F,H)+\k_0(F,H)<\infty$. Then there exists a universal
constant $T_0(\e)>0$ depending on $\l(\e)$ as defined in the
following, such that if  $0<T<T_0(\e)$, then there exists a unique
strong solution $(u,p)$ to the equation (\ref{linear}) on $[0,T]$
such that
\begin{equation}
\begin{array}{ll}
\dt^ju\ts{2N-2j+1}\cap\cs{2N-2j}& for \quad j=0,\ldots,N\\
\dt^jp\ts{2N-2j}\cap\cs{2N-2j-1}& for \quad j=0,\ldots,N-1\\
\dt^{N+1}u\in\xxst
\end{array}
\end{equation}
Also the pair $(u,p)$ satisfies the estimate
\begin{equation}\label{linear higher regularity estimate}
\k(u,p)\ls\l(\e)\bigg(\hm{u_0}{2N}^2+\h_0(F,H)+\k(F,H)\bigg)
\end{equation}
where
\begin{equation}
\l(\e)=P(\k(\e))\exp\bigg(T P(\k(\e))\bigg)
\end{equation}
Furthermore, the pair $(D_t^ju,\dt^jp)$ satisfies the equation
\begin{equation}\label{linear higher DT equation}
\left\{
\begin{array}{ll}
\dt(D_t^ju)-\la(D_t^ju)+\na(\dt^jp)=F^j&in\quad\Omega\\
\da(D_t^ju)=0&in\quad\Omega\\
\sa(\dt^jp,D_t^ju)\n=H^j&on\quad\Sigma\\
D_t^ju=0&on\quad\Sigma_b
\end{array}
\right.
\end{equation}
in the strong sense for $j=0,\ldots,N-1$ and in the weak sense for
$j=N$.
\end{theorem}
\begin{proof}
Similar to lower regularity case, this theorem is almost identical
to theorem 4.8 in \cite{book1}. Due to the same reason that our free
surface $\e$ can be arbitrary, the only difference here is that we
keep the polynomial of $\k(\e)$ and $\h_0(\e)$ in the final estimate
which cannot be further simplified. Hence, our result easily
follows.
\end{proof}

\subsection{Transport Equation}

In this section, we will prove the wellposedness of the transport
equation
\begin{equation}\label{transport}
\left\{
\begin{array}{ll}
\dt\e+u_1\p_1\e+u_2\p_2\e=u_3&\rm{on}\quad\Sigma\\
\e(t=0)=\ee&\rm{on}\quad\Sigma
\end{array}
\right.
\end{equation}
Throughout this section, we can always assume $u$ is known and
satisfies some bounded properties, which will be specified in the
following.

\subsubsection{Wellposedness of Transport Equation}

Define the quantity
\begin{eqnarray}
\q(u)=\sum_{j=0}^{N}\tnm{\dt^ju}{2N-2j+1}^2+\sum_{j=0}^{N-1}\inm{\dt^ju}{2N-2j}^2
\end{eqnarray}
Obviously, we have the relation $\q(u)\leq\k(u,p)$.

\begin{theorem}\label{transport estimate}
Suppose that $\k_0(\ee)<\infty$ and $\q(u)<\infty$. Then the problem
(\ref{transport}) admits a unique solution $\e$ satisfying
$\k(\e)<\infty$ and achieving the initial data $\dt^j\e(0)$ for
$j=0,\ldots, N$. Moreover, there exists a $0<\bar T(u)<1$, depending
on $\q(u)$, such that if $0<T<\bar T$, then we have the estimate
\begin{eqnarray}
\k(\e)&\ls&P(\k_0(\ee))+P(\q(u))
\end{eqnarray}
where $P(\cdot)$ is a polynomial satisfying $P(0)=0$.
\end{theorem}
\begin{proof}
We divide the proof into several steps.
\ \\
\item
Step 1: Solvability of the equation:\\
Similar to proof of theorem 5.4 in \cite{book1}, based on
proposition 2.1 in \cite{book3}, since $u\ths{2N+1/2}$, then the
equation (\ref{transport}) admits a unique solution $\e\ihs{2N+1/2}$
achieving the initial data $\e(0)=\ee$.
\ \\
\item
Step 2: Estimate of $\ce(\e)$:\\
By lemma \ref{Appendix transport estimate}, we have
\begin{equation}
\inms{\e}{2N+1/2}{\Sigma}\leq\exp\bigg(C\int_0^T\hms{u(t)}{2N+1/2}{\Sigma}dt\bigg)\bigg(\sqrt{\k(\ee)}+\int_0^T\hms{u_3(t)}{2N+1/2}{\Sigma}dt\bigg)
\end{equation}
for $C>0$. The Cauchy-Schwarz inequality implies
\begin{equation}
C\int_0^T\hms{u(t)}{2N+1/2}{\Sigma}dt\ls
C\int_0^T\hm{u(t)}{2N+1}dt\ls\sqrt{T}\sqrt{\q(u)}
\end{equation}
So if $T\leq 1/(2\q(u))$, then
\begin{eqnarray}
\exp\bigg(C\int_0^T\hms{u(t)}{2N+1/2}{\Sigma}dt\bigg)\leq 2
\end{eqnarray}
Hence, we have
\begin{eqnarray}
\inms{\e}{2N+1/2}{\Sigma}^2\ls\k_0(\ee)+T\q(u)\ls \k_0(\ee)+1
\end{eqnarray}
Based on the equation, we further have
\begin{eqnarray}
\inms{\dt\e}{2N-1/2}{\Sigma}^2&\ls& \inms{u_3}{2N-1/2}{\Sigma}^2+\inms{u}{2N-1/2}{\Sigma}^2\inms{\e}{2N+1/2}{\Sigma}^2\\
&\ls&\k_0(\ee)\q(u)+\q(u)+T\q(u)^2\nonumber\\
&\ls&\bigg(\k_0(\ee)+1\bigg)\q(u)\nonumber
\end{eqnarray}
The last inequality is based on the choice of $T$ as above. The
solution $\e$ is temporally differentiable to get the equation
\begin{equation}
\dt(\dt\e)+u_1\p_1(\dt\e)+u_2\p_2(\dt\e)=\dt u_3-\dt u_1\p_1\e-\dt
u_2\p_2\e
\end{equation}
So we have the estimate
\begin{eqnarray}
\quad\inms{\dt^2\e}{2N-5/2}{\Sigma}^2&\ls&
\inms{u}{2N-5/2}{\Sigma}^2\inms{\dt\e}{2N-3/2}{\Sigma}^2
+\inms{\dt u}{2N-5/2}{\Sigma}^2\\
&&+\inms{\dt u}{2N-5/2}{\Sigma}^2\inms{\e}{2N-3/2}{\Sigma}^2\nonumber\\
&\ls&\bigg(\k_0(\ee)+1\bigg)\bigg(\q(u)^2+\q(u)\bigg)\nonumber
\end{eqnarray}
In a similar fashion, we can achieve the estimate for $\dt^j\e$ as
$j=1,\ldots,N$.
\begin{equation}
\inms{\dt^j\e}{2N-2j+3/2}{\Sigma}^2\ls\bigg(\k_0(\ee)+1\bigg)\sum_{i=1}^{j}\q(u)^i
\end{equation}
To sum up, we have
\begin{equation}
\ce(\e)\ls (1+\k_0(\ee))P(\q(u))
\end{equation}
\ \\
\item
Step 3: Estimate of $\d(\e)$:\\
Similar to above argument, by differentiating on both sides and
utilize the transport equation iteratively, we can show
\begin{equation}
\d(\e)\ls(1+\k_0(\ee))P(\q(u))
\end{equation}
Naturally, summarizing all above, we have
\begin{equation}
\k(\e)=\ce(\e)+\d(\e)\ls(1+\k_0(\ee))P(\q(u))
\end{equation}
Then our result easily follows.
\end{proof}
Next, we introduce a lemma to describe the different between $\e$
and $\ee$ in a small time period.
\begin{lemma}\label{transport diff}
If $\q(u)+\k_0(\ee)<\infty$, then for any $\epsilon>0$, there exists
a $\tilde T(\epsilon)>0$ depending on $\q(u)$ and $\k_0(\ee)$, such
that for any $0<T<\tilde T$, we have the estimate
\begin{eqnarray}
\inms{\e-\ee}{2N+1/2}{\Sigma}^2\leq \epsilon
\end{eqnarray}
\end{lemma}
\begin{proof}
$\xi=\e-\ee$ satisfies the transport equation
\begin{equation}\label{transport difference}
\left\{
\begin{array}{ll}
\dt\xi+u_1\p_1\xi+u_2\p_2\xi=u_3-u_1\p_1\ee-u_2\p_2\ee&on\quad\Sigma\\
\xi(0)=0
\end{array}
\right.
\end{equation}
By the transport estimate stated above, we have
\begin{eqnarray}
&&\inms{\xi}{2N+1/2}{\Sigma}\\
&\leq&\exp\bigg(C\int_0^T\hms{u(t)}{2N+1/2}{\Sigma}dt\bigg)
\bigg(\int_0^T\hms{u_3(t)-u_1(t)\p_1\ee-u_2(t)\p_2\ee}{2N+1/2}{\Sigma}dt\bigg)\nonumber\\
&\ls&T\q(u)\k_0(\ee)\nonumber
\end{eqnarray}
Hence, when $T$ is small enough, our result naturally follows.
\end{proof}

\subsubsection{Forcing Estimates}

Now we need to estimate the forcing terms which appears on the
right-hand side of the linear Navier-Stokes equation. The forcing
terms is as follows.
\begin{eqnarray}\label{forcing term}
F&=&\dt\bar\e\tilde{b}K\p_3u-u\cdot\na u\\
H&=&\e\n
\end{eqnarray}
Recall that we define the forcing quantities as follows.
\begin{eqnarray}
\k(F,H)&=&\sum_{j=0}^{N-1}\tnm{\dt^jF}{2N-2j-1}^2+\nm{\dt^NF}_{\xxst}^2+\sum_{j=0}^{N}\tnms{\dt^jH}{2N-2j-1/2}{\Sigma}^2\\
&&+\sum_{j=0}^{N-1}\inm{\dt^jF}{2N-2j-2}^2+\sum_{j=0}^{N-1}\inms{\dt^jH}{2N-2j-3/2}{\Sigma}^2\nonumber\\
\h_0(F,H)&=&\sum_{j=0}^{N-1}\hm{\dt^jF(0)}{2N-2j-2}^2+\sum_{j=0}^{N-1}\hms{\dt^jH(0)}{2N-2j-3/2}{\Sigma}^2
\end{eqnarray}
In the estimate of Navier-Stokes-transport system, we also need some
other forcing quantities.
\begin{eqnarray}
\f(F,H)&=&\sum_{j=0}^{N-1}\tnm{\dt^jF}{2N-2j-1}^2+\tnm{\dt^NF}{0}^2
+\sum_{j=0}^{N}\inms{\dt^jH}{2N-2j-1/2}{\Sigma}^2\\
\h(F,H)&=&\sum_{j=0}^{N-1}\tnm{\dt^jF}{2N-2j-1}^2+\sum_{j=0}^{N-1}\tnms{\dt^jH}{2N-2j-1/2}{\Sigma}^2
\end{eqnarray}
\begin{theorem}\label{forcing estimate}
The forcing terms satisfies the estimate
\begin{eqnarray}
\k(F,H)&\ls&P(\k(\e))+P(\q(u))\label{forcing estimate 1}\\
\h_0(F,H)&\ls&P(\k_0(\ee))+P(\k_0(u_0)\label{forcing estimate 4})\\
\f(F,H)&\ls&P(\k(\e))+P(\q(u))\label{forcing estimate 2}\\
\h(F,H)&\ls&T\bigg(P(\k(\e))+P(\q(u))\bigg)\label{forcing estimate
3}
\end{eqnarray}
where $P(\cdot)$ is a polynomial with $P(0)=0$.
\end{theorem}
\begin{proof}
The estimates follow from simple but lengthy computation, involving
standard argument, so we will only present a sketch of the proof here.\\
After we expand all the terms with Leibniz rule and plug in the
definition of $F^j$ and $H^j$, it is easy to see that the basic form
of estimate is $\tnm{XY}{k}^2$ or $\inm{XY}{k}^2$ where $X$ includes
the terms only involving $\eb$ and $Y$ includes the terms involving
$u$, $p$, $F$ or $H$. Employing lemma \ref{Appendix product}, we can
always bound $\tnm{XY}{k}^2\ls \tnm{X}{k_1}^2\inm{Y}{k_2}^2$ or
$\tnm{XY}{k}^2\ls \inm{X}{k_1}^2\tnm{Y}{k_2}^2$, and also
$\inm{XY}{k}^2\ls \inm{X}{k_1}^2\inm{Y}{k_2}^2$. The principle of
choice in these two forms for $L^2H^k$ is not to exceed the order of
derivative in the definition of right-hand sides in the estimates. A
detailed estimate can show this goal can always be achieved. In
(\ref{forcing estimate 1}), note the trivial bound
\begin{eqnarray}
\nm{\dt^NF}_{\xxst}^2\ls \tnm{\dt^NF}{0}^2
\end{eqnarray}
Also in (\ref{forcing estimate 3}), the key part is the appearance
of bounding constant $T$. We first trivially bound it as follows.
\begin{eqnarray*}
&&\sum_{j=0}^{N-1}\tnm{\dt^jF}{2N-2j-1}^2+\sum_{j=0}^{N-1}\tnms{\dt^jH}{2N-2j-1/2}{\Sigma}^2\\
&&\qquad\qquad\leq
T\bigg(\sum_{j=0}^{N-1}\inm{\dt^jF}{2N-2j-1}^2+\sum_{j=0}^{N-1}\inms{\dt^jH}{2N-2j-1/2}{\Sigma}^2\bigg)\qquad
\end{eqnarray*}
Finally, an application of inequality $2ab\leq a^2+b^2$ will imply
the required estimate.
\end{proof}
\begin{remark}
The reason why we can get the constant $T$ lies in that, in the
nonlinear Navier-Stokes equation, the $u$ in nonlinear terms
actually has one less derivative than the linear part, which makes
it available to use lemma \ref{Appendix connection 2}. The
appearance of $T$ in (\ref{forcing estimate 3}) will play a key role
in the following nonlinear iteration argument.
\end{remark}

\subsection{Navier-Stokes-Transport System}

In this section, we will study the nonlinear system
\begin{equation}
\left\{
\begin{array}{ll}
\partial_tu-\partial_t\bar\eta\tilde{b}K\partial_3u
+u\cdot\nabla_{\mathcal {A}}u-\Delta_{\mathcal {A}}u+\nabla_{\mathcal {A}}p=0 &\rm{in}\quad \Omega\\
\nabla_{\mathcal {A}}\cdot u=0  &\rm{in} \quad\Omega \\
S_{\mathcal {A}}(p,u)\mathcal {N}=\bar\eta\mathcal {N}& \rm{on} \quad\Sigma\\
u=0  &\rm{on}\quad\Sigma_{b}\\
u(x,0)=u_0(x)\\
&\\
\partial_t\eta+u_1\partial_1\eta+u_2\partial_2\eta=u_3 &\rm{on}\quad \Sigma\\
\eta(x',0)=\eta_0(x')
\end{array}
\right.
\end{equation}
We will first show a revised version of construction of the initial
data and then employ an iteration argument to show the wellposedness
of this system.

\subsubsection{Initial Data and Compatible Condition}

Before we turn into the nonlinear equation, we first need to determine the higher order initial condition.\\
Define
\begin{eqnarray}
\k_0=\k_0(u_0)+\k_0(\ee)
\end{eqnarray}
We say $(u_0,\ee)$ satisfies the $N^{th}$ order compatible condition
if
\begin{equation}\label{compatible condition}
\left\{
\begin{array}{ll}
\nabla_{\a_0}\cdot(D_t^ju(0))=0&in\quad\Omega\\
D_t^ju(0)=0&on\quad\Sigma_b\\
\Pi_0(H^j(0)+\dm_{\a_0}D_t^ju(0)\n_0)=0&on\quad\Sigma
\end{array}
\right.
\end{equation}
for $j=0,\ldots,N-1$. Then we can employ an iterative argument to
construct the initial condition of higher order. We will not repeat
the detailed process here (see section 5.2 of \cite{book1}) and only
give an estimate for the constructed data.
\begin{theorem}\label{initial data theorem}
Suppose $(u_0,\ee)$ satisfies $\k_0<\infty$. Let $\dt^ju(0)$,
$\dt^j\e(0)$ for $j=0,\ldots,N$ and $\dt^jp(0)$ for $j=0,\ldots,N-1$
defined as we stated. Then
\begin{eqnarray}
\k_0\leq \h_0(u,p)+\h_0(\e)\ls P(\k_0)
\end{eqnarray}
for a given polynomial $P(\cdot)$ with $P(0)=0$.
\end{theorem}
\begin{proof}
This is a natural corollary of our construction, so we omit the
proof here.
\end{proof}

\subsubsection{Construction of Iteration}

For given initial data $(u_0,\ee)$, by extension lemma \ref{Appendix
extension 1}, there exists $u^0$ defined in $\Omega\times[0,\infty)$
such that $\q(u^0)\ls\k_0(u_0)$ achieving the initial data to
$N^{th}$ order. By solving transport equation with respect to $u^0$,
theorem \ref{transport estimate} implies that there exists $\e^0$
defined in $\Omega\times[0,T_0)$ such that
$\k(\e^0)\ls(1+\k_0(\ee))P(\q(u^0))\ls P(\k_0)<\infty$. This is our
start point.\\
For any integer $m\geq1$, define the approximate solution $(u^m,
p^m, \e^m)$ on the existence interval $[0,T_m)$ by the following
iteration.
\begin{equation}\label{iteration}
\left\{
\begin{array}{ll}
\dt u^{m}-\Delta_{\a^{m-1}}u^m+\nabla_{\a^{m-1}}p^m=\dt\eb^{m-1}\tilde{b}K^{m-1}\p_3u^{m-1}&\rm{in}\quad\Omega\\
\nabla_{\a^{m-1}}\cdot u^m=0&\rm{in}\quad\Omega\\
(p^mI-\dm_{\a^{m-1}}u^m)\n^{m-1}=\e^{m-1}\n^{m-1}&\rm{on}\quad\Sigma\\
u^m=0&\rm{on}\quad\Sigma_b\\
&\\
\dt\e^m+u_1^m\p_1\e^m+u_2^m\p_2\e^m=u_3^m&\rm{on}\quad\Sigma
\end{array}
\right.
\end{equation}
where $(u^m, \e^m)$ achieves the same initial data $(u_0, \ee)$, and
$\a^m$, $\n^m$, $K^m$ are given in terms of $\e^m$.\\
This is only a formal definition of iteration. In the following
theorems, we will finally prove that this approximate sequence can
be defined for any $m\in\mathbb{N}$ and the existence interval $T_m$
will not shrink to $0$ as $m\rightarrow\infty$.

\subsubsection{Boundedness Theorem}

\begin{remark}\label{boundedness remark}
Before we start to prove the boundedness result, it is useful to
notice that, based on the linear estimate (\ref{linear higher
regularity estimate}) and forcing estimate in lemma \ref{forcing
estimate}, this sequence can always be constructed and we can
directly derive an estimate
\begin{equation}
\k(u^{m+1},p^{m+1})+\k(\e^{m+1})\ls P(\q(u^m)+\k(\e^m)+\k_0)
\end{equation}
when $T$ is sufficiently small, where $P(\cdot)$ is a polynomial
satisfying $P(0)=0$. Since the initial data can be arbitrarily
large, this estimate cannot meet our requirement. Hence, we have to
go back to the energy structure and derive a stronger estimate.
However, this result naturally implies a lemma, which will
be used in the following: \\
If
\begin{equation}
\q(u^{m-1})+\k(\e^{m-1})\leq \z
\end{equation}
then
\begin{equation}
\k(u^m,p^m)+\k(\e^m)\leq C P(\k_0+\z)
\end{equation}
\end{remark}
\begin{theorem}\label{boundedness1}
Assume $J^0>\delta>0$ and the initial data $(u_0,\e_0)$ satisfies
the $N^{th}$ compatible condition. Then there exists a constant
$0<\z<\infty$ and $0<\bar T<1$ depending on $\k_0$, such that if
$0<T\leq\bar T$ and $\k_0<\infty$, then there exists an infinite
sequence $(u^m,p^m,\e^m)_{m=0}^{\infty}$ satisfying the iteration
equation (\ref{iteration}) within the existence interval $[0,T)$ and
the following properties.
\begin{enumerate}
\item
The iteration sequence satisfies the estimate
\begin{displaymath}
\q(u^m)+\k(\e^m)\leq \z
\end{displaymath}
for arbitrary $m$, where the temporal norm is taken with respect to
$T$.
\item
For any $m$, $J^m(t)>\delta/2$ for $0\leq t\leq T$.
\end{enumerate}
\end{theorem}
\begin{proof}
Let's denote the above two assertions related to $m$ as statement
$\mathbb{P}_m$. We use induction to prove this theorem. To note that
in the following, we will extensively use the notation $P(\cdot)$,
however, these polynomials should be understood as explicitly given,
but not necessarily written out here. Also they can change from line
to line.
\ \\
\item
Step 1: $\mathbb{P}_0$ case:\\
This is only related to the initial data. Obviously, the
construction of $u^0$ leads to $\q(u^0)\ls \k_0$. By transport
estimate (\ref{transport estimate}), we have $\k(\e^0)\ls P(\k_0)$
for $P(0)=0$. We can choose $\z\geq P(\k_0)$, so the first assertion
is verified.\\
Define $\xx^0=\e^0-\e_0$ the difference between free surface at
later time and its initial data. Then the estimate in lemma
\ref{transport diff} implies $\inm{\xx^0}{5/2}\ls T\z$. Naturally,
$\sup_{t\in[0,T]}\abs{J^0(t)-J^0(0)}\ls \inm{\xx^0}{5/2}\ls T\z$.
Thus if we take $T\leq \delta/(2\z)$, then $J^0(t)\geq\delta/2$. So
the second assertion is also verified. In a similar fashion, we can
verify the closeness assumption we made in linear Navier-Stokes
equation, i.e. $\e$ and $\ee$ is close enough within $[0,T]$ can
also be verified. Hence, $\mathbb{P}_0$ is
true.\\
In the following, we will assume $\mathbb{P}_{m-1}$ is true for
$m\geq1$ and prove $\mathbb{P}_m$ is also true. As long as we can
show this, by induction, $\mathbb{P}_n$ is valid for arbitrary
$n\in\mathbb{N}$. Certainly, the induction hypothesis and above
remark implies $\q(u^{m-1})+\k(\e^{m-1})\leq\z$ and
$\k(u^m,p^m)+\k(\e^m)\leq C P(\k_0+\z)$.
\ \\
\item
Step 2: $\mathbb{P}_m$ case: estimate of $u^m$ via energy structure\\
By theorem \ref{linear higher regularity}, the pair
$(D_t^{N}u^m,\dt^{N}u^m)$ satisfies the equation (\ref{linear higher
DT equation}) in the weak sense, i.e.
\begin{equation}
\left\{
\begin{array}{ll}
\dt(D^{N}_tu^m)-\Delta_{\a^{m-1}}(D^{N}_tu^m)+\nabla_{\a^{m-1}}(\dt^{N}p^m)=F^{N}&in\quad\Omega\\
\nabla_{\a^{m-1}}\cdot(D^{N}tu^m)=0&in\quad\Omega\\
\ss_{\a^{m-1}}(D^{N}_tu^m,\dt^{N}p^m)\n=H^{N}&on\quad\Sigma\\
D^N_tu^m=0&on\quad\Sigma_b
\end{array}
\right.
\end{equation}
where $F^{N}$ and $H^{N}$ is given in terms of $u^{m}$ and
$\e^{m-1}$. Hence, we have the standard weak formulation, i.e. for
any test function $\psi\in \x^{m-1}$, the following holds
\begin{equation}
\brht{\dt D_t^{N}u^m,
\psi}+\half\brht{\dm_{\a^{m-1}}D_t^{N}u^m,\dm_{\a^{m-1}}\psi}=\brht{D_t^{N}u^m,
F^{N}}-\brhs{D_t^{N}u^m,H^{N}}
\end{equation}
Therefore, when we plug in the test function $\psi=D_t^Nu^m$, we
have the natural energy structure
\begin{eqnarray}
\half\int_{\Omega}J\abs{D^N_tu^m}^2+\half\int_0^t\int_{\Omega}J\abs{\dm_{\a^{m-1}}
D_t^Nu^m}^2=\\\half\int_{\Omega}J(0)\abs{D^N_tu^m(0)}^2
+\half\int_0^t\int_{\Omega}\dt J\abs{D_t^Nu^m}^2
+\int_0^t\int_{\Omega}JF^N\cdot D_t^Nu^m
-\int_0^t\int_{\Omega}H^N\cdot D_t^Nu^m\nonumber
\end{eqnarray}
A preliminary estimate is as follows.
\begin{eqnarray*}
LHS&\gtrsim&\inm{D^N_tu^m}{0}^2+\tnm{D^N_tu^m}{1}^2\\
RHS&\lesssim& P(\k_0)+T\z\inm{D^N_tu^m}{0}^2\\
&&+\sqrt{T}\z\tnm{F^N}{0}\inm{D^N_tu^m}{0}+\sqrt{T}\inms{H^N}{-1/2}{\Sigma}\tnms{D^N_tu^m}{1/2}{\Sigma}\\
&\lesssim&P(\k_0)+T\z\inm{D^N_tu^m}{0}^2+\sqrt{T}\tnm{F^N}{0}^2+\sqrt{T}\z^2\inm{D^N_tu^m}{0}^2+\sqrt{T}\inms{H^N}{-1/2}{\Sigma}^2\\
&&+\sqrt{T}\tnm{D^N_tu^m}{1}^2
\end{eqnarray*}
for a polynomial $P(0)=0$. Taking $T\leq1/(16\z^4)$ and absorbing
the extra term on RHS into LHS implies
\begin{equation}
\inm{D^N_tu^m}{0}^2+\tnm{D^N_tu^m}{1}^2\lesssim
P(\k_0)+\sqrt{T}\tnm{F^N}{0}^2+\sqrt{T}\inms{H^N}{-1/2}{\Sigma}^2
\end{equation}
Drop the $\inm{D^N_tu^m}{0}^2$ term and we derive further the
estimate for $\dt^Nu^m$
\begin{equation}
\tnm{\dt^Nu^m}{1}^2\lesssim
P(\k_0)+\sqrt{T}\tnm{F^N}{0}^2+\sqrt{T}\inms{H^N}{-1/2}{\Sigma}^2+\tnm{D_t^Nu^m-\dt^Nu^m}{1}^2
\end{equation}
Then we need to estimate each term on RHS. For the middle two terms,
it suffices to show it is bounded, however, for the last term, we
need a temporal constant $T$ within the estimate, which can be done
by lemma \ref{Appendix connection 2}. Since these estimates is
similar to the proof of lemma \ref{linear temp 0}, we will not give
the details here.
\begin{eqnarray*}
\tnm{F^N}{0}^2&\ls&P(\k(\e^{m-1}))\bigg(\sum_{j=0}^{N-1}\tnm{\dt^ju^m}{2}^2+\sum_{j=0}^{N-1}\tnm{\dt^jp^m}{1}^2\bigg)+\f(F,H)\\
&\ls& P(\k_0+\z)+\f(F,H)
\end{eqnarray*}
\begin{eqnarray*}
\inms{H^N}{-1/2}{\Sigma}^2&\ls&P(\k(\e^{m-1}))\bigg(\sum_{j=0}^{N-1}\inm{\dt^ju^m}{2}^2+\sum_{j=0}^{N-1}\inm{\dt^jp^m}{1}^2\bigg)+\f(F,H)\\
&\ls&P(\k_0+\z)+\f(F,H)
\end{eqnarray*}
\begin{eqnarray*}
\tnm{D_t^Nu^m-\dt^Nu^m}{1}&\ls&T\inm{D_t^Nu^m-\dt^Nu^m}{1}\\
&\ls&T P(\k(\e^{m-1}))\bigg(\sum_{j=0}^{N-1}\inm{\dt^ju^m}{1}^2\bigg)\\
&\ls&T P(\k_0+\z)
\end{eqnarray*}
Therefore, to sum up, we have
\begin{equation}\label{iteration estimate 1}
\tnm{\dt^Nu^m}{1}^2\ls P(\k_0)+\sqrt{T}P(\k_0+\z)+\sqrt{T}\f(F,H)
\end{equation}
\ \\
\item
Step 3: $\mathbb{P}_m$ case: estimate of $u^m$ via elliptic estimate\\
For $0\leq n\leq N-1$, the $n^{th}$ order Navier-Stokes equation is
as follows.
\begin{equation}
\left\{
\begin{array}{ll}
\dt(D^n_tu^m)-\la(D^n_tu^m)+\na(\dt^np^m)=F^n&in\quad\Omega\\
\da(D^n_tu^m)=0&in\quad\Omega\\
\sa(D^n_tu^m,\dt^np^m)\n=H^n&on\quad\Sigma\\
D^n_tu^m=0&on\quad\Sigma_b
\end{array}
\right.
\end{equation}
where $F^n$ and $H^n$ is given in terms of $u^{m}$ and $\e^{m-1}$.
\\
A straightforward application of elliptic estimate reveals the fact.
\begin{equation}
\tnm{D_t^nu^m}{2N-2n+1}^2\lesssim\tnm{\dt
D_t^nu^m}{2N-2n-1}^2+\tnm{F^n}{2N-2n-1}^2+\tnm{H^n}{2N-2n-1/2}^2
\end{equation}
A more reasonable form is as follows.
\begin{eqnarray}
\\
\tnm{\dt^nu^m}{2N-2n+1}^2&\ls&\tnm{F^n}{2N-2n-1}^2+\tnm{H^n}{2N-2n-1/2}^2+\tnm{\dt^{n+1}u^m}{2N-2n-1}^2\nonumber\\
&&+\tnm{\dt(D_t^nu^m-\dt^nu^m)}{2N-2n-1}^2+\tnm{D_t^nu^m-\dt^nu^m}{2N-2n+1}^2\nonumber
\end{eqnarray}
Then we give a detailed estimate for each term on RHS. These
estimates can be easily obtained as what we did before, so we omit
the details here. It is noticeable that the appearance of $T$ is due
to lemma \ref{Appendix connection 2}.
\begin{eqnarray*}
\tnm{F^n}{2N-2n-1}^2
&\ls&T P(\k(\e^{m-1}))\bigg(\sum_{j=0}^{N-2}\inm{\dt^ju^m}{2N-2j-1}^2+\sum_{j=0}^{N-2}\inm{\dt^jp^m}{2N-2j-2}^2\bigg)+\h(F,H)\\
&\ls&T P(\k_0+\z)+\h(F,H)
\end{eqnarray*}
\begin{eqnarray*}
\tnms{H^n}{2N-2n-1/2}{\Sigma}^2&\ls&T
P(\k(\e^{m-1}))\bigg(\sum_{j=0}^{N-2}\inm{\dt^ju^m}{2N-2j-1}^2+\sum_{j=0}^{N-2}\inm{\dt^jp^m}{2N-2j-2}^2\bigg)
+\h(F,H)\\
&\ls&T P(\k_0+\z)+\h(F,H)
\end{eqnarray*}
\begin{eqnarray*}
\tnm{\dt(D_t^nu^m-\dt^nu^m)}{2N-2n-1}^2&\ls&T
P(\k(\e^{m-1}))\bigg(\sum_{j=0}^{N-1}\inm{\dt^ju^m}{2N-2j-1}^2\bigg)\\
&\ls&T P(\k_0+\z)
\end{eqnarray*}
\begin{eqnarray*}
\tnm{D_t^nu^m-\dt^nu^m}{2N-2n+1}^2&\ls&T
P(\k(\e^{m-1}))\bigg(\sum_{j=0}^{N-2}\inm{\dt^ju^m}{2N-2j+1}^2\bigg)\\
&\ls&T P(\k_0+\z)
\end{eqnarray*}
Therefore, to sum up, we have
\begin{equation}\label{iteration estimate 2}
\sum_{n=0}^{N-1}\tnm{\dt^nu^m}{2N-2n+1}^2\ls T P(\k_0+\z)+\h(F,H)
\end{equation}
\ \\
\item
Step 4: $\mathbb{P}_m$ case: synthesis of estimate for $u^m$\\
Combining (\ref{iteration estimate 1}) and (\ref{iteration estimate
2}) with lemma \ref{Appendix connection 1} implies
\begin{equation}
\q(u^m)\ls P(\k_0)+\sqrt{T} P(1+\k_0+\z)+\sqrt{T}\f(F,H)+\h(F,H)
\end{equation}
Then, by forcing estimate (\ref{forcing estimate}), we have
\begin{eqnarray*}
\f(F^0,H^0)&\ls&P(\k(\e^{m-1}))+P(\q(u^{m-1}))\ls P(\z)\\
\h(F^0,H^0)&\ls&T\bigg(P(\k(\e^{m-1}))+P(\q(u^{m-1}))\bigg)\ls T
P(\z)
\end{eqnarray*}
Hence, to sum up, we achieve the estimate
\begin{equation}
\q(u^m)\ls P(\k_0)+\sqrt{T} P(\k_0+\z)
\end{equation}
This is actually
\begin{equation}
\q(u^m)\leq CP(\k_0)+\sqrt{T}C P(\k_0+\z)
\end{equation}
for some universal constant $C>0$. So we can take $\z\geq
2CP(\k_0)$. If $T$ is sufficiently small depending on $\z$, we can
bound $\q(u^m)\ls 2CP(\k_0)\ls\z$.
\ \\
\item
Step 5: $\mathbb{P}_m$ case: estimate of $\e^m$ via transport estimate \\
Employing transport estimate in lemma \ref{transport estimate}
\begin{eqnarray}
\k(\e^m)\ls P(\k_0)+P(\q(u^m))
\end{eqnarray}
for $P(0)=0$. Hence, we can take $\z\geq P(\k_0)+P(2CP(\k_0))$, then
we have
\begin{eqnarray}
\k(\e^m)\ls\z
\end{eqnarray}
\ \\
\item
Step 6: $\mathbb{P}_m$ case: estimate of $J^m(t)$\\
Define $\xx^m=\e^m-\e_0$, then transport estimate in lemma
\ref{transport diff} implies $\inm{\xx^m}{5/2}\ls T\z$. Naturally,
$\sup_{t\in[0,T]}\abs{J^m(t)-J^m(0)}\ls \inm{\xx^m}{5/2}\ls T\z$.
Thus if we take $T\leq \delta/(2\z)$, then $J^m(t)\geq\delta/2$. A
similar argument can justify the closeness assumption of $\e^m$ and
$\ee$ in studying linear Navier-Stokes equation.
\ \\
\item
Synthesis:\\
Above estimates reveals that if we take $\z=CP(\k_0)$ where the
polynomial can be given explicitly by summarizing all above and $T$
small enough depending on $\z$, we have
\begin{equation}
\q(u^m)+\k(\e^m)\leq \z
\end{equation}
and
\begin{equation}
J^m(t)\geq\delta/2\quad for\quad t\in[0,T]
\end{equation}
Therefore, $\mathbb{P}_m$ is verified. By induction, we conclude
that $\mathbb{P}_n$ is valid for any $n\in\mathbb{N}$. It is
noticeable that $\z=P(\k_0)$ where $P(\cdot)$ actually satisfies
$P(0)=0$.

\end{proof}
\begin{theorem}\label{boundedness2}
Assume exactly the same condition as Theorem (\ref{boundedness1}),
then actually we have the estimate
\begin{equation}
\k(u^m,p^m)+\k(\e^m)\ls P(\k_0)
\end{equation}
for a polynomial satisfying $P(0)=0$.
\end{theorem}
\begin{proof}
This is a natural corollary of above theorem. Since we know
$\q(u^{m-1})\ls \z$, by remark \ref{boundedness remark}, it implies
$\k(u^m,p^m)$ is bounded for any $m\in\mathbb{N}$. For convenience,
in the following we also call this bound $\z$.
\end{proof}

\subsubsection{Contraction Theorem}

Define
\begin{equation}
\begin{array}{l}
\nn(v,q;T)=\inm{v}{2}^2+\tnm{v}{3}^2+\inm{\dt v}{0}^2+\tnm{\dt v}{1}^2+\inm{q}{1}^2+\tnm{q}{2}^2\\
\mm(\zeta;T)=\inms{\zz}{5/2}{\Sigma}^2+\inms{\dt\zz}{3/2}{\Sigma}^2+\tnms{\dt^2\zz}{1/2}{\Sigma}^2
\end{array}
\end{equation}
\begin{theorem}\label{contraction}
For $j=1,2$, Suppose that $v^j$, $q^j$, $w^j$ and $\zz^j$ achieve
the same initial condition for different $j$ and satisfy
\begin{equation}
\left\{
\begin{array}{ll}
\dt v^j-\Delta_{\mathcal{A}^j} v^j+\nabla_{\mathcal{A}^j} q^j=\dt\bar\zz^j\tilde{b}K^j\p_3w^j-w^j\cdot\nabla_{\mathcal{A}^j} w^j&in\quad\Omega\\
\nabla_{\mathcal{A}^j}\cdot v^j=0&in\quad\Omega\\
\mathcal{S}_{\mathcal{A}^j}(q^j,v^j)\n^j=\zz^j\n^j&on\quad\Sigma\\
v^j=0&on\quad\Sigma_b\\
 \\
\dt\zz^j=w^j\cdot\n^j&in\quad\Omega
\end{array}
\right.
\end{equation}
where $\a^j$, $\n^j$ and $K^j$ are in terms of $\zz^j$. Suppose
$\k(w^j,0)$, $\k(v^j,q^j)$ and $\k(\zz^j)$ is bounded by $\z$. Then
there exists $0<\tilde T<1$ such that for any $0<T<\tilde T$, we
have the following contraction relation
\begin{eqnarray}
\nn(v^1-v^2,q^1-q^2; T)\leq \half\nn(w^1-w^2,0; T)\\
\mm(\zz^1-\zz^2;T)\ls\nn(w^1-w^2,0; T)
\end{eqnarray}
\end{theorem}

\begin{proof}
We divide this proof into several steps.
\ \\
\item
Step 1: Lower order equations\\
define $v=v^1-v^2$, $q=q^1-q^2$, $w=w^1-w^2$ and $\zz=\zz^1-\zz^2$,
which has trivial initial condition. Then they satisfy the equation
as follows.
\begin{equation}
\left\{
\begin{array}{ll}
\dt v+\nabla_{\mathcal{A}^1}\cdot\mathcal{S}_{\mathcal{A}^1}(q,v)=H^1+\nabla_{\a^1}\cdot(\dm_{(\a^1-\a^2)}v^2)&in\quad\Omega\\
\nabla_{\mathcal{A}_1}\cdot v=H^2&in\quad\Omega\\
\mathcal{S}_{\mathcal{A}_1}(q,v)\n^1=H^3+\dm_{(\a^1-\a^2)}v^2\n^1&on\quad\Sigma\\
v=0&on\quad\Sigma_b
\end{array}
\right.
\end{equation}
where
\begin{equation}
\begin{array}{ll}
H^1=&\nabla_{(\a^1-\a^2)}(\dm_{\a^2}v^2)-(\a^1-\a^2)\nabla p^2
+\dt\zz^1\tilde bK^1\p_3w+\dt\zz\tilde bK^1\p_3w^2\\
&+\dt\zz^1\tilde b(K^1-K^2)\p_3w^2
-w\cdot\nabla_{\a^1}w^1-w^2\cdot\nabla_{\a^1}w-w^2\cdot\nabla_{(\a^1-\a^2)}w^2\\
H^2=&-\nabla_{(\a^1-\a^2)}\cdot v^2\\
H^3=&-q^2(\n^1-\n^2)+\dm_{\a^1}v^2(\n^1-\n^2)-\dm_{(\a^1-\a^2)}v^2(\n^1-\n^2)+\zz\n^1+\zz^2(\n^1-\n^2)
\end{array}
\end{equation}
The solutions are sufficiently regular for us to differentiate in
time, which is the following equation.
\begin{equation}
\left\{
\begin{array}{ll}
\dt(\dt v)+\nabla_{\mathcal{A}^1}\cdot\mathcal{S}_{\mathcal{A}^1}(\dt q,\dt v)=\tilde H^1+\nabla_{\a^1}\cdot(\dm_{(\dt\a^1-\dt\a^2)}v^2)&in\quad\Omega\\
\nabla_{\mathcal{A}_1}\cdot \dt v=\tilde H^2&in\quad\Omega\\
\mathcal{S}_{\mathcal{A}_1}(\dt q,\dt v)\n^1=\tilde H^3+\dm_{(\dt\a^1-\dt\a^2)}v^2\n^1&on\quad\Sigma\\
\dt v=0&on\quad\Sigma_b
\end{array}
\right.
\end{equation}
where
\begin{equation}
\begin{array}{ll}
\tilde H^1=&\dt
H^1+\nabla_{\dt\a^1}\cdot(\dm_{(\a^1-\a^2)}v^2)+\nabla_{\a^1}\cdot(\dm_{(\a^1-\a^2)}\dt
v^2)+\nabla_{\dt\a^1}\cdot(\dm_{\a^1}v)
+\nabla_{\a^1}(\dm_{\dt\a^1}v)-\nabla_{\dt\a^1}q\\
\tilde H^2=&\dt H^2-\nabla_{\dt\a^1}\cdot v\\
\tilde H^3=&\dt H^3+\dm_{(\a^1-\a^2)}\dt
v^2\n^1+\dm_{(\a^1-\a^2)}v^2\dt\n^1-\ss_{\a^1}(q,v)\dt\n^1+\dm_{\dt\a^1}v\n^1
\end{array}
\end{equation}
\ \\
\item
Step 2: Energy evolution for $\dt v$\\
Multiply $J^1\dt v$ on both sides to get the natural energy
structure:
\begin{eqnarray}
\half\int_{\Omega}\abs{\dt
v}^2J^1+\half\int_0^t\int_{\Omega}\abs{\dm_{\a^1}\dt v}^2J^1&=&
\int_0^t\int_{\Omega}J^1\tilde H^1\cdot\dt v
+\int_0^t\int_{\Omega}J^1\tilde H^2\dt q-\int_0^t\int_{\Sigma}\tilde H^3\cdot\dt v\nonumber\\
&+&\half\int_0^t\int_{\Omega}\abs{\dt v}^2\dt
J^1-\half\int_0^t\int_{\Omega}J^1\dm_{(\dt\a^1-\dt\a^2)}v^2:\dm_{\a^1}\dt
v\nonumber
\end{eqnarray}
LHS is simply the energy and dissipation term, so we focus on
estimate of RHS.
\begin{eqnarray*}
\int_0^t\int_{\Omega}J^1\tilde H^1\cdot\dt v&\ls&\z\int_0^t\hm{\tilde H^1}{0}\hm{\dt v}{0}\ls\z\inm{\dt v}{0}\sqrt{T}\tnm{\tilde H^1}{0}\nonumber\\
&\ls&\sqrt{T}\z\sqrt{\nn^v}\tnm{\tilde H^1}{0}
\end{eqnarray*}
\begin{eqnarray*}
\int_0^t\int_{\Omega}J^1\tilde H^2\dt q&\ls&\int_{\Omega}J^1q\tilde H^2-\int_0^t\int_{\Omega}(\dt J^1q\tilde H^2+J^1q\dt\tilde H^2)\nonumber\\
&\ls&\int_{\Omega}J^1q\tilde H^2+\inm{q}{0}\sqrt{T}\z\tnm{\tilde H^2}{0}+\inm{q}{0}\sqrt{T}\z\tnm{\dt\tilde H^2}{0}\nonumber\\
&\ls&\int_{\Omega}J^1q\tilde
H^2+\sqrt{T}\z\sqrt{\nn^v}\bigg(\tnm{\tilde H^2}{0}+\tnm{\dt\tilde
H^2}{0}\bigg)\nonumber
\end{eqnarray*}
\begin{eqnarray*}
\int_0^t\int_{\Sigma}\tilde H^3\cdot\dt v&\ls&\int_0^t\hms{\tilde
H^3}{-1/2}{\Sigma}\hms{\dt v}{1/2}{\Sigma}\ls\sqrt{T}\inms{\tilde
H^3}{-1/2}{\Sigma}
\tnms{\dt v}{1/2}{\Sigma}\nonumber\\
&\ls&\sqrt{T}\sqrt{\nn^v}\inms{\tilde H^3}{-1/2}{\Sigma}
\end{eqnarray*}
\begin{eqnarray*}
\int_0^t\int_{\Omega}\abs{\dt v}^2\dt J^1&\ls&\inm{\dt J^1}{2}T\inm{\dt v}{0}^2\nonumber\\
&\ls&T\z\nn^v
\end{eqnarray*}
So we need several estimate on $\tilde H^i$. For the following
terms, it suffices to show they are bounded. We use the usual way to
estimate these quadratic terms. Since we have repeatedly used this
method, we will not give the details here.
\begin{eqnarray*}
\tnm{\tilde H^1}{0}&\ls&\z\bigg(\tnm{\zz}{3/2}+\tnm{\dt\zz}{1/2}+\tnm{\dt^2\zz}{-1/2}+\tnm{w}{1}+\tnm{\dt w}{1}\\
&&+\tnm{v}{2}+\tnm{q}{1}\bigg)\nonumber\\
&\ls&\z\bigg(\sqrt{\nn^v}+\sqrt{\nn^w}+\sqrt{\mm}\bigg)
\end{eqnarray*}
\begin{eqnarray*}
\tnm{\tilde
H^2}{0}&\ls&\z\bigg(\tnm{\dt\zz}{1/2}+\tnm{\zz}{1/2}+\tnm{v}{1}\bigg)\ls\z\bigg(\sqrt{\nn^v}+\sqrt{\mm}\bigg)
\end{eqnarray*}
\begin{eqnarray*}
\tnm{\dt\tilde H^2}{0}&\ls&\z\bigg(\tnm{\dt^2\zz}{1/2}+\tnm{\dt\zz}{1/2}+\tnm{\zz}{1/2}+\tnm{\dt v}{1}+\tnm{v}{1}\bigg)\nonumber\\
&\ls&\z\bigg(\sqrt{\nn^v}+\sqrt{\mm}\bigg)
\end{eqnarray*}
\begin{eqnarray*}
\inms{\tilde
H^3}{-1/2}{\Sigma}&\ls&\z\bigg(\inm{\zz}{1/2}+\inm{\dt\zz}{1/2}+\inm{q}{1}+\inm{v}{2}\bigg)\nonumber\\
&\ls&\z\bigg(\sqrt{\nn^v}+\sqrt{\mm}\bigg)
\end{eqnarray*}
Also we have the following estimate.
\begin{eqnarray*}
\int_{\Omega}J^1q\tilde H^2&\ls&\z\inm{q}{0}\inm{\tilde H^2}{0}\\
&\ls&\z\inm{q}{0}\bigg(\inm{\zz}{1/2}+\inm{\dt\zz}{1/2}+\inm{v}{1}\bigg)\\
&\ls&\z\sqrt{\nn^v}\bigg(\inm{\zz}{1/2}+\inm{\dt\zz}{1/2}+\inm{v}{1}\bigg)
\end{eqnarray*}
$\zz$ satisfies the equation
\begin{displaymath}
\left\{
\begin{array}{l}
\dt\zz+w_1^1\p_1\zz+w_2^1\p_2\zz=-\n^2\cdot w\\
\zz(0)=0
\end{array}
\right.
\end{displaymath}
Employing transport estimate and the boundedness of higher order
norms, we have
\begin{eqnarray*}
\inm{\zz}{5/2}\ls\sqrt{T}\z\tnm{w}{5/2}\ls\sqrt{T}\z\sqrt{\nn^w}
\end{eqnarray*}
$\dt\zz$ satisfies the equation
\begin{displaymath}
\left\{
\begin{array}{l}
\dt(\dt\zz)+w_1^1\p_1(\dt\zz)+w_2^1\p_2(\dt\zz)=-\n^2\cdot\dt w-\dt\n^2\cdot w-\dt w_1^1\p_1\zz-\dt w_2^1\p_2\zz\\
\dt\zz(0)=0
\end{array}
\right.
\end{displaymath}
Similar argument as above shows that
\begin{eqnarray*}
\inm{\dt\zz}{1/2}\ls\sqrt{T}\z\bigg(\tnm{\dt
w}{1/2}+\tnm{w}{1/2}\bigg)\ls\sqrt{T}\z\sqrt{\nn^w}
\end{eqnarray*}
Combining above transport estimate and lemma \ref{Appendix
connection 1}, we have the final version
\begin{eqnarray*}
\int_{\Omega}J^1q\tilde H^2&\ls&\sqrt{T}\z\sqrt{\nn^v}\sqrt{\nn^w}+\z\sqrt{\nn^v}\inm{v}{1}\\
&\ls&\sqrt{T}\z\sqrt{\nn^v}\sqrt{\nn^w}+\z\sqrt{\nn^v}\bigg(\tnm{v}{2}+\tnm{\dt v}{0}\bigg)\\
&\ls&\sqrt{T}\z\sqrt{\nn^v}\sqrt{\nn^w}+\sqrt{T}\z\sqrt{\nn^v}\bigg(\inm{v}{2}+\inm{\dt v}{0}\bigg)\\
&\ls&\sqrt{T}\z\bigg(\sqrt{\nn^v}\sqrt{\nn^w}+\nn^v\bigg)
\end{eqnarray*}
Then we consider simplify the last term in RHS of energy structure.
\begin{eqnarray*}
\int_0^t\int_{\Omega}J^1\dm_{(\dt\a^1-\dt\a^2)}v^2:\dm_{\a^1}\dt
v&\ls&\frac{1}{4\epsilon}\tnm{J^1\dm_{(\dt\a^1-\dt\a^2)}v^2}{0}^2
+\epsilon\tnm{\dm_{\a^1}\dt v}{0}^2\quad
\end{eqnarray*}
where $\epsilon$ should be determined in order for the second term
in RHS to be absorbed in LHS.
\begin{eqnarray*}
\tnm{J^1\dm_{(\dt\a^1-\dt\a^2)}v^2}{0}^2\ls\sqrt{T}\inm{J^1\dm_{(\dt\a^1-\dt\a^2)}v^2}{0}^2\ls\sqrt{T}\z
\mm
\end{eqnarray*}
To summarize, using Cauchy inequality
\begin{eqnarray}\label{contraction temp 1}
\inm{\dt v}{0}^2+\tnm{\dt v}{1}^2&\ls&\sqrt{T}\z(\nn^w+\nn^v+\mm)
\end{eqnarray}
\ \\
\item
Step 3: Elliptic estimate for $v$\\
based on standard elliptic regularity theory, we have the estimate
\begin{eqnarray*}
\hm{v}{r+2}^2+\hm{q}{r+1}^2&\ls&\hm{\dt
v}{r}^2+\hm{H^1}{r}^2+\hm{H^2}{r+1}^2+\hms{H^3}{r+\half}{\Sigma}^2\nonumber\\
&&+\hm{\nabla_{\a^1}\cdot(\dm_{(\a^1-\a^2)}v^2)}{r}^2
+\hms{\dm_{(\a^1-\a^2)}v^2\n^1}{r+\half}{\Sigma}^2
\end{eqnarray*}
set $r=0$ and take $L^{\infty}$ on both sides
\begin{eqnarray*}
\inm{v}{2}^2+\inm{q}{1}^2&\ls&\inm{\dt
v}{0}^2+\inm{H^1}{0}^2+\inm{H^2}{1}^2+\inms{H^3}{\half}{\Sigma}^2\nonumber\\
&&+\inm{\nabla_{\a^1}\cdot(\dm_{(\a^1-\a^2)}v^2)}{0}^2
+\inms{\dm_{(\a^1-\a^2)}v^2\n^1}{\half}{\Sigma}^2
\end{eqnarray*}
set $r=1$ and take $L^2$ on both sides
\begin{eqnarray*}
\tnm{v}{3}^2+\tnm{q}{2}^2&\ls&\tnm{\dt
v}{1}^2+\tnm{H^1}{1}^2+\tnm{H^2}{2}^2+\tnms{H^3}{\frac{3}{2}}{\Sigma}^2\nonumber\\
&&+\tnm{\nabla_{\a^1}\cdot(\dm_{(\a^1-\a^2)}v^2)}{1}^2
+\tnms{\dm_{(\a^1-\a^2)}v^2\n^1}{\frac{3}{2}}{\Sigma}^2
\end{eqnarray*}
We need to estimate all the RHS terms. We can employ the usual way
to estimate quadratic terms in $L^2H^k$ norm, however, for
$L^{\infty}H^k$ norm, we use lemma \ref{Appendix connection 1}. Then
we have the estimate
\begin{equation}
RHS\ls\tnm{\dt v}{1}^2+\inm{\dt
v}{0}^2+T\z\bigg(\nn^w+\mm\bigg)\nonumber
\end{equation}
To summarize, we have
\begin{eqnarray}\label{contraction temp 2}
\\
\inm{v}{2}^2+\inm{q}{1}^2+\tnm{v}{3}^2+\tnm{q}{2}^2&\ls&\tnm{\dt
v}{1}^2+\inm{\dt v}{0}^2+T\z\bigg(\nn^w+\mm\bigg)\nonumber
\end{eqnarray}
In total of (\ref{contraction temp 1}) and (\ref{contraction temp
2}), we get the succinct form of estimate
\begin{equation}\label{contraction temp 3}
\nn^v\ls \sqrt{T}\z(\nn^w+\nn^v+\mm)
\end{equation}
\ \\
\item
Step 4: Transport estimate for $\zz$\\
$\zz$ satisfies the equation
\begin{displaymath}
\left\{
\begin{array}{l}
\dt\zz+w_1^1\p_1\zz+w_2^1\p_2\zz=-\n^2\cdot w\\
\zz(0)=0
\end{array}
\right.
\end{displaymath}
A straightforward application of transport estimate can show
\begin{eqnarray*}
\inms{\zz}{5/2}{\Sigma}^2&\ls&\exp\bigg(C\int_0^t\hm{w_1}{3}\bigg)\int_0^t\hms{\n^2\cdot w}{5/2}{\Sigma}\\
&\ls&\sqrt{T}\z\tnms{w}{5/2}{\Sigma}^2\ls \sqrt{T}\z\nn^w
\end{eqnarray*}
Similar to the proof of transport estimate in theorem \ref{transport
estimate}, we can use the transport equation to estimate the higher
order derivatives.
\begin{eqnarray*}
\inms{\dt\zz}{3/2}{\Sigma}^2&\ls&\inms{\zz^2}{5/2}{\Sigma}^2\inms{w}{3/2}{\Sigma}^2+\inms{\zz}{5/2}{\Sigma}^2\inms{w_1}{3/2}{\Sigma}^2\ls\z\nn^w
\end{eqnarray*}
In the same fashion, we can easily show
\begin{eqnarray*}
\tnm{\dt^2\zz}{1/2}^2&\ls&\z\nn^w
\end{eqnarray*}
To sum up, we have the estimate
\begin{equation}\label{contraction temp 4}
\mm\ls\nn^w
\end{equation}
\ \\
\item
Synthesis:\\
In (\ref{contraction temp 3}), for $T$ sufficiently small, we can
easily absorbed all $\nn^v$ term from RHS to LHS and replace all
$\mm$ with $\nn^w$ to achieve
\begin{equation}
\nn^v\ls \sqrt{T}\z\nn^w
\end{equation}
Certainly, the smallness of $T$ can guarantee the contraction.

\end{proof}

\subsubsection{Local Wellposedness Theorem}

Now we can combine theorem \ref{boundedness2} and \ref{contraction}
to produce a unique strong solution to the equation system
(\ref{transform}).
\begin{theorem}\label{wellposedness}
Assume that $(u^0,\e^0)$ satisfies $\k_0<\infty$ and that the
initial data are constructed to satisfy the $N^{th}$ compatible
condition. Then there exist $0<T_0<1$ such that if $0<T<T_0$, then
there exists a solution triple $(u,p,\e)$ to the Navier-Stokes
equation on the time interval $[0,T]$ that achieves the initial data
and satisfies
\begin{equation}
\k(u,p)+\k(\e)\leq CP(\k_0)
\end{equation}
for a universal constant $C>0$ and a polynomial $P(\cdot)$
satisfying $P(0)=0$. The solution is unique among functions that
achieve the initial data and satisfy $\k(u,p)+\k(\e)<\infty$.
Moreover, $\e$ is such that the mapping $\Phi$ is a $C^{2N-2}$
diffeomorphism for each $t\in[0,T]$.
\end{theorem}
\begin{proof}
Since we have proved the boundedness and contraction of the
iteration sequence as \cite{book1} did, then a similar argument as
in theorem 6.3 in \cite{book1} can justify our result. Hence, we
will omit the details here. The basic idea is that using boundedness
to derive weak convergence and using the interpolation between lower
order and higher order and the contraction to derive strong
convergence such that we could pass to limit in the iteration
sequence.
\end{proof}

\section{Global Wellposedness for Horizontally Infinite Domain}

\subsection{Preliminaries}

In this section, we will present two types of formulation for the
system and describe the corresponding energy structure. Since they
are almost identical to that in section 2 of \cite{book9}, we will
not give detailed proofs.

\subsubsection{Geometric Structure Form}

Suppose that $\e, u$ are known and $\a$, $\n$, etc. are given in
terms of $\e$ as usual. Then we consider the linear equation for
$(v,q,\xi)$
\begin{equation}\label{geometric structure form equation}
\left\{
\begin{array}{ll}
\partial_tv-\partial_t\bar\eta\tilde{b}K\partial_3v
+u\cdot\nabla_{\mathcal {A}}v+\da\sa(q,v)=F^1 &\rm{in}\quad \Omega\\
\nabla_{\mathcal {A}}\cdot v=F^2  &\rm{in} \quad\Omega \\
S_{\mathcal {A}}(q,v)\mathcal {N}=\xi\mathcal {N}+F^3& \rm{on} \quad\Sigma\\
\partial_t\xi-v\cdot\n=F^4 &\rm{on}\quad
\Sigma\\
v=0  &\rm{on}\quad\Sigma_{b}
\end{array}
\right.
\end{equation}
We have the following energy structure.
\begin{lemma}\label{geometric structure form}
Suppose $(u,p,\e)$ satisfies system (\ref{transform}) and
$(v,q,\xi)$ is the solution of system (\ref{geometric structure form
equation}), then
\begin{eqnarray}
\\
\dt\bigg(\half\int_{\Omega}J\abs{v}^2+\half\int_{\Sigma}\abs{\xi}^2\bigg)+\half\int_{\Omega}J\abs{\ma
v}^2=\int_{\Omega}J(v\cdot F^2+qF^2)+\int_{\Sigma}-v\cdot F^3+\xi
F^4\nonumber
\end{eqnarray}
\end{lemma}
\ \\
The geometric structure form will be utilized to estimate the
temporal derivatives. Hence, we may apply differential operator
$\p^{\alpha}=\dt^{\alpha_0}$ to system (\ref{transform}) with the
resulting equation being (\ref{geometric structure form equation})
for $v=\p^{\alpha}u$, $q=\p^{\alpha}p$ and $\xi=\p^{\alpha}\e$,
where

\begin{eqnarray}
F^1&=&F^{1,1}+F^{1,2}+F^{1,3}+F^{1,4}+F^{1,5}+F^{1,6}\\
F^{1,1}_i&=&\sum_{0<\beta<\alpha}C_{\alpha,\beta}\p^{\beta}(\dt\eb\tilde{b}K)\p^{\alpha-\beta}\p_3u_i
+\sum_{0<\beta\leq\alpha}C_{\alpha,\beta}\p^{\alpha-\beta}\dt\eb\p^{\beta}(\tilde bK)\p_3u_i\\
F^{1,2}_i&=&-\sum_{0<\beta\leq\alpha}C_{\alpha,\beta}\bigg(\p^{\beta}(u_j\a_{jk})\p^{\alpha-\beta}\p_ku_i
+\p^{\beta}\a_{jk}\p^{\alpha-\beta}\p_kp\bigg)\\
F^{1,3}_i&=&\sum_{0<\beta\leq\alpha}C_{\alpha,\beta}\p^{\beta}\a_{jl}\p^{\alpha-\beta}\p_l(\a_{im}\p_mu_j+\a_{jm}\p_mu_i)\\
F^{1,4}_i&=&\sum_{0<\beta<\alpha}C_{\alpha,\beta}\a_{jk}\p_k(\p^{\beta}\a_{il}\p^{\alpha-\beta}\p_lu_j
+\p^{\beta}\a_{jl}\p^{\alpha-\beta}\p_lu_i)\\
F^{1,5}_i&=&\p^{\alpha}\dt\eb\tilde bK\p_3u_i\\
F^{1,6}_i&=&\a_{jk}\p_k(\p^{\alpha}\a_{il}\p_lu_j+\p^{\alpha}\a_{jl}\p_lu_i)
\end{eqnarray}

\begin{eqnarray}
F^2&=&F^{2,1}+F^{2,2}\\
F^{2,1}&=&-\sum_{0<\beta<\alpha}C_{\alpha,\beta}\p^{\beta}\a_{ij}\p^{\alpha-\beta}\p_ju_i\\
F^{2,2}&=&-\p^{\alpha}\a_{ij}\p_ju_i
\end{eqnarray}

\begin{eqnarray}
F^3&=&F^{3,1}+F^{3,2}\\
F^{3,1}&=&\sum_{0<\beta\leq\alpha}C_{\alpha,\beta}\p^{\beta}D\e(\p^{\alpha-\beta}\e-\p^{\alpha-\beta}p)\\
F^{3,2}&=&\sum_{0<\beta\leq\alpha}C_{\alpha,\beta}\bigg(\p^{\beta}(\n_j\a_{im})\p^{\alpha-\beta}\p_mu_j
+\p^{\beta}(\n_j\a_{jm})\p^{\alpha-\beta}\p_mu_i\bigg)
\end{eqnarray}

\begin{equation}
F^4=\sum_{0<\beta\leq\alpha}C_{\alpha,\beta}\p^{\beta}D\e\cdot\p^{\alpha-\beta}u
\end{equation}
In all above, $C_{\alpha,\beta}$ represents constants depending on
$\alpha$ and $\beta$.

\subsubsection{Perturbed Linear Form}

We may also take the system (\ref{transform}) as the perturbation of
the linear system with constant coefficients.
\begin{equation}\label{perturbed linear form equation}
\left\{
\begin{array}{ll}
\dt u-\Delta u+\nabla p=G^1 &\rm{in}\quad \Omega\\
\nabla\cdot u=G^2  &\rm{in} \quad\Omega \\
(pI-\dm u-\e I)e_3=G^3& \rm{on} \quad\Sigma\\
\dt\e-u_3=G^4 &\rm{on}\quad
\Sigma\\
u=0  &\rm{on}\quad\Sigma_{b}
\end{array}
\right.
\end{equation}
This form satisfies the following energy structure.
\begin{lemma}\label{perturbed linear form}
Suppose $(u,p,\e)$ satisfies system (\ref{perturbed linear form
equation}), then
\begin{eqnarray}
\\
\dt\bigg(\half\int_{\Omega}\abs{u}^2+\half\int_{\Sigma}\abs{\e}^2\bigg)+\half\int_{\Omega}J\abs{\dm
u}^2=\int_{\Omega}u\cdot G^2+pG^2+\int_{\Sigma}-u\cdot G^3+\e
G^4\nonumber
\end{eqnarray}
\end{lemma}
\ \\
Compare to the original equation (\ref{transform}), we can write the
detailed form for the nonlinear terms.

\begin{eqnarray}
G^1&=&G^{1,1}+G^{1,2}+G^{1,3}+G^{1,4}+G^{1,5}\\
G^{1,1}_i&=&(\delta_{ij}-\a_{ij})\p_jp\\
G^{1,2}_i&=&u_j\a_{jk}\p_ku_i\\
G^{1,3}_i&=&[K^2(1+A^2+B^2)-1]\p_{33}u_i-2AK\p_{13}u_i-2BK\p_{23}u_i\\
G^{1,4}_i&=&[-K^3(1+A^2+B^2)\p_3J+AK^2(\p_1J+\p_3A)\\
&&+BK^2(\p_2J+\p_3B)-K(\p_1A+\p_2B)]\p_3u_i\nonumber\\
G^{1,5}_i&=&\dt\eb(1+x_3/b)K\p_3u_3
\end{eqnarray}

\begin{equation}
G^2=(1-J)(\p_1u_1+\p_2u_2)+A\p_2u_1+B\p_3u_2
\end{equation}

\begin{equation}
\begin{array}{r}
G^3=\p_1\e\bigg(
\begin{array}{c}
p-\e-2(\p_1u_1-AK\p_3u_1)\\
-\p_2u_1-\p_1u_2+BK\p_3u_1+AK\p_3u_2\\
-\p_1u_3-K\p_3u_1+AK\p_3u_3
\end{array}
\bigg) +\p_2\e\bigg(
\begin{array}{c}
-\p_2u_1-\p_1u_2+BK\p_3u_1+AK\p_3u_2\\
p-\e-2(\p_2u_2-BK\p_3u_2)\\
-\p_2u_3-K\p_3u_2+BK\p_3u_3
\end{array}
\bigg)\\
+\bigg(
\begin{array}{c}
(K-1)\p_3u_1+AK\p_3u_3\\
(K-1)\p_3u_2+BK\p_3u_3\\
2(K-1)\p_3u_3
\end{array}
\bigg)
\end{array}
\end{equation}
where based on the equations we have that
\begin{eqnarray}
\\
p-\e=\p_1\e(-\p_1u_3-K\p_3u_1+AK\p_3u_3)+\p_2\e(-\p_2u_3-K\p_3u_2+BK\p_3u_3)+2K\p_3u_3\nonumber
\end{eqnarray}

\begin{equation}
G^4=-D\e\cdot u
\end{equation}

\subsubsection{Preliminary Estimates}

Here we record several preliminary lemmas in estimating.
\begin{lemma}
There exists a universal $0<\delta<1$ such that if
$\hms{\e}{5/2}{\Sigma}^2\leq\delta$, then
\begin{eqnarray}
\lnm{J-1}{\infty}^2+\lnm{A}{\infty}^2+\lnm{B}{\infty}^2\leq\half
\end{eqnarray}
\end{lemma}
\begin{proof}
The same as lemma 2.4 in \cite{book9}.
\end{proof}

\begin{lemma}
For $i=1,2$, define $U_i: \Sigma\rightarrow\mathbb{R}$ by
\begin{eqnarray}
U_i(x')=\int_{-b}^0J(x',x_3)u_i(x',x_3)\ud{x_3}
\end{eqnarray}
Then $\dt\e=-\p_1U_1-\p_2U_2$ on $\Sigma$.
\end{lemma}
\begin{proof}
The same as lemma 2.5 in \cite{book9}.
\end{proof}

\subsection{Definition of Energy and Dissipation}

In this section we will give the definition of energy and
dissipation with or without minimum count. Compared to those in
\cite{book9} (2.46-2.55), the main difference here lies in the full
dissipation part, where we consider at most two vertical derivatives
for velocity and one vertical derivative for pressure. It is this
key improvement that largely simplifies the whole proof.

\subsubsection{Energy and Dissipation}

We define the energy and dissipation with $2$-minimum count as
follows.
\begin{eqnarray}\label{definition 1}
\bar\ce_{n,2}&=&\hm{\bar D_2^{2n-1}}{0}^2+\hm{D\bar D^{2n-1}}{0}^2+\hm{\sqrt{J}\dt^{n}u}{0}^2+\hms{\bar D_2^{2n}\e}{0}{\Sigma}^2\\
\nonumber\\
\bar\d_{n,2}&=&\hm{\bar D_2^{2n}\dm u}{0}^2\\
\nonumber\\
\ce_{n,2}&=&\hm{D^2u}{2n-2}^2+\sum_{j=1}^n\hm{\dt^j
u}{2n-2j}^2+\hm{D^2p}{2n-3}^2+\sum_{j=1}^{n-1}\hm{\dt^jp}{2n-2j-1}^2\\
&&+\hms{D^2\e}{2n-2}{\Sigma}^2+\sum_{j=1}^n\hms{\dt^j\e}{2n-2j}{\Sigma}^2\nonumber\\
\nonumber\\
\d_{n,2}&=&\hm{\bar D_2^{2n-1}u}{2}^2+\hm{\nabla\bar
D_2^{2n-1}p}{0}^2\\
&&+\hms{D_3^{2n-1}\e}{1/2}{\Sigma}^2+\hms{\dt
D_1^{2n-1}\e}{1/2}{\Sigma}^2+\sum_{j=2}^{n+1}\hms{\dt^jD^{2n-2j+2}_0\e}{1/2}{\Sigma}^2\nonumber
\end{eqnarray}
We define the energy and dissipation without minimum count as
follows, where $\i u$ and $\i\e$ denote the Riesz potential of $u$
and $\e$ as defined in (\ref{appendix riesz potential 1}) and
(\ref{appendix riesz potential 2}).
\begin{eqnarray}
\bar\ce_{n}&=&\hm{\i u}{0}^2+\hms{\i\e}{0}{\Sigma}^2+\hm{\bar D^{2n}_0u}{0}^2+\hms{\bar D^{2n}_0\e}{0}{\Sigma}^2\\
\nonumber\\
\bar\d_{n}&=&\hm{\dm\i u}{0}^2+\hm{\bar D^{2n}_0\dm u}{0}^2\\
\nonumber\\
\ce_{n}&=&\hm{\i
u}{0}^2+\hms{\i\e}{0}{\Sigma}^2+\sum_{j=0}^n\hm{\dt^ju}{2n-2j}^2+\sum_{j=0}^{n-1}\hm{\dt^jp}{2n-2j-1}^2\\
&&+\sum_{j=0}^n\hms{\dt^j\e}{2n-2j}{\Sigma}^2\nonumber
\nonumber\\
\d_{n}&=&\hm{\i u}{1}^2+\hm{\bar D_0^{2n-1}u}{2}^2+\hm{\nabla\bar
D_0^{2n-1}p}{0}^2\\
&&+\hms{D^{2n-1}_0\e}{1/2}{\Sigma}^2+\hms{\dt
D^{2n-2}_0\e}{1/2}{\Sigma}^2+\sum_{j=2}^{n+1}\hms{\dt^jD^{2n-2j+1}_0\e}{1/2}{\Sigma}^2\nonumber
\end{eqnarray}
Finally, we define the energy with $1$-minimum count as follows.
\begin{eqnarray}
\ce_{n,1}&=&\hm{Du}{2n-1}^2+\sum_{j=1}^n\hm{\dt^j
u}{2n-2j}^2+\hm{Dp}{2n-2}^2+\sum_{j=1}^{n-1}\hm{\dt^jp}{2n-2j-1}^2\\
&&+\hms{D\e}{2n-1}{\Sigma}^2+\sum_{j=1}^n\hms{\dt^j\e}{2n-2j}{\Sigma}^2\nonumber
\end{eqnarray}
Since we only consider the decaying for energy and utilize the
interpolation relation between $\ce_{n,2}$ and $\ce_{n,1}$ to
estimate, it is not necessary to define the dissipation with
$1$-minimum count and the horizontal quantities.

\subsubsection{Other Necessary Quantities}

We define some useful quantities in the estimate as follows.
\begin{eqnarray}
\k=\lnm{\nabla
u}{\infty}^2+\lnm{\nabla^2u}{\infty}^2+\sum_{i=1}^2\hms{Du_i}{2}{\Sigma}^2
\end{eqnarray}
\begin{eqnarray}
\f_{2N}=\hms{\e}{4N+1/2}{\Sigma}^2
\end{eqnarray}
The total energy in our estimate is defined as follows.
\begin{eqnarray}\label{definition 2}
\g_N(t)=\sup_{0\leq r\leq t}\ce_{2N}(r)+\sum_{m=1}^2\sup_{0\leq
r\leq t}(1+r)^{m+\lambda}\ce_{N+2,m}(r)+\sup_{0\leq r\leq
t}\frac{\f_{2N}(r)}{(1+r)}
\end{eqnarray}
Note that our definition of total energy is different from that of
(2.58) in \cite{book9}, since we do not include the integral of
dissipation and concentrate on the instantaneous quantities.

\subsubsection{Some Initial Estimates}

Based on remark 2.6-2.8 in \cite{book9}, we naturally have the
following lemmas.
\begin{lemma}\label{interpolation temp 1}
$\bar\ce_{n,2}$ satisfies that
\begin{eqnarray}
\half(\hm{\bar D_2^{2n}u}{0}^2+\hms{\bar
D_2^{2n}\e}{0}{\Sigma}^2)\leq\bar\ce_{n,2}\leq\frac{3}{2}(\hm{\bar
D_2^{2n}u}{0}^2+\hms{\bar D_2^{2n}\e}{0}{\Sigma}^2)
\end{eqnarray}
\end{lemma}
\begin{lemma}\label{interpolation temp 2}
For $N\geq4$, we have $\ce_{N+2,2}\ls\ce_{2N}$ and
$\d_{N+2,2}\ls\ce_{2N}$.
\end{lemma}
\begin{lemma}\label{interpolation temp 3}
We have the estimate
$\hms{\i\dt\e}{0}{\Sigma}^2\ls\hm{u}{0}^2\ls\ce_{2N}$.
\end{lemma}

\subsection{Interpolation Estimates}

\subsubsection{Lower Interpolation Estimates}

In the following, we will provide several interpolation estimates in
the form of
\begin{eqnarray}
\nm{X}^2\ls(\ce_{N+2,2})^{\theta}(\ce_{2N})^{1-\theta}\\
\nm{X}^2\ls(\d_{N+2,2})^{\theta}(\ce_{2N})^{1-\theta}
\end{eqnarray}
where $\theta\in[0,1]$, $X$ is some quantity, and $\nm{\cdot}$ is
some norm.\\
For brevity, we will write the interpolation power in the tables
below. For example, the following part of the table
\begin{equation}
\begin{tabular}{|c|c|c|}
  \hline
  $L^2$ & $\ce_{N+2,2}$& $\d_{N+2,2}$ \\\hline
  $X$ & $\theta$ &$\kappa$\\
  \hline
\end{tabular}
\end{equation}
should be understand as
\begin{eqnarray}
\nm{X}_{L^2}^2\ls(\ce_{N+2,2})^{\theta}(\ce_{2N})^{1-\theta}\\
\nm{X}_{L^2}^2\ls(\d_{N+2,2})^{\kappa}(\ce_{2N})^{1-\kappa}
\end{eqnarray}
When we write $\ce_{N+2,2}\sim\d_{N+2,2}$ in a table, it means that
$\theta$ is the same when interpolating between $\ce_{N+2,2}$ and
$\ce_{2N}$ and between $\d_{N+2,2}$ and $\ce_{2N}$. When we write
multiple entries, it means the same interpolation estimates hold for
each item listed. Sometimes, an interpolation index $r$ appears, we
can always take arbitrary
$r\in(0,1)$. \\
Although most of the estimates are straightforward, some of them
have more complicated structures and needs further computation. The
most common cases are as follows.
\begin{enumerate}
\item
For $0\leq\theta\leq\kappa\leq1$
\begin{eqnarray}
\ce_{N+2,2}^{\theta}\ce_{2N}^{1-\theta}+\ce_{N+2,2}^{\kappa}\ce_{2N}^{1-\kappa}&=&
\ce_{N+2,2}^{\theta}\ce_{2N}^{1-\theta}+\ce_{N+2,2}^{\theta}\ce_{N+2,2}^{\kappa-\theta}\ce_{2N}^{1-\kappa}\\
&\ls&\ce_{N+2,2}^{\theta}\ce_{2N}^{1-\theta}+\ce_{N+2,2}^{\theta}\ce_{2N}^{\kappa-\theta}\ce_{2N}^{1-\kappa}\ls\ce_{N+2,2}^{\theta}\ce_{2N}^{1-\theta}\nonumber
\end{eqnarray}
We may apply the same fashion for $\ce_{N+2,2}$ replaced by
$\d_{N+2,2}$.
\item
For $\theta_1+\theta_2\geq1$
\begin{eqnarray}
\nm{X}^2&\ls&\ce_{N+2,2}^{\theta_1}\ce_{2N}^{1-\theta_1}\ce_{N+2,2}^{\theta_2}\ce_{2N}^{1-\theta_2}
\ls\ce_{N+2,2}\ce_{N+2,2}^{\theta_1+\theta_2-1}\ce_{2N}^{2-\theta_1-\theta_2}\\
&\ls&\ce_{N+2,2}\ce_{2N}\ls\ce_{N+2,2}\nonumber
\end{eqnarray}
where we utilize the bound $\ce_{2N}\leq1$.
\end{enumerate}
With this notation explained, we can state some basic interpolation
estimates now.
\begin{lemma}\label{lower interpolation estimate}
The following tables encodes the power in the $L^{\infty}$ or $L^2$
interpolation estimates for $u$ with its derivatives either in
$\Omega$ or $\Sigma$.
\begin{equation}
\begin{array}{|c|c|}
\hline L^{\infty}&\ce_{N+2,2}\sim\d_{N+2,2}\\\hline u&1/2\\\hline
Du&2/(2+r)\\\hline \nabla u&1/2\\\hline D\nabla u&2/(2+r)\\\hline
\nabla^2u&1/2\\\hline D\nabla^2u&2/(2+r)\\\hline
\end{array}
\quad
\begin{array}{|c|c|}
\hline L^{2}&\ce_{N+2,2}\sim\d_{N+2,2}\\\hline
u&\lambda/(\lambda+2)\\\hline Du&(\lambda+1)/(\lambda+2)\\\hline
\nabla u&\lambda/(\lambda+2)\\\hline D\nabla
u&(\lambda+1)/(\lambda+2)\\\hline
\nabla^2u&\lambda/(\lambda+2)\\\hline
D\nabla^2u&(\lambda+1)/(\lambda+2)\\\hline
\end{array}
\end{equation}
The following tables encodes the power in the $L^{\infty}$ or $L^2$
interpolation estimates for $\e$ and $\eb$ with their derivatives
either in $\Omega$ or $\Sigma$.
\begin{equation}
\begin{array}{c}
\begin{array}{|c|c|c|}
\hline L^{\infty}&\ce_{N+2,2}&\d_{N+2,2}\\\hline
\e,\eb&(\lambda+1)/(\lambda+2)&(\lambda+1)/(\lambda+3)\\\hline
D\e,\nabla\eb&(\lambda+2)/(\lambda+2+r)&(\lambda+2)/(\lambda+3)\\\hline
D^2\e,\nabla^2\eb&1&(\lambda+3)/(\lambda+3+r)\\\hline
\dt\e,\dt\eb&1&2/(2+r)\\\hline
\end{array}
\\
\\
\begin{array}{|c|c|c|} \hline
L^{2}&\ce_{N+2,2}&\d_{N+2,2}\\\hline
\e,\eb&\lambda/(\lambda+2)&\lambda/(\lambda+3)\\\hline
D\e,\nabla\eb&(\lambda+1)/(\lambda+2)&(\lambda+1)/(\lambda+3)\\\hline
D^2\e,\nabla^2\eb&1&(\lambda+2)/(\lambda+3)\\\hline
\dt\e,\dt\eb&1&1/2\\\hline
\end{array}
\end{array}
\end{equation}
The following tables encodes the power in the $L^{\infty}$ or $L^2$
interpolation estimates for $p$ with its derivatives either in
$\Omega$ or $\Sigma$.
\begin{equation}
\begin{array}{|c|c|c|}
\hline L^{\infty}&\ce_{N+2,2}&\d_{N+2,2}\\\hline \nabla
p&1/2&1/2\\\hline Dp&2/(2+r)&2/3\\\hline D\nabla
p&2/(2+r)&2/(2+r)\\\hline D^2p&1&2/(2+r)\\\hline
\end{array}
\quad
\begin{array}{|c|c|c|}
\hline L^{2}&\ce_{N+2,2}&\d_{N+2,2}\\\hline \nabla p&0&0\\\hline
Dp&1/2&1/3\\\hline D\nabla p&1/2&1/2\\\hline D^2p&1&2/3\\\hline
\end{array}
\end{equation}
The following tables encodes the power in the $L^{\infty}$ or $L^2$
interpolation estimates for $G^i$.
\begin{equation}
\begin{array}{c}
\begin{array}{|c|c|c|}
\hline L^{\infty}&\ce_{N+2,2}&\d_{N+2,2}\\\hline
G^1&1&(3\lambda+5)/(2\lambda+6)\\\hline DG^1&1&1\\\hline
G^2&1&1\\\hline DG^2&1&1\\\hline \nabla G^2&1&1\\\hline
G^3&1&(3\lambda+5)/(2\lambda+6)\\\hline DG^3&1&1\\\hline
G^4&1&1\\\hline
\end{array}
\\
\\
\begin{array}{|c|c|c|}
\hline L^{2}&\ce_{N+2,2}&\d_{N+2,2}\\\hline
G^1&(3\lambda+2)/(2\lambda+4))&(3\lambda+3)/(2\lambda+6)\\\hline
DG^1&1&(3\lambda+5)/(2\lambda+6)\\\hline G^2&1&m_2\\\hline
DG^2&1&1\\\hline \nabla G^2&1&m_2\\\hline
G^3&(2\lambda+1)/(\lambda+2)&m_1\\\hline DG^3&1&m_2\\\hline
D^2G^3&1&1\\\hline G^4&1&m_2\\\hline
\end{array}
\end{array}
\end{equation}
for
\begin{eqnarray*}
m_1=\min\{(2\lambda^2+6\lambda+2)/(\lambda^2+5\lambda+6),(3\lambda+3)/(2\lambda+6)\}\\
m_2=\min\{(2\lambda^2+7\lambda+4)/(\lambda^2+5\lambda+6),(3\lambda+5)/(2\lambda+6)\}
\end{eqnarray*}
\end{lemma}
\begin{proof}
The estimate follows directly from the Sobolev embedding and lemma
\ref{appendix interpolation 2}-\ref{appendix interpolation 4}, using
the bounds $\hm{\i\e}{0}^2\ls\ce_{2N}$ and
$\hm{\i\dt\e}{0}^2\ls\ce_{2N}$ which is a natural corollary of lemma
\ref{interpolation temp 3}. For the terms of $G^i$, we may utilize
the estimating method for complicated structures stated at the
beginning of this section and the natural product rule
\begin{eqnarray}
\lnm{XY}{\infty}^2&\ls&\lnm{X}{\infty}^2\lnm{Y}{\infty}^2\\
\lnm{XY}{2}^2&\ls&\lnm{X}{\infty}^2\lnm{Y}{2}^2\label{interpolation temp 4}\\
\lnm{XY}{2}^2&\ls&\lnm{X}{2}^2\lnm{Y}{\infty}^2 \label{interpolation
temp 5}
\end{eqnarray}
where in $L^2$ bound, we may take the larger value of $\theta$
produced by the two options.
\end{proof}
\ \\
Then we can show several important lemmas which greatly improve the
estimates above and will play the key role in the a priori
estimates.
\begin{lemma}\label{lower interpolation estimate improved 1}
We have the estimate
\begin{eqnarray}
\k&\ls&\ce_{2N}^{r/(2+r)}\ce_{N+2}^{2/(2+r)}
\end{eqnarray}
\end{lemma}
\begin{proof}
\begin{eqnarray}
\k=\lnm{\nabla
u}{\infty}^2+\lnm{\nabla^2u}{\infty}^2+\sum_{i=1}^2\hms{Du_i}{2}{\Sigma}^2
\end{eqnarray}
We need to estimate each term above.\\
1. $\lnm{\nabla
u}{\infty}^2$:\\
By the definition, we have
\begin{eqnarray}
\lnm{\nabla u}{\infty}^2\ls\lnm{Du}{\infty}^2+\lnm{\p_3u}{\infty}^2
\end{eqnarray}
By the divergence equation, we have
\begin{eqnarray}
\lnm{\p_3u_3}{\infty}^2\ls \lnm{Du}{\infty}^2+\lnm{G^2}{\infty}^2
\end{eqnarray}
For $i=1,2$, by Poincare we can naturally estimate
\begin{eqnarray}
\lnm{\p_3u_i}{\infty}^2\ls\lnm{\p_3^2u_i}{\infty}^2+\lnms{\p_3u_i}{\infty}{\Sigma}^2
\end{eqnarray}
The upper boundary condition implies that
\begin{eqnarray}
\lnms{\p_3u_i}{\infty}{\Sigma}^2\ls\lnms{Du_3}{\infty}{\Sigma}^2+\lnms{G^3}{\infty}{\Sigma}^2
\ls\lnm{D\nabla u_3}{\infty}^2+\lnms{G^3}{\infty}{\Sigma}^2
\end{eqnarray}
Considering the Navier-Stokes equation, we have
\begin{eqnarray}
\lnm{\p_3^2u_i}{\infty}^2\ls\lnm{\dt
u_i}{\infty}^2+\lnm{D^2u_i}{\infty}^2+\lnm{Dp}{\infty}^2+\lnm{G^1}{\infty}^2
\end{eqnarray}
Therefore, to summarize all above
\begin{eqnarray}
\\
\lnm{\nabla
u}{\infty}^2&\ls&\lnm{Du}{\infty}^2+\lnm{G^2}{\infty}^2+\lnm{D\nabla
u_3}{\infty}^2+\lnms{G^3}{\infty}{\Sigma}^2+\lnm{\dt
u_i}{\infty}^2+\lnm{D^2u_i}{\infty}^2\nonumber\\
&&+\lnm{G^1}{\infty}^2+\lnm{Dp}{\infty}^2\nonumber
\end{eqnarray}
The estimates of each term above with lemma \ref{lower interpolation estimate} satisfy our requirement\\
\ \\
2. $\lnm{\nabla^2u}{\infty}^2$:\\
By the definition, we have
\begin{eqnarray}
\lnm{\nabla^2u}{\infty}^2\ls\lnm{D\nabla
u}{\infty}^2+\lnm{\p_3^2u_i}{\infty}^2+\lnm{\p_3^2u_3}{\infty}^2
\end{eqnarray}
The first term naturally satisfies our requirement while the second
term has been successfully estimated in the first step above. Hence,
we only need to estimate the third one. Taking $\p_3$ in the
divergence equation reveals
\begin{eqnarray}
\lnm{\p_3^2u_3}{\infty}^2\ls\lnm{D\nabla u_i}{\infty}^2+\lnm{\nabla
G^2}{\infty}^2
\end{eqnarray}
where all the terms can be easily estimated via lemma \ref{lower interpolation estimate}.\\
\ \\
3. $\sum_{i=1}^2\hms{Du_i}{2}{\Sigma}^2$:\\
By trace theorem and Poincare lemma, we have
\begin{eqnarray}
\sum_{i=1}^2\hms{Du_i}{2}{\Sigma}^2\ls\hm{D^3\nabla
u_i}{0}^2+\hm{D^2\nabla u_i}{0}^2+\hm{Du_i}{1}^2
\end{eqnarray}
The first and second term are naturally estimated, so we only need
to focus on the last one. By Poincare, we have
\begin{eqnarray}
\hm{Du_i}{1}^2\ls\hm{D\p_3u_i}{0}^2+\hm{D^2u_i}{0}^2
\end{eqnarray}
and also
\begin{eqnarray}
\hm{D\p_3u_i}{0}^2\ls\hm{D\p_3^2u_i}{0}^2+\hms{D\p_3u_i}{0}{\Sigma}^2
\end{eqnarray}
The boundary condition shows that
\begin{eqnarray}
\hms{D\p_3u_i}{0}{\Sigma}^2\ls\hms{D^2u_3}{0}{\Sigma}^2+\hms{DG^3}{0}{\Sigma}^2
\ls\hm{D^2\p_3u_3}{0}^2+\hms{DG^3}{0}{\Sigma}^2
\end{eqnarray}
By Naiver-Stokes equation, we have
\begin{eqnarray}
\hm{D\p_3^2u_i}{0}^2\ls\hm{D\dt
u_i}{0}^2+\hm{D^3u_i}{0}^2+\hm{D^2p}{0}^2+\hm{DG^1}{0}^2
\end{eqnarray}
Hence, in total, we have
\begin{eqnarray}
\\
\sum_{i=1}^2\hm{Du_i}{2}^2&\ls&\hm{D^2\nabla
u_i}{0}^2+\hm{D^2u}{0}^2+\hm{D^2\p_3u_3}{0}^2+\hms{DG^3}{0}{\Sigma}^2+\hm{D\dt
u_i}{0}^2\nonumber\\
&&+\hm{D^2p}{0}^2+\hm{DG^1}{0}^2\nonumber
\end{eqnarray}
The estimate of each term above  with lemma \ref{lower interpolation
estimate} satisfies our requirement, so we finish this proof.
\end{proof}
\begin{lemma}\label{lower interpolation estimate improved 2}
We have the estimate
\begin{eqnarray}
\bar\ce_{N+2,2}&\ls&\ce_{2N}^{1/(\lambda+3)}\d_{N+2,2}^{(\lambda+2)/(\lambda+3)}
\end{eqnarray}
\end{lemma}
\begin{proof}
The dissipation $\d_{N+2,2}$ can control all the terms in
$\bar\ce_{N+2,2}$ except that $\dt\e$, $D^2\e$ and $D^{2(N+2)}\e$. A
direct application of lemma \ref{appendix interpolation 1} reveals
that
\begin{eqnarray}
\hm{D^{2(N+2)}\e}{0}^2\ls\ce_{2N}^{1/(4N-7)}\d_{N+2,2}^{(4N-8)/(4N-7)}
\end{eqnarray}
Since $N\geq3$, above estimate satisfies our requirement. Also lemma
\ref{lower interpolation estimate} reveals that the estimate of
$D^2\e$ will not be an obstacle. Hence, the dominating term is
$\dt\e$. By transport equation, we have
\begin{eqnarray}
\hms{\dt\e}{0}{\Sigma}^2\ls\hms{u_3}{0}{\Sigma}^2+\hms{G^4}{0}{\Sigma}^2
\end{eqnarray}
By Poincare, we can estimate
\begin{eqnarray}
\hms{u_3}{0}{\Sigma}^2\ls\hm{\p_3u_3}{0}^2
\end{eqnarray}
while the divergence equation and Poincare imply that
\begin{eqnarray}
\hm{\p_3u_3}{0}^2\ls\hm{Du_i}{0}^2+\hm{G^2}{0}^2\ls\hm{D\p_3u_i}{0}^2+\hm{G^2}{0}^2
\end{eqnarray}
Similar to the third step in the proof of lemma \ref{lower
interpolation estimate improved 1}, by Poincare, we have
\begin{eqnarray}
\hm{D\p_3u_i}{0}^2\ls\hm{D\p_3^2u_i}{0}^2+\hms{D\p_3u_i}{0}{\Sigma}^2
\end{eqnarray}
The boundary condition shows that
\begin{eqnarray}
\hms{D\p_3u_i}{0}{\Sigma}^2\ls\hms{D^2u_3}{0}{\Sigma}^2+\hms{DG^3}{0}{\Sigma}^2
\ls\hm{D^2\p_3u_3}{0}^2+\hms{DG^3}{0}{\Sigma}^2
\end{eqnarray}
By Naiver-Stokes equation, we have
\begin{eqnarray}
\hm{D\p_3^2u_i}{0}^2\ls\hm{D\dt
u_i}{0}^2+\hm{D^3u_i}{0}^2+\hm{D^2p}{0}^2+\hm{DG^1}{0}^2
\end{eqnarray}
where $D^2p$ dominates. Utilizing Poincare, we have
\begin{eqnarray}
\hm{D^2p}{0}^2\ls\hm{D^2\p_3p}{0}^2+\hms{D^2p}{0}{\Sigma}^2
\end{eqnarray}
in which we may again use the boundary condition to estimate
\begin{eqnarray}
\hms{D^2p}{0}{\Sigma}^2\ls\hms{D^2\e}{0}{\Sigma}^2+\hms{D^2\p_3u_3}{0}{\Sigma}^2+\hms{D^2G^3}{0}{\Sigma}^2
\end{eqnarray}
Hence, in total, we can estimate that
\begin{eqnarray}
\\
\hm{\dt\e}{0}^2&\ls&\hm{D^2\p_3u_3}{0}^2+\hms{DG^3}{0}{\Sigma}^2+\hm{D\dt
u_i}{0}^2+\hm{D^3u_i}{0}^2+\hm{DG^1}{0}^2\nonumber\\
&&+\hm{D^2\p_3p}{0}^2+\hms{D^2\e}{0}{\Sigma}^2+\hms{D^2\p_3u_3}{0}{\Sigma}^2+\hm{G^2}{0}^2+\hms{G^4}{0}{\Sigma}^2\nonumber\\
&&+\hms{D^2G^3}{0}{\Sigma}^2\nonumber
\end{eqnarray}
We may easily check via lemma \ref{lower interpolation estimate}
that each term above satisfies our requirement, so we finish this
proof.
\end{proof}
\begin{lemma}\label{lower interpolation estimate improved 3}
We have the estimate
\begin{eqnarray}
\ce_{N+2,1}\ls\ce_{N+2,2}^{(\lambda+1)/(\lambda+2)}\ce_{2N}^{1/(\lambda+2)}
\end{eqnarray}
\end{lemma}
\begin{proof}
We only need to estimate the lower order terms which merely appears
in $\ce_{N+2,1}$, which are $Du$, $Dp$ and $D\e$. Based on the lemma
\ref{lower interpolation estimate}, the estimates of $Du$ and $D\e$
have satisfied our requirement, so we focus on $Dp$. By a similar
argument as lemma \ref{lower interpolation estimate improved 2}, we
have that
\begin{eqnarray}
\hm{Dp}{0}^2&\ls&\hms{Dp}{0}{\Sigma}^2+\hm{D\p_3p}{0}^2\\
&\ls&\hms{D\e}{0}{\Sigma}^2+\hm{D\p_3^2u}{0}^2+\hms{G^3}{0}{\Sigma}^2+\hm{D\Delta
u}{0}^2+\hm{DG^1}{0}^2\nonumber
\end{eqnarray}
where each term can be estimated directly through lemma \ref{lower
interpolation estimate}. Then our result easily follows.
\end{proof}

\subsubsection{Higher Interpolation Estimates}

Then we need to give several estimates for the higher order terms.
\begin{lemma}
We have the estimate
\begin{eqnarray}
\hms{D^{2N+4}\e}{1/2}{\Sigma}^2+\hm{\nabla^{2N+5}\eb}{0}^2\ls\ce_{2N}^{2/(4N-7)}\d_{N+2,2}^{(4N-9)/(4N-7)}
\end{eqnarray}
\end{lemma}
\begin{proof}
See lemma 3.18 in \cite{book9}.
\end{proof}
\begin{lemma}
We have the estimate
\begin{eqnarray}
\hms{u}{2}{\Sigma}^2+\hms{\nabla u}{2}{\Sigma}^2\ls \d_{N+2,2}^{1/3}
\end{eqnarray}
\end{lemma}
\begin{proof}
Trace theorem implies that
\begin{eqnarray}
\hms{u}{2}{\Sigma}^2+\hms{\nabla u}{2}{\Sigma}^2\ls\hm{\nabla
u}{0}^2+\hm{\nabla^2u}{0}^2+\hm{D^2\nabla
u}{0}^2+\hm{D^2\nabla^2u}{0}^2
\end{eqnarray}
The last two terms naturally satisfy our requirement, so we
concentrate on the first two. Similar to the argument in lemma
\ref{lower interpolation estimate improved 1}, we have
\begin{eqnarray}
\hm{\nabla
u}{0}^2+\hm{\nabla^2u}{0}^2&\ls&\hm{Du}{0}^2+\hm{G^2}{0}^2+\hm{D\nabla
u}{0}^2+\hms{G^3}{0}{\Sigma}^2\\
&&+\hm{\dt
u}{0}^2+\hm{D^2u}{0}^2+\hm{Dp}{0}^2+\hm{G^1}{0}^2+\hm{\nabla
G^2}{0}^2\nonumber
\end{eqnarray}
The estimate of each term on the right hand side easily lead to our
result.
\end{proof}
\begin{lemma}\label{higher interpolation estimate}
Let $P=P(K,D\e)$ be a polynomial of $K$ and $D\e$. Then there exists
a $\theta>0$ such that
\begin{eqnarray}
\hms{(D^{2N+4}\e)u}{1/2}{\Sigma}^2+\hms{(D^{2N+4}\e)P\nabla
u}{1/2}{\Sigma}^2\ls\ce^{\theta}_{2N}\d_{N+2,2}
\end{eqnarray}
Let $Q=Q(K,\tilde b, \nabla\eb)$ be a polynomial. Then there exists
a $\theta>0$, such that
\begin{eqnarray}
\hm{(\nabla^{2N+5}\eb)Q\nabla u}{0}^2\ls\ce^{\theta}_{2N}\d_{N+2,2}
\end{eqnarray}
\end{lemma}
\begin{proof}
This lemma is identical to lemma 3.19 in \cite{book9}. The above two
lemmas implies that all the results used in proving lemma 3.19 have
been recovered now, so this lemma is still valid.
\end{proof}

\subsection{Nonlinear Estimates}

\subsubsection{Estimates of Perturbation Terms}

Now we will give several estimates for the nonlinear terms in
perturbed linear form. Similar to theorem 4.1 and 4.2 in
\cite{book9}, the estimating methods for these terms are quite
standard, so we will merely give the basic rules in estimating:
\begin{enumerate}
\item
Since all the terms are quadratic or of higher order, then we may
expand the differential operators with Leibniz rule and estimate
resulting quadratic terms by the basic interpolation estimate in
lemma \ref{lower interpolation estimate} either in $L^{\infty}$ or
$L^2$ norm, combining with the definition of energy and dissipation,
trace theorem and Sobolev embedding. For the choices of form
(\ref{interpolation temp 4}) or (\ref{interpolation temp 5}), we
should always take the higher interpolation index for $\ce_{N+2,2}$
and $\d_{N+2,2}$.
\item
For the $2N$ level, there is no minimum count for the derivatives,
hence we only need to refer to the definition of energy and
dissipation; however, in $N+2$ level, we have to resort to lemma
\ref{lower interpolation estimate} for terms that cannot be
estimated directly. Also, lemma \ref{higher interpolation estimate}
can be utilized for more complicated terms.
\item
Note that the most important difference between our proof and that
of theorem 4.1 and 4.2 in \cite{book9} is that we concentrate on the
horizontal derivatives instead of the full derivatives. Thereafter,
in the estimates related to $\d_{N+2,2}$, we should try to avoid the
introduction of vertical derivatives via Sobolev embedding,
especially for pressure $p$.
\item
For the term in form of $\hms{X}{2}{\Sigma}$, instead of using trace
theorem directly, we should first write it into
$\hms{X}{0}{\Sigma}+\hms{DX}{0}{\Sigma}+\hms{D^2X}{0}{\Sigma}$ and
then apply trace theorem, which will only produce one vertical
derivatives rather than three.
\end{enumerate}
Considering that the proofs for the following estimates are lengthy
but not difficult based on the above rules, and also most of them
have been achieved in that of theorem 4.1 and 4.2 of \cite{book9},
we will omit the details here and only present the theorems.
\begin{lemma}\label{lower nonlinear estimate}
There exists a $\theta>0$ such that
\begin{eqnarray}
\\
\hm{\bar D^2\bar\nabla_0^{2(N+2)-4}G^1}{0}^2+\hm{\bar
D^2\bar\nabla_0^{2(N+2)-3}G^2}{0}^2+\hms{\bar
D^2\bar\nabla_0^{2(N+2)-4}G^3}{1/2}{\Sigma}^2
\ls\ce_{2N}^{\theta}\ce_{N+2,2}\nonumber
\end{eqnarray}
and
\begin{eqnarray}
\\
\hm{\bar D_2^{2(N+2)-1}G^1}{0}^2+\hm{\bar
D_1^{2(N+2)-1}G^2}{1}^2+\hms{\bar
D_2^{2(N+2)-1}G^3}{1/2}{\Sigma}^2+\hms{\bar
D_1^{2(N+2)-1}G^4}{1/2}{\Sigma}^2\nonumber\\
\ls\ce_{2N}^{\theta}\d_{N+2,2}\nonumber
\end{eqnarray}
\end{lemma}
\begin{lemma}\label{higher nonlinear estimate}
There exists a $\theta>0$ such that
\begin{eqnarray}
\\
\hm{\bar \nabla_0^{4N-2}G^1}{0}^2+\hm{\bar
\nabla_0^{4N-1}G^2}{0}^2+\hms{\bar
\nabla_0^{4N-2}G^3}{1/2}{\Sigma}^2\ls\ce_{2N}^{1+\theta}\nonumber
\end{eqnarray}
and
\begin{eqnarray}
\\
\hm{\bar D_0^{4N-1}G^1}{0}^2+\hm{\bar D_0^{4N-1}G^2}{1}^2+\hms{\bar
D_0^{4N-1}G^3}{1/2}{\Sigma}^2+\hms{\bar
D_0^{4N-1}G^4}{1/2}{\Sigma}^2\nonumber\\
\ls\ce_{2N}^{\theta}\d_{2N}+\k\f_{2N}\nonumber
\end{eqnarray}
\end{lemma}

\subsubsection{Estimates of Other Nonlinearities}

Here is a few lemmas for some other nonlinear terms which have been
proved in proposition 4.3-4.5 of \cite{book9}. Although our
definition of energy and dissipation is a little bit different, we
do not even need to revise the proof here.
\begin{lemma}\label{riesz potential nonlinear estimate 1}
We have that
\begin{eqnarray}
\hm{\i G^1}{1}^2+\hm{\i G^2}{2}^2+\hm{\i\dt G^2}{0}^2\ls
\ce_{2N}\min\{\ce_{2N},\d_{2N}\}
\end{eqnarray}
and
\begin{eqnarray}
\hms{\i G^3}{1}{\Sigma}^2+\hms{\i
G^4}{1}{\Sigma}^2\ls\ce_{2N}\min\{\ce_{2N},\d_{2N}\}
\end{eqnarray}
Also
\begin{eqnarray}
\hms{\i G^4}{0}{\Sigma}^2\ls \d_{2N}^2
\end{eqnarray}
\end{lemma}
\begin{lemma}\label{riesz potential nonlinear estimate 2}
It holds that
\begin{eqnarray}
\hm{\i[(AK)\p_3u_1+(BK)\p_3u_2]}{0}^2+\sum_{i=1}^2\hm{\i[u\p_iK]}{0}^2\ls\d_{2N}^2
\end{eqnarray}
and
\begin{eqnarray}
\hm{\i[(1-K)u]}{0}^2\ls\ce_{2N}^{1/(1+\lambda)}\d_{2N}^{(1+2\lambda)/(1+\lambda)}
\end{eqnarray}
Also
\begin{eqnarray}
\hm{\i[(1-K)G^2]}{0}^2\ls\ce_{2N}\d_{2N}^2
\end{eqnarray}
\end{lemma}
\begin{lemma}\label{riesz potential nonlinear estimate 3}
We have that
\begin{eqnarray}
\hm{\dt^{2N+1}\a}{0}^2\ls\d_{2N}
\end{eqnarray}
and
\begin{eqnarray}
\hm{\dt^{N+3}\a}{0}^2\ls\d_{N+2,2}
\end{eqnarray}
\end{lemma}

\subsection{Energy Estimates}

In this section, we will present the energy estimate for horizontal
derivatives. Since this part has been carefully completed in section
5 of \cite{book9} and is not the main improvement of our new proof,
it is not necessary for us to show all the details. Hence, we will
not give all the details of the proofs and only comment when we need
different technique to achieve the estimates due to the changing of
definition of energy and dissipation.

\subsubsection{Estimates of Highest and Lowest Temporal Derivatives}

\begin{lemma}
Let $\p^{\alpha}=\dt^{2N}$ and let $F^i$ be defined as in geometric
structure form. Then
\begin{eqnarray}
\hm{F^1}{0}^2+\hm{\dt(JF^2)}{0}^2+\hms{F^3}{0}{\Sigma}^2+\hms{F^4}{0}{\Sigma}^2\ls\ce_{2N}\d_{2N}
\end{eqnarray}
\end{lemma}
\begin{proof}
This is identical to theorem 5.1 in \cite{book9}, so we omit the
proof here.
\end{proof}
\begin{lemma}
Let $\p^{\alpha}=\dt^{N+2}$ and let $F^i$ be defined as in geometric
structure form. Then
\begin{eqnarray}
\hm{F^1}{0}^2+\hm{\dt(JF^2)}{0}^2+\hms{F^3}{0}{\Sigma}^2+\hms{F^4}{0}{\Sigma}^2\ls\ce_{2N}\d_{N+2,2}
\end{eqnarray}
Also, if $N\geq3$, then there exists a $\theta>0$ such that
\begin{eqnarray}
\hm{F^2}{0}^2\ls\ce_{2N}^{\theta}\ce_{N+2,2}
\end{eqnarray}
\end{lemma}
\begin{proof}
This is a revised version of theorem 5.2 in \cite{book9}. The first
part is exactly the same as the original theorem, so it is naturally
valid. However, we will utilize different techniques in the second
part. By the definition, $F^2=F^{2,1}+F^{2,2}$. By the same argument
as (5.5) in \cite{book9}, we have
\begin{eqnarray}
\hm{F^{2,1}}{0}^2\ls\ce_{2N}^{\theta}\ce_{N+2,2}
\end{eqnarray}
For $F^{2,2}$, a calculation reveals that
\begin{eqnarray}
F^{2,2}=\dt^{N+2}(\p_1\eb\tilde bK)\p_3u_1+\dt^{N+2}(\p_2\eb\tilde
bK)\p_3u_2-\dt^{N+2}K\p_3u_3
\end{eqnarray}
We can estimate
\begin{eqnarray}
\hm{\dt^{N+2}(\p_1\eb\tilde
bK)\p_3u_1}{0}^2\ls\hm{\dt^{N+2}(\p_1\eb\tilde
bK)}{0}^2\lnm{\p_3u_1}{\infty}^2
\end{eqnarray}
Lemma \ref{lower interpolation estimate improved 1} implies that
\begin{eqnarray}
\lnm{\p_3u_1}{\infty}^2\ls\lnm{\nabla
u}{\infty}^2\ls\ce_{2N}^{r/(2+r)}\ce_{N+2,2}^{2/(2+r)}
\end{eqnarray}
Also, we have that for $0\leq\abs{\alpha}\leq N+2$, there exists a
$\theta>0$ such that
\begin{eqnarray}
\hm{\dt^{\alpha}\p_1\eb}{0}^2\ls\hms{\dt^{\alpha}\e}{1/2}{\Sigma}^2
\ls\hms{\dt^{\alpha}\e}{0}{\Sigma}\hms{\dt^{\alpha}\e}{1}{\Sigma}\ls\ce_{2N}^{1-\theta}\ce_{N+2,2}^{\theta}
\end{eqnarray}
Hence we have
\begin{eqnarray}
\hm{\dt^{N+2}(\p_1\eb\tilde
bK)}{0}^2\ls\hm{\dt^{\alpha}\p_1\eb}{0}^2\lnm{\dt^{N+2-\alpha}(\tilde
bK)}{\infty}^2\ls\ce_{2N}^{1-\theta}\ce_{N+2,2}^{\theta}
\end{eqnarray}
Then when $r$ is sufficiently small, we can naturally estimate
\begin{eqnarray}
\hm{\dt^{N+2}(\p_1\eb\tilde
bK)\p_3u_1}{0}^2\ls\ce_{2N}^{\theta}\ce_{N+2,2}
\end{eqnarray}
Similarly, we have
\begin{eqnarray}
\hm{\dt^{N+2}(\p_2\eb\tilde
bK)\p_3u_2}{0}^2\ls\ce_{2N}^{\theta}\ce_{N+2,2}
\end{eqnarray}
The same argument as (5.11) in \cite{book9} shows that
\begin{eqnarray}
\hm{\dt^{N+2}K\p_3u_3}{0}^2\ls\ce_{2N}^{\theta}\ce_{N+2,2}
\end{eqnarray}
Therefore, our result easily follows.
\end{proof}
\ \\
The following three lemmas is just the restatement of proposition
5.3-5.5 in \cite{book9}, then we will only present the lemmas
without the proofs.
\begin{lemma}\label{energy temp 1}
There exists a $\theta>0$ such that
\begin{eqnarray}
\\
\hm{\dt^{2N}u(t)}{0}^2+\hms{\dt^{2N}\e(t)}{0}{\Sigma}^2+\int_0^t\hm{\dm\dt^{2N}u}{0}^2\ls
\ce_{2N}(0)+(\ce_{2N}(t))^{3/2}+\int_0^t\ce_{2N}^{\theta}\d_{2N}\nonumber
\end{eqnarray}
\end{lemma}
\begin{lemma}\label{energy temp 2}
Let $F^2$ be given with $\dt^{\alpha}=\dt^{N+2}$. Then it holds that
\begin{eqnarray}
\\
\dt\bigg(\hm{\sqrt{J}\dt^{N+2}u}{0}^2+\hms{\dt^{N+2}\e}{0}{\Sigma}^2-2\int_{\Omega}J\dt^{N+1}pF^2\bigg)
+\hm{\dm\dt^{N+2}u}{0}^2\ls\sqrt{\ce_{2N}}\d_{N+2,2}\nonumber
\end{eqnarray}
\end{lemma}
\begin{lemma}\label{energy temp 3}
It holds that
\begin{eqnarray}
\hm{u(t)}{0}^2+\hms{\e(t)}{0}{\Sigma}^2+\int_0^t\hm{\dm
u}{0}^2\ls\ce_{2N}(0)+\int_0^t\sqrt{\ce_{2N}}\d_{2N}
\end{eqnarray}
\end{lemma}

\subsubsection{Estimates of Mixed Derivatives}

\begin{lemma}
Let $\alpha\in\mathbb{N}^2$ be such that $\abs{\alpha}=4N$. Then
\begin{eqnarray}
\abs{\int_{\Sigma}\p^{\alpha}\e\p^{\alpha}G^4}\ls\sqrt{\ce_{2N}}\d_{2N}+\sqrt{\d_{2N}\k\f_{2N}}
\end{eqnarray}
\end{lemma}
\begin{proof}
Essentially, this lemma is completely identical to lemma 6.1 in
\cite{book9}. We only need to notice that although our definition of
energy and dissipation is different from that in \cite{book9}, the
terms related to $\e$ is exactly the same. Also we should refer to
the general rules for estimating in the section of nonlinear
estimates to deal with $\hms{X}{2}{\Sigma}$ form. Hence, this lemma
is still valid now.
\end{proof}
\begin{lemma}\label{energy temp 4}
Suppose that $\alpha\in\mathbb{N}^{1+2}$ is such that $\alpha_0\leq
2N-1$ and $1\leq\abs{\alpha}\leq4N$. Then there exists a $\theta>0$
such that
\begin{eqnarray}
\hm{\bar D_1^{4N-1}u}{0}^2+\hm{D\bar D^{4N-1}u}{0}^2+\hms{\bar
D_1^{4N-1}\e}{0}{\Sigma}^2+\hms{D\bar D^{4N-1}\e}{0}{\Sigma}^2\\
+\int_0^t\hm{\bar D_1^{4N-1}\dm u}{0}^2+\hm{D\bar D^{4N-1}\dm
u}{0}^2\ls\bar\ce_{2N}(0)+\int_0^t\ce_{2N}^{\theta}\d_{2N}+\sqrt{\d_{2N}\k\f_{2N}}\nonumber
\end{eqnarray}
\end{lemma}
\begin{proof}
This lemma is identical to proposition 6.2 in \cite{book9}, and the
proof is almost the same. We only need to notice that now we should
utilize lemma \ref{higher nonlinear estimate} to estimate the
nonlinear terms to replace lemma 4.2 in \cite{book9}. Hence, the
result still holds.
\end{proof}
\begin{lemma}\label{energy temp 999}
We have the estimate
\begin{eqnarray}
\hms{D^{2N+3}G^4}{1/2}{\Sigma}^2\ls\d_{N+2,2}^{1+2/(4N-7)}
\end{eqnarray}
Also, there exists a $\theta>0$ such that
\begin{eqnarray}
\hms{\bar
D^2G^4}{0}{\Sigma}^2\ls\ce_{2N}^{\theta}\d_{N+2,2}^{1+1/(\lambda+3)}
\end{eqnarray}
and
\begin{eqnarray}
\lnm{\bar
D^2G^2}{1}^2\ls\ce_{2N}^{\theta}\d_{N+2,2}^{1+1/(\lambda+3)}
\end{eqnarray}
\end{lemma}
\begin{proof}
This is basically lemma 6.3 in \cite{book9}, the only difference is
that we need to estimate $\hm{\nabla u}{0}$ directly now. By a
similar argument as in lemma \ref{lower interpolation estimate
improved 1}, we have
\begin{eqnarray}
\\
\hm{\nabla u}{0}^2&\ls&\hm{Du}{0}^2+\hm{G^2}{0}^2+\hm{D\nabla
u_3}{0}^2+\hms{G^3}{0}{\Sigma}^2+\hm{\dt
u_i}{0}^2+\hm{D^2u_i}{0}^2\nonumber\\
&&+\hm{G^1}{0}^2+\hm{Dp}{0}^2\nonumber
\end{eqnarray}
Hence, an application of lemma \ref{lower interpolation estimate}
implies that $\hm{\nabla
u}{0}^2\ls\d_{N+2,2}^{(\lambda+1)/(\lambda+3)}$. Using exactly the
same argument for the other part of the proof, we can see the
validity of this lemma.
\end{proof}
\begin{lemma}\label{energy temp 5}
Suppose that $\alpha\in\mathbb{N}^{1+2}$ is such that $\alpha_0\leq
N+1$ and $2\leq\abs{\alpha}\leq 2(N+2)$. Then there exists a
$\theta>0$ such that
\begin{eqnarray}
\dt\bigg(\hm{\bar D_2^{2N+3}u}{0}^2+\hm{D\bar
D^{2N+3}u}{0}^2+\hms{\bar D_2^{2N+3}\e}{0}{\Sigma}^2+\hms{D\bar
D^{2N+3}\e}{0}{\Sigma}^2\bigg)\\
+\hm{\bar D_2^{2N+3}\dm u}{0}^2+\hm{D\bar D^{2N+3}\dm
u}{0}^2\ls\ce_{2N}^{\theta}\d_{N+2,2}\nonumber
\end{eqnarray}
\end{lemma}
\begin{proof}
This is almost the same as proposition 6.4 in \cite{book9}, so we
omit the detailed proof here. We only need to note that for higher
derivatives, we should use lemma \ref{lower nonlinear estimate} to
replace lemma 4.1 in \cite{book9} to estimate nonlinear terms. Then
for lower derivatives, we should utilize lemma \ref{lower
interpolation estimate}, \ref{lower interpolation estimate improved
2} and \ref{energy temp 999} to replace lemma 3.1, 3.16 and 6.3 in
\cite{book9}, which share exactly the same results. Therefore, the
conclusion naturally holds.
\end{proof}

\subsubsection{Estimates of Riesz Potentials}

In the following, we will record several lemmas related to Riesz
potentials, which have been proved in lemma 6.5-6.7 in \cite{book9}.
Note that all the preliminary lemmas have been recovered in lemma
\ref{riesz potential nonlinear estimate 1} and \ref{riesz potential
nonlinear estimate 2}, so we do not even need to revise the proofs.
\begin{lemma}
It holds that
\begin{eqnarray}
\hm{\i p}{0}^2&\ls&\ce_{2N}\\
\hm{\i
Dp}{0}^2&\ls&\ce_{2N}^{\lambda/(1+\lambda)}\d_{2N}^{1/(1+\lambda)}
\end{eqnarray}
\end{lemma}
\begin{lemma}
It holds that
\begin{eqnarray}
\abs{\int_{\Omega}\i p\i G^2}\ls\sqrt{\ce_{2N}}\d_{2N}
\end{eqnarray}
\end{lemma}
\begin{lemma}\label{energy temp 6}
It holds that
\begin{eqnarray}
\hm{\i u}{0}^2+\hms{\i\e}{0}{\Sigma}^2+\int_0^t\hm{\dm\i
u}{0}^2\ls\ce_{2N}(0)+\int_0^t\sqrt{\ce_{2N}}\d_{2N}
\end{eqnarray}
\end{lemma}

\subsubsection{Synthesis of Energy Estimates}

Assembling lemma \ref{energy temp 1}, \ref{energy temp 2},
\ref{energy temp 3}, \ref{energy temp 4}, \ref{energy temp 5} and
\ref{energy temp 6}, we have the following energy estimate.
\begin{proposition}\label{lower energy estimate}
Let $F^2$ be given with $\dt^{N+2}$. Then there exists a $\theta>0$
such that
\begin{eqnarray}
\dt\bigg(\bar\ce_{N+2,2}-2\int_{\Omega}K\dt^{N+1}pF^2\bigg)+\bar\d_{N+2,2}\ls\ce_{2N}^{\theta}\d_{N+2,2}
\end{eqnarray}
\end{proposition}
\begin{proposition}\label{higher energy estimate}
There exists a $\theta>0$ such that
\begin{eqnarray}
\bar\ce_{2N}(t)+\int_0^t\bar\d_{2N}(r)\ud{r}\ls\ce_{2N}(0)+(\ce_{2N}(t))^{3/2}+\int_0^t(\ce_{2N}(r))^{\theta}\d_{2N}\ud{r}\\
+\int_0^t\sqrt{\d_{2N}(r)\k(r)\f_{2N}(r)}\ud{r}\nonumber
\end{eqnarray}
\end{proposition}

\subsection{Comparison Estimates}

In this section, we will prove several comparison results between
horizontal derivatives and full derivatives, both in energy and
dissipation.

\subsubsection{Dissipation Comparison Estimates}

\begin{proposition}\label{lower comparison dissipation estimate}
There exists a $\theta>0$ such that
\begin{eqnarray}
\d_{N+2,2}\ls\bar\d_{N+2,2}+\ce_{2N}^{\theta}\d_{N+2,2}
\end{eqnarray}
\end{proposition}
\begin{proof}
We define
\begin{eqnarray}
\y_{N+2,2}=\hm{\bar D_2^{2(N+2)-1}G^1}{0}^2+\hm{\bar
D_2^{2(N+2)-1}G^2}{1}^2+\hms{\bar
D_2^{2(N+2)-1}G^3}{1/2}{\Sigma}^2\\
+\hms{\bar D_1^{2(N+2)-1}G^4}{1/2}{\Sigma}^2\nonumber
\end{eqnarray}
and
\begin{eqnarray}
\z=\bar\d_{N+2,2}+\y_{N+2,2}
\end{eqnarray}
Then the proof is divided into several steps. \\
\ \\
Step 1: Application of Korn's inequality.\\
Since any horizontal derivatives of velocity vanishes on the lower
boundary, we can directly apply Korn's inequality (see lemma
\ref{appendix poincare 3}) to achieve that
\begin{eqnarray}
\hm{\bar D_{2}^{2(N+2)}u}{1}^2\ls\hm{\bar D_{2}^{2(N+2)}\dm
u}{0}^2=\bar\d_{N+2,2}\ls\z
\end{eqnarray}

\ \\
Step 2: Initial estimates of velocity and pressure.\\
Let $0\leq j\leq (N+2)-1$ and $\alpha\in\mathbb{N}^2$ be such that
$2\leq 2j+\abs{\alpha}\leq 2(N+2)-1$. Hence, we naturally have
\begin{eqnarray}
\hm{\p^{\alpha}\dt^{j+1}u}{0}^2\leq \hm{\bar
D_2^{2(N+2)}u}{1}^2\ls\z
\end{eqnarray}
We apply the operator $\p^{\alpha}\dt^j\p_3$ to the divergence
equation in (\ref{perturbed linear form equation}) and rearrange the
terms
\begin{eqnarray}
\p^{\alpha}\dt^j\p_3(\p_3u_3)=\p^{\alpha}\dt^j\p_3(-\p_1u_1-\p_2u_2)+\p^{\alpha}\dt^j\p_3G^2
\end{eqnarray}
Thus we have
\begin{eqnarray}
\hm{\p^{\alpha}\dt^j\p_3^2u_3}{0}^2\ls\hm{\bar
D_2^{2(N+2)}u}{1}^2+\hm{\p^{\alpha}\dt^jG^2}{1}^2\ls\z
\end{eqnarray}
Hence, it is easily implied that
\begin{eqnarray}
\hm{\Delta\p^{\alpha}\dt^ju_3}{0}^2\ls\z
\end{eqnarray}
Apply the operator $\p^{\alpha}\dt^j$ to the third equation in
(\ref{perturbed linear form equation}), then we have
\begin{eqnarray}
\hm{\p^{\alpha}\dt^j\p_3p}{0}^2\ls\hm{\p^{\alpha}\dt^j\Delta
u}{0}^2+\hm{\p^{\alpha}\dt^jG^1}{0}^2\ls\z
\end{eqnarray}
So we only need to estimate the terms related to $\p_3^2u_i$ and
$\p_ip$ for $i=1,2$.

\ \\
Step 3: Estimates of velocity and pressure in the upper domain.\\
In order to achieve above estimate, we need to employ a localization
argument. Define a cutoff function $\chi_1=\chi_1(x_3)$ given by
$\chi_1\in C_c^{\infty}(\mathbb{R})$ with $\chi(x_3)=1$ for
$x_3\in\Omega_1=[-2b/3,0]$ and $\chi_1(x_3)=0$ for
$x_3\notin(-3b/4,1/2)$.\\
Define the curl of the velocity $w=\nabla\times u$, which means
$w_1=\p_2u_3-\p_3u_2$ and $w_2=\p_3u_1-\p_1u_3$. Therefore,
$\chi_1w_i$ for $i=1,2$ satisfy the equation
\begin{eqnarray}
\\
\Delta\p^{\alpha}\dt^j(\chi_1w_i)=\chi_1\p^{\alpha}\dt^{j+1}w_i+
2(\p_3\chi_1)(\p^{\alpha}\dt^j\p_3w_i)+(\p_3^2\chi_1)(\p^{\alpha}\dt^jw_i)
-\chi_1\nabla\times(\p^{\alpha}\dt^jG^1)\nonumber
\end{eqnarray}
in $\Omega$ as well as the boundary condition
\begin{equation}
\left\{
\begin{array}{ll}
\p^{\alpha}\dt^j(\chi_1w_1)=2\p_2\p^{\alpha}\dt^ju_3+\p^{\alpha}\dt^jG^3\cdot
e_2&\rm{on}\ \Sigma\\
\p^{\alpha}\dt^j(\chi_1w_2)=-2\p_1\p^{\alpha}\dt^ju_3-\p^{\alpha}\dt^jG^3\cdot
e_1&\rm{on}\ \Sigma\\
\p^{\alpha}\dt^j(\chi_1w_1)=\p^{\alpha}\dt^j(\chi_1w_2)=0&\rm{on}\
\Sigma_b
\end{array}
\right.
\end{equation}
Based on the standard elliptic estimate for Poisson equation, we
have
\begin{eqnarray}
\hm{\p^{\alpha}\dt^j(\chi_1w_i)}{1}^2\ls\hm{\Delta\p^{\alpha}\dt^j(\chi_1w_i)}{-1}^2
+\hms{\p^{\alpha}\dt^j(\chi_1w_i)}{1/2}{\Sigma}^2
\end{eqnarray}
Hence, we have to estimate each term on the right hand side.\\
Let $\phi\in H_0^1(\Omega)$. For $\alpha\neq0$ we may write
$\alpha=\beta+(\alpha-\beta)$ with $\abs{\beta}=1$. Then an
integration by parts reveals that
\begin{eqnarray}
\abs{\int_{\Omega}\phi\chi_1\p^{\alpha}\dt^{j+1}w_i}=
\abs{\int_{\Omega}\p^{\beta}\phi\chi_1\p^{\alpha-\beta}\dt^{j+1}w_i}
\leq\hm{\phi}{1}\hm{\chi_1\bar D_2^{2(N+2)}w_i}{0}
\end{eqnarray}
Since $2(j+1)+\abs{\alpha-\beta}\in[3,2(N+2)]$, we naturally have
\begin{eqnarray}
\hm{\chi_1\bar D_2^{2(N+2)}w_i}{0}^2\ls\hm{\bar
D_2^{2(N+2)}u}{1}^2\ls\z
\end{eqnarray}
Then if we take the supreme over all $\phi$ such that
$\hm{\phi}{1}\leq1$, then we may get
\begin{eqnarray}
\hm{\chi_1\p^{\alpha}\dt^{j+1}w_i}{-1}^2\ls\z
\end{eqnarray}
A similar argument without integration by parts shows that the above
result is also true for $\alpha=0$ since in this case $j\leq
(N+2)-1$ implies $4\leq 2(j+1)\leq 2(N+2)$. In a similar fashion, we
may apply integration by parts for $\p_3$ for the other terms in
$H^{-1}$ norm and achieve that
\begin{eqnarray}
\\
\hm{2(\p_3\chi_1)(\p^{\alpha}\dt^j\p_3w_i)}{-1}^2\ls(\lnm{\p_3\chi_1}{\infty}^2+\lnm{\p_3^2\chi_1}{\infty}^2)
\hm{\bar D_2^{2(N+2)}w_i}{0}^2\ls
\hm{D_2^{2(N+2)}u}{1}^2\ls\z\nonumber
\end{eqnarray}
\begin{eqnarray}
\hm{(\p_3^2\chi_1)(\p^{\alpha}\dt^jw_i)}{-1}^2\ls\lnm{\p_3^2\chi_1}{\infty}^2
\hm{\bar D_2^{2(N+2)}w_i}{0}^2\ls \hm{D_2^{2(N+2)}u}{1}^2\ls\z
\end{eqnarray}
\begin{eqnarray}
\hm{\chi_1\nabla\times(\p^{\alpha}\dt^jG^1)}{-1}^2\ls(\lnm{\chi_1}{\infty}^2+\lnm{\p_3\chi_1}{\infty}^2)
\hm{\bar D_2^{2(N+2)-1}G^1}{0}^2\ls\z
\end{eqnarray}
Above four estimates together yield
\begin{eqnarray}
\hm{\Delta\p^{\alpha}\dt^j(\chi_1w_i)}{-1}^2\ls\z
\end{eqnarray}
Then for the boundary terms, a direct application of trace theorem
shows that
\begin{eqnarray}
\hms{\p^{\alpha}\dt^j\p_1u_3}{1/2}{\Sigma}^2+\hms{\p^{\alpha}\dt^j\p_2u_3}{1/2}{\Sigma}^2
\ls\hm{\bar D_2^{2(N+2)}u}{1}^2\ls\z
\end{eqnarray}
\begin{eqnarray}
\hms{\p^{\alpha}\dt^jG^3}{1/2}{\Sigma}^2\ls\hms{\bar
D_2^{2(N+2)-1}G^3}{1/2}{\Sigma}^2\ls\z
\end{eqnarray}
Then above implies that
\begin{eqnarray}
\hms{\p^{\alpha}\dt^j(\chi_1w_i)}{1/2}{\Sigma}^2\ls\z
\end{eqnarray}
Hence, the elliptic estimate gives us the final form
\begin{eqnarray}
\hms{\p^{\alpha}\dt^jw_i}{1}{\Omega_1}^2\ls\hm{\p^{\alpha}\dt^j(\chi_1w_i)}{1}^2\ls\z
\end{eqnarray}
Since $\p_3^2\p^{\alpha}\dt^ju_1=\p_3\p^{\alpha}\dt^j(w_2+\p_1u_3)$
and $\p_3^2\p^{\alpha}\dt^ju_2=\p_3\p^{\alpha}\dt^j(\p_2u_3-w_1)$,
we can deduce from above that for $i=1,2$,
\begin{eqnarray}
\hms{\p_3^2\p^{\alpha}\dt^ju_i}{0}{\Omega_1}^2\ls\z
\end{eqnarray}
which, by considering the Navier Stokes equation, further implies
\begin{eqnarray}
\hms{\p_i\p^{\alpha}\dt^jp}{0}{\Omega_1}^2\ls\z
\end{eqnarray}
Combining all above, we have achieved that
\begin{eqnarray}
\hms{\bar D_2^{2(N+2)-1}u}{2}{\Omega_1}^2+\hms{\bar
D_2^{2(N+2)-1}\nabla p}{0}{\Omega_1}^2\ls\z
\end{eqnarray}

\ \\
Step 4: Estimates of velocity and pressure in the lower domain.\\
Now we extend above estimate to the lower domain
$\Omega_2=[-b,-b/3]$. Define a cutoff function $\chi_2\in
C_c^{\infty}(\mathbb{R})$ such that $\chi_2(x_3)=1$ for
$x_3\in\Omega_2$ and $\chi_2(x_3)=0$ for $x_3\notin (-2b,-b/6)$.
Notice that this cutoff function satisfies that
$supp(\nabla\chi_2)\in\Omega_1$, which will play a key role in the
following.\\
Then the localized variables satisfy the equation
\begin{equation}\label{comparison temp 1}
\left\{
\begin{array}{ll}
-\Delta(\chi_2u)+\nabla(\chi_2p)=-\dt(\chi_2u)+\chi_2G^1+H^1&\rm{in}\
\Omega\\
\nabla\cdot(\chi_2u)=\chi_2G^2+H^2&\rm{in}\ \Omega\\
\chi_2u=0&\rm{on}\ \Sigma\\
\chi_2u=0&\rm{on}\ \Sigma_b
\end{array}
\right.
\end{equation}
where
\begin{eqnarray}
H^1=\p_3\chi_2(pe_3-2\p_3u)-(\p_3^2\chi_2)u\qquad H^2=\p_3\chi_2u_3
\end{eqnarray}
Let $0\leq j\leq (N+2)-1$ and $\alpha\in\mathbb{N}^2$ such that
$2\leq 2j+\abs{\alpha}\leq 2(N+2)-1$. Then we may apply the
differential operator $\p^{\alpha}\dt^j$ on (\ref{comparison temp
1}) and consider the elliptic estimate (\ref{appendix elliptic weak
estimate}) for Navier Stokes equation, which implies that
\begin{eqnarray}
\\
\hm{\p^{\alpha}\dt^j\p_i(\chi_2u)}{1}^2+\hm{\p^{\alpha}\dt^j\p_i(\chi_2p)}{0}^2\ls
\hm{\p^{\alpha}\dt^{j+1}\p_i(\chi_2u)}{-1}^2\nonumber\\
+\hm{\p^{\alpha}\dt^j\p_i(\chi_2G^1+H^1)}{-1}^2+\hm{\p^{\alpha}\dt^j\p_i(\chi_2G^2+H^2)}{0}^2\nonumber
\end{eqnarray}
Similar to the argument in the upper domain, via integration by
parts, it is easy to deduce the $H^{-1}$ estimates that
\begin{eqnarray}
\hm{\p^{\alpha}\dt^{j+1}\p_i(\chi_2u)}{-1}^2
+\hm{\p^{\alpha}\dt^j\p_i(\chi_2G^1)}{-1}^2+\hm{\p^{\alpha}\dt^j\p_i(\chi_2G^2)}{0}^2\ls\z
\end{eqnarray}
For the remaining part, we may directly estimate
\begin{eqnarray}
\hm{\p^{\alpha}\dt^j\p_iH^1}{-1}^2+\hm{\p^{\alpha}\dt^j\p_iH^2}{0}^2\ls
\hm{\p^{\alpha}\dt^j\p_iH^1}{0}^2+\hm{\p^{\alpha}\dt^j\p_iH^2}{0}^2\\
\ls\hms{\bar D_2^{2(N+2)-1}u}{2}{\Omega_1}^2+\hms{\bar
D_2^{2(N+2)-1}\nabla p}{0}{\Omega_1}^2\ls\z\nonumber
\end{eqnarray}
where we utilize the property for $\nabla\chi_2$ stated above. Thus
we have
\begin{eqnarray}
\hms{\p_i\p^{\alpha}\dt^jp}{0}{\Omega_2}^2\ls\hm{\p^{\alpha}\dt^j\p_i(\chi_2p)}{0}^2\ls\z
\end{eqnarray}
which, based on the Navier-Stokes equation (\ref{perturbed linear
form equation}), further implies
\begin{eqnarray}
\hms{\p_3^2\p^{\alpha}\dt^ju}{0}{\Omega_2}^2\ls\z
\end{eqnarray}
Combining all above, we have
\begin{eqnarray}
\hms{\bar D_2^{2(N+2)-1}u}{2}{\Omega_2}^2+\hms{\bar
D_2^{2(N+2)-1}\nabla p}{0}{\Omega_2}^2\ls\z
\end{eqnarray}

\ \\
Step 5: Estimates of free surface.\\
With above estimates in hand, we may derive the estimates for $\e$
by employing the boundary condition of (\ref{perturbed linear form
equation}), i.e.
\begin{eqnarray}
\e=p-3\p_3u_3-G^3_3
\end{eqnarray}
We may take horizontal derivative on both sides of above equation,
i.e. for $\alpha\in\mathbb{N}^2$ and $3\leq\abs{\alpha}\leq
2(N+2)-1$, we have
\begin{eqnarray}
\hms{\p^{\alpha}\e}{1/2}{\Sigma}^2&\ls&\hms{\p^{\alpha}p}{1/2}{\Sigma}^2+\hms{\p^{\alpha}\p_3u_3}{1/2}{\Sigma}^2+\hms{\p^{\alpha}G^3_3}{1/2}{\Sigma}^2\\
&\ls&\hm{\nabla\bar D_2^{2(N+2)-1}p}{0}^2+\hm{\bar
D_2^{2(N+2)-1}u}{2}^2+\hms{\p^{\alpha}G^3_3}{1/2}{\Sigma}^2\nonumber\\
&\ls&\z\nonumber
\end{eqnarray}
For the temporal derivative of $\e$, we need to resort to the
transport equation on the upper boundary, i.e.
\begin{eqnarray}
\dt\e=u_3+G^4
\end{eqnarray}
For $\alpha\in\mathbb{N}^2$ and $1\leq\abs{\alpha}\leq 2(N+2)-1$, by
trace theorem, Poincare theorem and the divergence equation
$\nabla\cdot u=G^2$, we have
\begin{eqnarray}
\hms{\p^{\alpha}\dt\e}{1/2}{\Sigma}^2&\ls&\hms{\p^{\alpha}u_3}{1/2}{\Sigma}^2+\hms{\p^{\alpha}G^4}{1/2}{\Sigma}^2\\
&\ls&\hm{\p^{\alpha}u_3}{1}^2+\z\nonumber\\
&\ls&\hm{\p^{\alpha}u_3}{0}^2+\hm{\p^{\alpha}Du_3}{0}^2+\hm{\p^{\alpha}\p_3u_3}{0}^2+\z\nonumber\\
&\ls&\hm{\p^{\alpha}\p_3u_3}{0}^2+\z\nonumber\\
&\ls&\hm{\p^{\alpha}Du}{0}^2+\hm{\p^{\alpha}G^2}{0}^2+\z\nonumber\\
&\ls&\z\nonumber
\end{eqnarray}
Similarly, for $j=2,\ldots,(N+2)+1$, we may apply $\dt^{j-1}$ to
transport equation and for $\alpha\in\mathbb{N}^2$ and
$0\leq\abs{\alpha}\leq 2(N+2)-2j+2$ estimate
\begin{eqnarray}
\hms{\p^{\alpha}\dt^j\e}{1/2}{\Sigma}^2&\ls&\hms{\p^{\alpha}\dt^{j-1}u_3}{1/2}{\Sigma}^2+\hms{\p^{\alpha}\dt^{j-1}G^4}{1/2}{\Sigma}^2\\
&\ls&\hm{\p^{\alpha}\dt^{j-1}u_3}{1}^2+\z\nonumber\\
&\ls&\z\nonumber
\end{eqnarray}

\ \\
Step 6: Synthesis.\\
To summarize all above, since $\Omega=\Omega_1\cup\Omega_2$ we have
the estimate
\begin{eqnarray}
\d_{N+2,2}\ls\bar\d_{N+2,2}+\y_{N+2,2}
\end{eqnarray}
By lemma \ref{lower nonlinear estimate}, it is obvious that there
exists a $\theta>0$ such that
\begin{eqnarray}
\y_{N+2,2}\ls\ce_{2N}^{\theta}\d_{N+2,2}
\end{eqnarray}
Therefore, our result naturally follows.
\end{proof}
\begin{proposition}\label{higher comparison dissipation estimate}
There exists a $\theta>0$ such that
\begin{eqnarray}
\d_{2N}\ls\bar\d_{2N}+\ce_{2N}^{\theta}\d_{2N}+\k\f_{2N}
\end{eqnarray}
\end{proposition}
\begin{proof}
It is easy to see that most part of the proof is identical to that
of proposition \ref{lower comparison dissipation estimate}, except
that we do not need the minimum count and interpolation now. Also we
will utilize lemma \ref{higher nonlinear estimate} instead to
estimate the nonlinear terms in perturbed linear form. Finally, a
direct estimate for Riesz potential term
\begin{eqnarray}
\hm{\i u}{1}^2\ls\hm{\dm\i u}{0}^2
\end{eqnarray}
can close the proof.
\end{proof}

\subsubsection{Energy Comparison Estimates}

\begin{proposition}\label{lower comparison energy estimate}
There exists a $\theta>0$ such that
\begin{eqnarray}
\ce_{N+2,2}\ls\bar\ce_{N+2,2}+\ce_{2N}^{\theta}\ce_{N+2,2}
\end{eqnarray}
\end{proposition}
\begin{proof}
We define
\begin{eqnarray}
\\
\w_{N+2,2}=\hm{\bar D^2\bar\nabla_0^{2(N+2)-4}G^1}{0}^2+\hm{\bar
D^2\bar\nabla_0^{2(N+2)-3}G^2}{0}^2+\hms{\bar
D^2\bar\nabla_0^{2(N+2)-4}G^3}{1/2}{\Sigma}^2\nonumber
\end{eqnarray}
and
\begin{eqnarray}
\z=\bar\ce_{N+2,2}+\w_{N+2,2}
\end{eqnarray}
Then we can directly apply elliptic estimate (\ref{appendix elliptic
strong estimate}) in perturbed linear form. For
$\alpha\in\mathbb{N}^2$ and $\abs{\alpha}=2$, we have
\begin{eqnarray}
\hm{\p^{\alpha}u}{2(N+2)-2}^2+\hm{\p^{\alpha}p}{2(N+2)-3}^2\ls\hm{\p^{\alpha}\dt
u}{2(N+2)-4}^2+\hm{\p^{\alpha}G^1}{2(N+2)-4}^2\\
+\hm{\p^{\alpha}G^2}{2(N+2)-3}^2+\hms{\p^{\alpha}\e}{2(N+2)-7/2}{\Sigma}^2+\hms{\p^{\alpha}G^3}{2(N+2)-7/2}{\Sigma}^2\nonumber
\end{eqnarray}
Naturally, we have
\begin{eqnarray}
\hm{\p^{\alpha}G^1}{2(N+2)-4}^2
+\hm{\p^{\alpha}G^2}{2(N+2)-3}^2+\hms{\p^{\alpha}G^3}{2(N+2)-7/2}{\Sigma}^2\ls\w_{N+2,2}\ls\z
\end{eqnarray}
and
\begin{eqnarray}
\hms{\p^{\alpha}\e}{2(N+2)-7/2}{\Sigma}^2\ls\bar\ce_{N+2,2}\ls\z
\end{eqnarray}
So the only remaining term is
\begin{eqnarray}
\hm{\p^{\alpha}\dt u}{2(N+2)-4}^2\ls\hm{\dt u}{2(N+2)-2}^2
\end{eqnarray}
which should be estimated with a finite induction in the
following.\\
For $j=1,\ldots,(N+2)-1$, we have the elliptic estimate
\begin{eqnarray}
\\
\hm{\dt^ju}{2(N+2)-2j}^2+\hm{\dt^jp}{2(N+2)-2j-1}^2\ls\hm{\dt^{j+1}
u}{2(N+2)-2j-2}^2+\hm{\dt^jG^1}{2(N+2)-2j-2}^2\nonumber\\
+\hm{\dt^jG^2}{2(N+2)-2j-1}^2+\hms{\dt^j\e}{2(N+2)-2j-3/2}{\Sigma}^2+\hms{\dt^jG^3}{2(N+2)-2j-3/2}{\Sigma}^2\nonumber
\end{eqnarray}
By the definition, it is easy to see
\begin{eqnarray}
\\
\hm{\dt^jG^1}{2(N+2)-2j-2}^2
+\hm{\dt^jG^2}{2(N+2)-2j-1}^2+\hms{\dt^jG^3}{2(N+2)-2j-3/2}{\Sigma}^2\ls\w_{N+2,2}\ls\z\nonumber
\end{eqnarray}
and
\begin{eqnarray}
\hms{\dt^j\e}{2(N+2)-2j-3/2}{\Sigma}^2\ls\bar\ce_{N+2,2}\ls\z
\end{eqnarray}
When $j=(N+2)-1$, we have
\begin{eqnarray}
\hm{\dt^{j+1}
u}{2(N+2)-2j-2}^2=\hm{\dt^{N+2}u}{0}^2\ls\bar\ce_{N+2,2}\ls\z
\end{eqnarray}
Hence, the elliptic estimate shows
\begin{eqnarray}
\hm{\dt^{N+1}u}{2}^2+\hm{\dt^{N+1}p}{1}^2\ls\z
\end{eqnarray}
Inductively, we can estimate all the temporal derivatives of
velocity and pressure for $j=1,\ldots,N$, which will finally close
the proof.\\
Therefore, we have shown that
\begin{eqnarray}
\ce_{N+2,2}\ls\bar\ce_{N+2,2}+\w_{N+2,2}
\end{eqnarray}
Based on lemma \ref{lower nonlinear estimate}, there exists a
$\theta>0$ such that
\begin{eqnarray}
\w_{N+2,2}\ls\ce_{2N}^{\theta}\ce_{N+2,2}
\end{eqnarray}
Hence, our result naturally follows.
\end{proof}
\begin{proposition}\label{higher comparison energy estimate}
There exists a $\theta>0$ such that
\begin{eqnarray}
\ce_{2N}\ls\bar\ce_{2N}+\ce_{2N}^{1+\theta}
\end{eqnarray}
\end{proposition}
\begin{proof}
It is easy to see that the proof is almost the same as that of
proposition \ref{lower comparison energy estimate} except that there
is no minimum count now and we should utilize lemma \ref{higher
nonlinear estimate} instead. So we omit the proof here.
\end{proof}

\subsection{A Priori Estimates}

First, we concentrate on the $2N$ level a priori estimate.
\begin{lemma}
There exists $C>0$ such that
\begin{eqnarray}
\sup_{0\leq r\leq
t}\f_{2N}(r)&\ls&\exp\bigg(C\int_0^t\sqrt{\k(r)}\ud{r}\bigg)\\
&&\bigg[\f_{2N}(0)+t\int_0^t(1+\ce_{2N}(r))\d_{2N}(r)\ud{r}+\bigg(\int_0^t\sqrt{\k(r)\f_{2N}(r)}\ud{r}\bigg)^2\bigg]\nonumber
\end{eqnarray}
\end{lemma}
\begin{proof}
This lemma is almost identical to lemma 9.1 in \cite{book9}, so we
omit the detailed proof here and only outline some differences here.
Based on trace theorem, we have that for $0\leq\abs{\beta}\leq 4N$
and $\beta\in\mathbb{N}^2$
\begin{eqnarray}
\hms{\p^{\beta}u}{1/2}{\Sigma}^2\ls\hm{\p^{\beta}u}{1}^2\ls\d_{2N}
\end{eqnarray}
Hence, although our new definition of dissipation is slightly
different from the original paper, we can still estimate and achieve
the same form.
\end{proof}
\begin{proposition}
There exists a universal constant $0<\delta<1$ such that if
$\g_N(T)\ls\delta$,  then
\begin{eqnarray}
\sup_{0\leq r\leq
t}\f_{2N}(r)\ls\f_{2N}(0)+t\int_0^t\d_{2N}(r)\ud{r}
\end{eqnarray}
and
\begin{eqnarray}
\int_0^t\k(r)\f_{2N}(r)\ud{r}\ls\delta^{(8+2\lambda)/(8+4\lambda)}\f_{2N}(0)+\delta^{(8+2\lambda)/(8+4\lambda)}\int_0^t\d_{2N}(r)\ud{r}
\end{eqnarray}
\begin{eqnarray}
\int_0^t\sqrt{\d_{2N}(r)\k(r)\f_{2N}(r)}\ud{r}\ls\f_{2N}(0)+\delta^{(8+2\lambda)/(16+6\lambda)}\int_0^t\d_{2N}(r)\ud{r}
\end{eqnarray}
for all $0\leq t\leq T$.
\end{proposition}
\begin{proof}
By lemma \ref{lower interpolation estimate improved 1}, we have that
$\k\ls\ce_{N+2,2}^{2/(2+r)}\ce_{2N}^{r/(2+r)}$ and $\ce_{2N}\leq 1$,
we may take $r=2\lambda/(4+\lambda)$ to achieve that
\begin{eqnarray}
\k\ls\ce_{N+2,2}^{(8+2\lambda)/(8+4\lambda)}
\end{eqnarray}
Then utilizing the decaying of $\ce_{N+2,2}$, the other part of the
proof is identical to proposition 9.2 and corollary 9.3 in
\cite{book9}, so we will omit it here.
\end{proof}
\begin{theorem}\label{higher apriori estimate}
There exists a universal constant $0<\delta<1$ such that if
$\g_N(T)\ls\delta$,  then
\begin{eqnarray}
\sup_{0\leq r\leq t}\ce_{2N}(r)+\int_0^t\d_{2N}+\sup_{0\leq r\leq
t}\frac{\f_{2N}(r)}{1+r}\ls\ce_{2N}(0)+\f_{2N}(0)
\end{eqnarray}
for all $0\leq t\leq T$.
\end{theorem}
\begin{proof}
This is identical to theorem 9.4 of \cite{book9}. Notice that in the
section of energy estimates and comparison estimates, we have
recovered all the necessary results for $2N$ level a priori estimate
here.
\end{proof}
Then we consider the a priori estimate for $N+2$ level.
\begin{lemma}
Let $F^2$ be defined with $\dt^{N+2}$. There exists a universal
constant $0<\delta<1$ such that if $\g_N(T)\ls\delta$,  then
\begin{eqnarray}
\frac{2}{3}\bar\ce_{N+2,2}(t)\leq\bar\ce_{N+2,2}-2\int_{\Omega}J(t)\dt^{N+1}p(t)F^2(t)\leq\frac{4}{3}\bar\ce_{N+2,2}(t)
\end{eqnarray}
for all $0\leq t\leq T$.
\end{lemma}
\begin{proof}
It is the same as lemma 9.6 of \cite{book9}.
\end{proof}
\begin{theorem}\label{lower apriori estimate}
There exists a universal constant $0<\delta<1$ such that if
$\g_N(T)\ls\delta$,  then
\begin{eqnarray}
\sup_{0\leq r\leq t}(1+r)^{2+\lambda}\ce_{N+2,2}(r)\ls\ce_{2N}(0)
\end{eqnarray}
for all $0\leq t\leq T$.
\end{theorem}
\begin{proof}
utilizing lemma \ref{lower interpolation estimate improved 2}, we
may employ the same argument as theorem 9.7 in \cite{book9} to show
this. Also notice that all the necessary lemmas have been recovered
in the section of energy estimates and comparison estimates.
\end{proof}
Next, we will give an interpolation argument to achieve the decaying
results for energy with different minimum count.
\begin{theorem}\label{lower apriori estimate improved}
There exists a universal constant $0<\delta<1$ such that if
$\g_N(T)\ls\delta$,  then
\begin{eqnarray}
\sup_{0\leq r\leq t}(1+r)^{1+\lambda}\ce_{N+2,1}(r)\ls\ce_{2N}(0)
\end{eqnarray}
for all $0\leq t\leq T$.
\end{theorem}
\begin{proof}
By lemma \ref{lower interpolation estimate improved 3}, we have that
\begin{eqnarray}
\ce_{N+2,1}\ls\ce_{N+2,2}^{(1+\lambda)/(2+\lambda)}\ce_{2N}^{1/(2+\lambda)}
\end{eqnarray}
Combining with theorem \ref{lower apriori estimate} and theorem
\ref{higher apriori estimate}, we can easily see that
\begin{eqnarray}
(1+r)^{1+\lambda}\ce_{N+2,1}(r)&\ls&\bigg((1+r)^{2+\lambda}\ce_{N+2,2}(r)\bigg)^{(1+\lambda)/(2+\lambda)}\bigg(\ce_{2N}(r)\bigg)^{1/(2+\lambda)}\\
&\ls&\bigg(\ce_{2N}(0)\bigg)^{(1+\lambda)/(2+\lambda)}\bigg(\ce_{2N}(0)\bigg)^{1/(2+\lambda)}=\ce_{2N}(0)\nonumber
\end{eqnarray}
Hence, our result easily follows.
\end{proof}
Finally, we can sum up all above to achieve the a priori estimate.
\begin{theorem}\label{apriori estimate}
There exists a universal constant $0<\delta<1$ such that if
$\g_N(T)\ls\delta$,  then
\begin{eqnarray}
\g_N(t)\ls\ce_{2N}(0)+\f_{2N}(0)
\end{eqnarray}
for all $0\leq t\leq T$.
\end{theorem}
\begin{proof}
To sum up the conclusion of theorem \ref{higher apriori estimate} ,
\ref{lower apriori estimate} and \ref{lower apriori estimate
improved} easily implies this.
\end{proof}

\subsection{Global Wellposedness}

Since we introduce several terms related to Riesz potential in the a
priori estimates, we have to revise our local wellposedness result
here to adopt this, which is perfectly completed in theorem 10.7 of
\cite{book9}.
\begin{theorem}\label{local wellposedness}
Suppose that initial data are given satisfying the $2N^{th}$
compatible condition and
$\hm{u_0}{4N}^2+\hms{\ee}{4N+1/2}{\Sigma}^2+\hm{\i
u_0}{0}^2+\hms{\i\e_0}{0}{\Sigma}^2<\infty$. Let $\epsilon>0$. There
exists a $\delta_0=\delta_0(\epsilon)$ and a
\begin{eqnarray}
T_0=C(\epsilon)\min\bigg\{1,\frac{1}{\hms{\ee}{4N+1/2}{\Sigma}^2}\bigg\}>0
\end{eqnarray}
where $C(\epsilon)>0$ is a constant depending on $\epsilon$, such
that if $0<T<T_0$ and
$\hm{u_0}{4N}^2+\hms{\ee}{4N}{\Sigma}^2\leq\delta_0$, then there
exists a unique solution $(u,p,\e)$ to system (\ref{transform}) on
the interval $[0,T]$ that achieves the initial data. The solution
obeys the estimates
\begin{eqnarray}
\sup_{0\leq t\leq T}\ce_{2N}(t)+\sup_{0\leq t\leq T}\hm{\i
p(t)}{0}^2+\int_0^T\d_{2N}(t)\ud{t}\\
+\int_0^T\bigg(\nm{\dt^{2N+1}u(t)}_{(_0H^1)^{\ast}}^2+\hm{\dt^{2N}p(t)}{0}^2\bigg)\ud{t}\ls\epsilon+\hm{\i
u_0}{0}^2+\hms{\i\ee}{0}{\Sigma}^2\nonumber
\end{eqnarray}
and
\begin{eqnarray}
\sup_{0\leq t\leq T}\bigg(\ce_{2N}(t)-\hm{\i
u}{0}^2-\hms{\i\e}{0}{\Sigma}^2\bigg)\leq\epsilon
\end{eqnarray}
\begin{eqnarray}
\sup_{0\leq t\leq T}\f_{2N}(t)\leq C\f_{2N}(0)+\epsilon
\end{eqnarray}
\end{theorem}
Finally, we come to the conclusion for global wellposedness. Since
we have recovered the a priori estimate and do not improve the
argument for this part, the reader may directly refer to theorem
11.1 and 11.2 in \cite{book9}, which completes the whole proof.
\begin{theorem}\label{global wellposedness}
Suppose that the initial data $(u_0,\ee)$ satisfy the $2N^{th}$
compatible condition. There exists a $\kappa>0$ such that if
$\hm{u_0}{4N}^2+\hms{\ee}{4N+1/2}{\Sigma}^2\leq\kappa$, then there
exists a unique solution $(u,p,\e)$ on the interval $[0,\infty)$
that achieves the initial data. The solution obeys the estimate
\begin{eqnarray}
\g_N(\infty)\leq
C(\hm{u_0}{4N}^2+\hms{\ee}{4N+1/2}{\Sigma}^2)<C\kappa
\end{eqnarray}
\end{theorem}

\appendix

\makeatletter
\renewcommand \theequation {%
A.%
\ifnum\c@subsection>\z@\@arabic\c@subsection.%
\fi\@arabic\c@equation} \@addtoreset{equation}{section}
\@addtoreset{equation}{subsection} \makeatother

\section{Analytic Tools}

\subsection{Products in Sobolev Space}

We will need some estimates of the products of functions in Sobolev
spaces. Since these results have been proved in lemma A.1 and lemma
A.2 of \cite{book1}, we will present the statement of the lemmas
here without proof.
\begin{lemma}\label{Appendix product}
Let $U$ denote either $\Sigma$ or $\Omega$.
\begin{enumerate}
\item
Let $0\leq r\leq s_1\leq s_2$ be such that $s_1>n/2$. Let $f\in
H^{s_1}(U)$, $g\in H^{s_2}(U)$. Then $fg\in H^r(U)$ and
\begin{equation}
\hm{fg}{r}\ls \hm{f}{s_1}\hm{g}{s_2}
\end{equation}
\item
Let $0\leq r\leq s_1\leq s_2$ be such that $s_2>r+n/2$. Let $f\in
H^{s_1}(U)$, $g\in H^{s_2}(U)$. Then $fg\in H^r(U)$ and
\begin{equation}
\hm{fg}{r}\ls \hm{f}{s_1}\hm{g}{s_2}
\end{equation}
\item
Let $0\leq r\leq s_1\leq s_2$ be such that $s_2>r+n/2$. Let
$f\hs{-r}$, $g\hs{s_2}$. Then $fg\hs{-s_1}$ and
\begin{equation}
\hm{fg}{-s_1}\ls \hm{f}{-r}\hm{g}{s_2}
\end{equation}
\end{enumerate}
\end{lemma}
\begin{lemma}
Suppose that $f\in C^1(\Sigma)$ and $g\hs{1/2}$. Then $fg\hs{1/2}$
and
\begin{equation}
\hm{fg}{1/2}\ls\nm{f}_{C^1}\hm{g}{1/2}
\end{equation}
\end{lemma}

\subsection{Poincare-Type Inequality}

We need several Poincare-type inequality in $\Omega$. Since all
these lemmas have be proved in lemma A.10-A.13 in \cite{book9}, we
will only give the statement here without proof.
\begin{lemma}\label{appendix poincare 1}
It holds that
\begin{eqnarray}
\nm{f}_{L^2(\Omega)}^2\ls\nm{f}_{L^2(\Sigma)}^2+\nm{\p_3f}_{L^2(\Omega)}^2
\end{eqnarray}
for all $f\in H^1(\Omega)$. Also, if $f\in W^{1,\infty}(\Omega)$,
then
\begin{eqnarray}
\nm{f}_{L^{\infty}(\Omega)}^2\ls\nm{f}_{L^{\infty}(\Sigma)}^2+\nm{\p_3f}_{L^{\infty}(\Omega)}^2
\end{eqnarray}
\end{lemma}
\begin{lemma}\label{appendix poincare 2}
It holds that $\hms{f}{0}{\Sigma}\ls\hm{\p_3f}{0}$ for $f\in
H^1(\Omega)$ such that $f=0$ on $\Sigma_b$. It also holds that
$\nm{f}_{L^{\infty}(\Sigma)}\ls\nm{\p_3f}_{L^{\infty}(\Omega)}$ for
$f\in W^{1,\infty}(\Omega)$ such that $f=0$ on $\Sigma_b$.
\end{lemma}
\begin{lemma}\label{appendix poincare 3}
It holds that $\hm{u}{1}\ls\hm{\dm u}{0}$ for all $u\in
H^1(\Omega;R^3)$ such that $f=0$ on $\Sigma_b$.
\end{lemma}
\begin{lemma}\label{appendix poincare 4}
It holds that $\hm{f}{1}\ls\hm{\nabla f}{0}$ for all $f\in
H^1(\Omega)$ such that $f=0$ on $\Sigma_b$. Also,
$\nm{f}_{W^{1,\infty}(\Omega)}\ls\nm{\nabla f}_{L^{\infty}(\Omega)}$
for all $f\in W^{1,\infty}(\Omega)$ such that $f=0$ on $\Sigma_b$.
\end{lemma}

\subsection{Poisson Integral}

For a function $f$ defined on $\Sigma=R^2$, the Poisson integral
$\pp f$ in $R^2\times(-\infty,0)$ is defined by
\begin{eqnarray}\label{appendix poisson integral}
\pp f(x',x_3)=\int_{R^2}\hat{f}(\xi)e^{2\pi\abs{\xi}x_3}e^{2\pi
ix'\cdot\xi}\ud{\xi}
\end{eqnarray}
where $\hat f(\xi)$ is the Fourier transform of $f(x')$ on $R^2$.
Then within the slab $R^2\times(-b,0)$, we have the following
estimate based on lemma A.5 in \cite{book9}.
\begin{lemma}
The Poisson integral satisfies that for $q\in\mathbb{N}$,
\begin{eqnarray}\label{appendix poisson integral estimate}
\hm{\nabla^q\pp
f}{0}^2\ls\nm{f}_{\dot{H}^{q-1/2}(\Sigma)}^2\\
\hm{\nabla^q\pp f}{0}^2\ls\nm{f}_{\dot{H}^{q}(\Sigma)}^2
\end{eqnarray}
where $\dot{H}^s$ denotes the usual homogeneous Sobolev space.
\end{lemma}

\subsection{Riesz Potential}

For a function $f$ defined in $\Omega$, we define the Riesz
potential
\begin{eqnarray}\label{appendix riesz potential 1}
\i
f(x',x_3)=\int_{-b}^0\int_{R^2}\hat{f}(\xi,x_3)\abs{\xi}^{-\lambda}e^{2\pi
ix'\cdot\xi}\ud{\xi}\ud{x_3}
\end{eqnarray}
Similarly, for $f$ defined on $\Sigma$, we set
\begin{eqnarray}\label{appendix riesz potential 2}
\i f(x')=\int_{R^2}\hat{f}(\xi)\abs{\xi}^{-\lambda}e^{2\pi
ix'\cdot\xi}\ud{\xi}
\end{eqnarray}
We have the following lemmas to describe the product of Riesz
potential and its interaction with the horizontal derivatives as
lemma A.3 and A.4 in \cite{book9}.
\begin{lemma}\label{appendix riesz potential estimate 1}
Let $\lambda\in(0,1)$. If $f\in H^0(\Omega)$ and $g,Dg\in
H^1(\Omega)$, then
\begin{eqnarray}
\hm{\i(fg)}{0}\ls \hm{f}{0}\hm{g}{1}^{\lambda}\hm{Dg}{1}^{1-\lambda}
\end{eqnarray}
If $f\in H^0(\Sigma)$ and $g,Dg\in H^1(\Sigma)$, then
\begin{eqnarray}
\hms{\i(fg)}{0}{\Sigma}\ls
\hms{f}{0}{\Sigma}\hms{g}{1}{\Sigma}^{\lambda}\hms{Dg}{1}{\Sigma}^{1-\lambda}
\end{eqnarray}
\end{lemma}
\begin{lemma}\label{appendix riesz potential estimate 2}
Let $\lambda\in(0,1)$. If $f\in H^k(\Omega)$ for $k\geq1$ an
integer, then
\begin{eqnarray}
\hm{\i D^kf}{0}\ls\hm{D^{k-1}f}{0}^{\lambda}\hm{D^kf}{0}^{1-\lambda}
\end{eqnarray}
\end{lemma}

\subsection{Interpolation Estimates}

Here we record several interpolation estimate utilized in our proof.
Since they have been proved in lemma A.6-A.8 and lemma 3.18 in
\cite{book9}, we will omit the proof now.
\begin{lemma}\label{appendix interpolation 1}
For $s,q>0$ and $0\leq r\leq s$, we have the estimate
\begin{eqnarray}
\hms{f}{s}{\Sigma}\ls\hms{f}{s-r}{\Sigma}^{q/(r+q)}\hms{f}{s+q}{\Sigma}^{r/(r+q)}
\end{eqnarray}
whenever the right hand side is finite.
\end{lemma}
\begin{lemma}\label{appendix interpolation 2}
Let $\pp f$ be the Poisson integral of $f$, defined on $\Sigma$. Let
$\lambda\geq0$, $q,s\in\mathbb{N}$, and $r\geq0$. Then the following
estimates hold.
\begin{enumerate}
\item
Let
\begin{eqnarray}
\theta=\frac{s}{q+s+\lambda} \quad
1-\theta=\frac{q+\lambda}{q+s+\lambda}
\end{eqnarray}
Then
\begin{eqnarray}
\hm{\nabla^q\pp f}{0}^2\ls\bigg(\hm{\i
f}{0}^2\bigg)^{\theta}\bigg(\hm{D^{q+s}f}{0}^2\bigg)^{1-\theta}
\end{eqnarray}
\item
Let $r+s>1$,
\begin{eqnarray}
\theta=\frac{r+s-1}{q+s+r+\lambda} \quad
1-\theta=\frac{q+\lambda+1}{q+s+r+\lambda}
\end{eqnarray}
Then
\begin{eqnarray}
\lnm{\nabla^q\pp f}{\infty}^2\ls\bigg(\hm{\i
f}{0}^2\bigg)^{\theta}\bigg(\hm{D^{q+s}f}{r}^2\bigg)^{1-\theta}
\end{eqnarray}
\item
Let $s>1$. Then
\begin{eqnarray}
\lnm{\nabla^q\pp f}{\infty}^2\ls\hm{D^qf}{s}^2
\end{eqnarray}
\end{enumerate}
\end{lemma}
\begin{lemma}\label{appendix interpolation 3}
Let $f$ be defined on $\Sigma$. Let $\lambda\geq0$. Then we have the
following estimates.
\begin{enumerate}
\item
Let $q,s\in(0,\infty)$ and
\begin{eqnarray}
\theta=\frac{s}{q+s+\lambda} \quad
1-\theta=\frac{q+\lambda}{q+s+\lambda}
\end{eqnarray}
Then
\begin{eqnarray}
\hm{D^q f}{0}^2\ls\bigg(\hm{\i
f}{0}^2\bigg)^{\theta}\bigg(\hm{D^{q+s}f}{0}^2\bigg)^{1-\theta}
\end{eqnarray}
\item
Let $q,s\in\mathbb{N}$, $r\geq0$, $r+s>1$,
\begin{eqnarray}
\theta=\frac{r+s-1}{q+s+r+\lambda} \quad
1-\theta=\frac{q+\lambda+1}{q+s+r+\lambda}
\end{eqnarray}
Then
\begin{eqnarray}
\lnm{D^q f}{\infty}^2\ls\bigg(\hm{\i
f}{0}^2\bigg)^{\theta}\bigg(\hm{D^{q+s}f}{r}^2\bigg)^{1-\theta}
\end{eqnarray}
\end{enumerate}
\end{lemma}
\begin{lemma}\label{appendix interpolation 4}
Let $f$ be defined in $\Omega$. Let $\lambda\geq0$,
$q,s\in\mathbb{N}$, and $r\geq0$. Then we have the following
estimates.
\begin{enumerate}
\item
Let
\begin{eqnarray}
\theta=\frac{s}{q+s+\lambda} \quad
1-\theta=\frac{q+\lambda}{q+s+\lambda}
\end{eqnarray}
Then
\begin{eqnarray}
\hm{D^q f}{0}^2\ls\bigg(\hm{\i
f}{0}^2\bigg)^{\theta}\bigg(\hm{D^{q+s}f}{0}^2\bigg)^{1-\theta}
\end{eqnarray}
\item
Let $r+s>1$
\begin{eqnarray}
\theta=\frac{r+s-1}{q+s+r+\lambda} \quad
1-\theta=\frac{q+\lambda+1}{q+s+r+\lambda}
\end{eqnarray}
Then
\begin{eqnarray}
\lnm{D^q f}{\infty}^2\ls\bigg(\hm{\i
f}{1}^2\bigg)^{\theta}\bigg(\hm{D^{q+s}f}{r+1}^2\bigg)^{1-\theta}
\end{eqnarray}
and
\begin{eqnarray}
\lnms{D^q f}{\infty}{\Sigma}^2\ls\bigg(\hm{\i
f}{1}^2\bigg)^{\theta}\bigg(\hm{D^{q+s}f}{r+1}^2\bigg)^{1-\theta}
\end{eqnarray}
\end{enumerate}
\end{lemma}

\subsection{Continuity and Temporal Derivative}

In the following, we give two important lemmas to connect $L^2H^k$
norm and $L^{\infty}H^k$ norm.
\begin{lemma}\label{Appendix connection 1}
suppose that $u\in L^2([0,T];H^{s_1}(\Omega))$ and $\partial_tu\in
L^2([0,T];H^{s_2}(\Omega))$ for $s_1\geq s_2\geq0$ and
$s=(s_1+s_2)/2$. Then $u\in C^0([0,T];H^s(\Omega))$ and satisfies
the estimate
\begin{equation}
\|u\|_{L^{\infty}H^s}^2\leq
\|u(0)\|_{H^s}^2+\|u\|_{L^2H^{s_1}}^2+\|\partial_tu\|_{L^2H^{s_2}}^2
\end{equation}
where the $L^2H^k$ norm and $L^{\infty}H^k$ norm are evaluated in
$[0,T]$.
\end{lemma}
\begin{proof}
Considering the extension theorem in Sobolev space, we only need to
prove this result in the $R^n$ case. The periodic case can be
derived in a similar fashion. Using Fourier transform,
\begin{eqnarray*}
\partial_t\|u(t)\|_{H^s}^2&=&2\mathfrak{R}\bigg(\int_{R^n}\langle\xi\rangle^{2s}\hat{u}(\xi,t)\overline{\partial_t\hat{u}(\xi,t)}d\xi\bigg)
\leq
2\int_{R^n}\langle\xi\rangle^{2s}|\hat{u}(\xi,t)||\partial_t\hat{u}(\xi,t)|d\xi\\
&=&2\int_{R^n}\langle\xi\rangle^{s_1}|\hat{u}(\xi,t)|\langle\xi\rangle^{s_2}|\partial_t\hat{u}(\xi,t)|d\xi\leq
\int_{R^n}\langle\xi\rangle^{2s_1}|\hat{u}(\xi,t)|^2d\xi+\int_{R^n}\langle\xi\rangle^{2s_2}|\partial_t\hat{u}(\xi,t)|^2d\xi\\
&=&\|u(t)\|^2_{H^{s_1}}+\|\partial_tu(t)\|^2_{H^{s_2}}
\end{eqnarray*}
So integrate with respect to time on $[0,t]$
\begin{displaymath}
\|u(t)\|_{H^s}^2\leq
\|u(0)\|_{H^s}^2+\|u\|_{L^2H^{s_1}}^2+\|\partial_tu\|_{L^2H^{s_2}}^2
\end{displaymath}
\end{proof}
The following lemma shows the estimate in another direction.
\begin{lemma}\label{Appendix connection 2}
For any $u\is{k}$ within $[0,T]$, we must have $u\ts{k}$ and
satisfies the estimate
\begin{equation}
\|u\|_{L^2H^k}^2\leq T\|u\|^2_{L^{\infty}H^k}
\end{equation}
\end{lemma}
\begin{proof}
The result is simply based on the definition of these two norms.
\end{proof}

\subsection{Extension Theorem}

The following are two extension theorems which will be used to
construct start point of iteration from initial data in proving
wellposedness of Naiver-Stokes-transport system. Since it is
identical as lemma A.5 and A.6 in \cite{book1}, we omit the proof
here.
\begin{lemma}\label{Appendix extension 1}
Suppose that $\dt^ju(0)\s{2N-2j}(\Omega)$ for $j=0,\ldots,N$, then
there exists a extension $u$ achieving the initial data, such that
\begin{displaymath}
\dt^ju\in L^2([0,\infty);H^{2N-2j+1}(\Omega))\cap
L^{\infty}([0,\infty); H^{2N-2j}(\Omega))
\end{displaymath}
for $j=0,\ldots,N$. Moreover,
\begin{eqnarray}
\sum_{j=0}^{N}\tnm{\dt^ju}{2N-2j+1}^2+\inm{\dt^ju}{2N-2j}^2\ls
\sum_{j=0}^{N}\hm{\dt^ju(0)}{2N-2j}^2
\end{eqnarray}
\end{lemma}
\begin{lemma}\label{Appendix extension 2}
Suppose that $\dt^jp(0)\s{2N-2j-1}(\Omega)$ for $j=0,\ldots,N-1$,
then there exists a extension $p$ achieving the initial data, such
that
\begin{displaymath}
\dt^jp\in L^2([0,\infty);H^{2N-2j}(\Omega))\cap
L^{\infty}([0,\infty); H^{2N-2j-1}(\Omega))
\end{displaymath}
for $j=0,\ldots,N-1$. Moreover,
\begin{eqnarray}
\sum_{j=0}^{N-1}\tnm{\dt^jp}{2N-2j}^2+\inm{\dt^jp}{2N-2j-1}^2\ls
\sum_{j=0}^{N-1}\hm{\dt^jp(0)}{2N-2j-1}^2
\end{eqnarray}
\end{lemma}

\makeatletter
\renewcommand \theequation {%
B.%
\ifnum\c@subsection>\z@\@arabic\c@subsection.%
\fi\@arabic\c@equation} \@addtoreset{equation}{section}
\@addtoreset{equation}{subsection} \makeatother

\section{Estimates for Fundamental Equations}

\subsection{Transport Estimates}

Let $\Sigma$ be either infinite or periodic. Consider the equation
\begin{equation}\label{Appendix transport equation}
\left\{
\begin{array}{ll}
\dt\e+u\cdot D\e=g &in\quad\Sigma\times(0,T)\\
\e(t=0)=\ee
\end{array}
\right.
\end{equation}
We have the following estimate of the regularity solution to this
equation, which is a particular case of a more general result proved
in proposition 2.1 of \cite{book3}. Note that the result in
\cite{book3} is stated for $\Sigma=R^2$, but the same result holds
in the periodic case as described in \cite{book7}.
\begin{lemma}\label{Appendix transport estimate}
Let $\e$ be a solution to equation (\ref{Appendix transport
equation}). Then there exists a universal constant $C>0$ such that
for any $s\geq3$
\begin{eqnarray}
\inm{\e}{s}\leq
\exp\bigg(C\int_0^t\hm{Du(r)}{3/2}\ud{r}\bigg)\bigg(\hm{\ee}{s}+\int_0^t\hm{g(r)}{s}\ud{r}\bigg)
\end{eqnarray}
\end{lemma}
\begin{proof}
Use $p=p_2=2$, $N=2$ and $\Sigma=s$ in proposition 2.1 of
\cite{book7}, along with the embedding $H^{3/2}\hookrightarrow
B_{2,\infty}^1\cap L^{\infty}$.
\end{proof}
\begin{lemma}\label{appendix transport estimate}
Let $\e$ be a solution to (\ref{Appendix transport equation}). Then
there is a universal constant $C>0$ such that for any $0\leq s<2$
\begin{eqnarray}
\\
\sup_{0\leq r\leq t}\hms{\e(r)}{s}{\Sigma}\leq
\exp\bigg(C\int_0^t\hms{Du(r)}{3/2}{\Sigma}\ud{r}\bigg)\bigg(\hms{\ee}{s}{\Sigma}+\int_0^t\hms{g(r)}{s}{\Sigma}\ud{r}\bigg)\nonumber
\end{eqnarray}
\end{lemma}
\begin{proof}
The same as lemma A.9 in \cite{book9}.
\end{proof}

\subsection{Elliptic Estimates}

We need two different type of elliptic estimates in our proof either
in weak sense or strong sense.
\begin{lemma}(weak solution)\label{appendix elliptic estimate 1}
We call $(u,p)\in H^1_0(\Omega)\times H^0(\Omega)$ a weak solution
to the elliptic equation
\begin{equation}\label{appendix elliptic equation 1}
\left\{
\begin{array}{ll}
-\Delta u+\nabla p=\phi\s{-1}(\Omega)&in\quad\Omega\\
\nabla\cdot u=\psi\s{0}(\Omega)&in\quad\Omega\\
u=0&on\quad\Sigma\\
u=0&on\quad\Sigma_b
\end{array}
\right.
\end{equation}
if $\nabla\cdot u=\psi$ in $\Omega$ and $(u,p)$ satisfy the weak
formulation that for arbitrary $f\in H^1_0(\Omega)$, it holds that
\begin{eqnarray}\label{appendix temp 2}
\br{\dm u:\dm f}_{H^0}+\br{p, \nabla\cdot
f}_{H^0}=\br{\phi,f}_{H^{-1}}
\end{eqnarray}
where $\br{\cdot,\cdot}_{H^0}$ represents the inner product in
$H^0(\Omega)$ and $\br{\cdot, \cdot}_{H^{-1}}$ denotes the dual
pairing between $H^1_0(\Omega)$ and $H^{-1}(\Omega)$. Then we have
the estimate
\begin{eqnarray}\label{appendix elliptic weak estimate}
\hm{u}{1}^2+\hm{p}{0}^2\ls \hm{\phi}{-1}^2+\hm{\psi}{0}^2
\end{eqnarray}
\end{lemma}
\begin{proof}
An obvious modification for the proof of lemma 3.3 in \cite{book2}
implies that for any $\psi\in H^0(\Omega)$, there exists a $v\in
H^1_0(\Omega)$ such that $\nabla\cdot v=\psi$ in $\Omega$ and
satisfies the estimate $\hm{v}{1}^2\ls\hm{\psi}{0}^2$. Then we may
change to the unknown $w=u-v$, which satisfies the system
\begin{equation}
\left\{
\begin{array}{ll}
-\Delta w+\nabla p=\phi+\Delta v&in\quad\Omega\\
\nabla\cdot w=0&in\quad\Omega\\
w=0&on\quad\Sigma\\
w=0&on\quad\Sigma_b
\end{array}
\right.
\end{equation}
Then we have the new weak formulation that for arbitrary $f\in
H^1_0(\Omega)$, it is valid that
\begin{eqnarray}\label{appendix temp 1}
\br{\dm w:\dm f}_{H^0}+\br{p, \nabla\cdot
f}_{H^0}=\br{\phi,f}_{H^{-1}}-\br{\dm v:\dm f}_{H^0}
\end{eqnarray}
We may further restrict that $f$ satisfies $\nabla\cdot f=0$ as the
pressureless weak formulation that
\begin{eqnarray}
\br{\dm w:\dm f}_{H^0}=\br{\phi,f}_{H^{-1}}-\br{\dm v:\dm f}_{H^0}
\end{eqnarray}
Riesz theorem implies that there exists a pressureless weak solution
$w\in H^1_0(\Omega)$ satisfying $\nabla\cdot w=0$ and the estimate
\begin{eqnarray}
\hm{w}{1}^2\ls
\hm{\phi}{-1}^2+\hm{v}{1}^2\ls\hm{\phi}{-1}^2+\hm{\psi}{0}^2
\end{eqnarray}
Then it is well-known that pressure can be taken as the Lagrangian
multiplier for the Navier-Stokes system. Similar to proposition 2.9
in \cite{book1}, we have that there exists a $p\in H^1(\Omega)$ such
that the weak formulation (\ref{appendix temp 1}) holds and it
satisfies
\begin{eqnarray}
\hm{p}{0}^2\ls\hm{\phi}{-1}^2+\hm{\psi}{0}^2+\hm{w}{1}^2\ls\hm{\phi}{-1}^2+\hm{\psi}{0}^2
\end{eqnarray}
Therefore, $(u,p)$ satisfy the weak formulation (\ref{appendix temp
2}) and our result easily follows.
\end{proof}

\begin{lemma}(strong solution)\label{appendix elliptic estimate 2}
Suppose $(u,p)$ solve the equation
\begin{equation}\label{appendix elliptic equation 2}
\left\{
\begin{array}{ll}
-\Delta u+\nabla p=\phi\s{r-2}(\Omega)&in\quad\Omega\\
\nabla\cdot u=\psi\s{r-1}(\Omega)&in\quad\Omega\\
(pI-\dm u)e_3=\varphi\s{r-3/2}(\Sigma)&on\quad\Sigma\\
u=0&on\quad\Sigma_b
\end{array}
\right.
\end{equation}
in the strong sense. Then for $r\geq2$,
\begin{eqnarray}\label{appendix elliptic strong estimate}
\hm{u}{r}^2+\hm{p}{r-1}^2\ls
\hm{\phi}{r-2}^2+\hm{\psi}{r-1}^2+\hm{\varphi}{r-3/2}^2
\end{eqnarray}
whenever the right hand side is finite.
\end{lemma}
\begin{proof}
The same as lemma A.15 in \cite{book1}.
\end{proof}


\begin{thebibliography}{1}


\bibitem{book1} Y. Guo, I. Tice. Local Well-Posedness of the Viscous
Surface Wave Problem without Surface Tension. (2012) To appear in
Adv. Partial Differ. Equ.

\bibitem{book9}Y. Guo, I. Tice. Decay of viscous surface waves without surface tension in horizontally infinite domains.
(2012) To appear in Anal. PDE.

\bibitem{book10}Y. Guo, I. Tice. Almost exponential decay of periodic viscous surface waves without surface tension.
(2012) To appear in Arch. Ration. Mech. Anal.

\bibitem{book2} J. Beale. Initial Value Problem for the Navier-Stokes
Equations with a Free Surface. Comm. Pure Appl. Math. 34 (1981), no.
3, 359-392.

\bibitem{book3} R. Danchin. Estimates in Besov spaces for transport
and transport-diffusion equations with almost Lipschitz coeffcients.
Rev. Mat. Iberoamericana 21 (2005), no. 3, 863-888.

\bibitem{book4} S. Agmon, A. Douglis, L. Nirenberg. Estimates near
the boundary for solutions of elliptic partial differential
equations satisfying general boundary conditions. I. Comm. Pure
Appl. Math. 12 (1959) 623-727.

\bibitem{book5}S. Agmon, A. Douglis, L. Nirenberg. Estimates near
the boundary for solutions of elliptic partial differential
equations satisfying general boundary conditions. II. Comm. Pure
Appl. Math. 17 (1964) 35-92.

\bibitem{book6}R. Adams. Sobolev spaces. Pure and Applied Mathematics,
Vol. 65. Academic Press, New York-London, 1975.

\bibitem{book7}R. Danchin. Fourier analysis methods for PDEs. Preprint (2005). Lab. d'Analyse et de Math. Appliques,
UMR 8050,
http://perso-math.univ-mlv.fr/users/danchin.raphael/recherche.html.

\bibitem{book8}L. Evans. Partial differential equations. Second edition. Graduate Studies in Mathematics, 19. American
Mathematical Society, Providence, RI, 2010.

\bibitem{book11}V. A. Solonnikov. Solvability of a problem on the motion of a viscous incompressible
uid that is bounded by a free surface. Math. USSR-Izv. 11 (1977),
no. 6, 1323-1358 (1978).

\bibitem{book12}S. Wu. Well-posedness in Sobolev spaces of the full water wave problem in 2-D. Invent. Math. 130 (1997),
no. 1, 39-72.

\bibitem{book13}S. Wu. Well-posedness in Sobolev spaces of the full water wave problem in 3-D. J. Amer. Math. Soc. 12
(1999), no. 2, 445-495.

\bibitem{book14}D. Lannes. Well-posedness of the water-waves equations. J. Amer. Math. Soc. 18 (2005), no. 3, 605-654.

\bibitem{book15}P. Zhang, Z. Zhang. On the free boundary problem of three-dimensional incompressible Euler equations.
Comm. Pure Appl. Math. 61 (2007), no. 7, 877-940.

\bibitem{book16} D. Christodoulou, H. Lindblad. On the motion of the free surface of a liquid. Comm. Pure Appl. Math. 53
(2000), no. 12, 1536-1602.

\bibitem{book17}H. Lindblad. Well-posedness for the motion of an incompressible liquid with free surface boundary. Ann. of
Math. 162 (2005), no. 1, 109-194.

\bibitem{book18}D. Coutand, S. Shkoller. Well-posedness of the free-surface incompressible Euler equations with or without
surface tension. J. Amer. Math. Soc. 20 (2007), no. 3, 829-930.

\bibitem{book19}J. Shatah, C. Zeng. Geometry and a priori estimates for free boundary problems of the Euler's equation.
Comm. Pure Appl. Math. 51 (2008), no. 5, 698-744.

\bibitem{book20}H. Abels. The initial-value problem for the Navier-Stokes equations with a free surface in Lq
-Sobolev spaces. Adv. Differential Equations 10 (2005), no. 1,
45-64.

\bibitem{book21}H. Bae. Solvability of the free boundary value problem of the Navier-Stokes equations. Discrete Contin. Dyn.
Syst. 29 (2011), no. 3, 769-801.

\bibitem{book22}S. Wu. Global well-posedness of the 3-D full water wave problem. Invent.
Math. 184 (2011), no. 1, 125-220

\bibitem{book23}P. Germain, N. Masmoudi, J. Shatah. Global solutions for the gravity water waves equation in dimension 3.
Ann. of Math. 175 (2012), no. 2, 691-754

\bibitem{book24}Y. Hataya. Decaying solution of a Navier-Stokes
ow without surface tension. J. Math. Kyoto Univ. 49 (2009), no. 4,
691-717.
\end{thebibliography}
\end{document}